\documentclass[11pt, eqno]{article}
\usepackage{bbm}
\usepackage{mathrsfs}
\usepackage{amsfonts}
\usepackage{amssymb}
\usepackage{graphicx}
\usepackage[all]{xy}

\usepackage{amsthm}
\usepackage{amsmath}
\usepackage{amsmath,amssymb,latexsym,color}
\usepackage[mathscr]{eucal}
\usepackage{CJK}
\usepackage{cases}
\usepackage{graphics}

\usepackage{psfrag}
\usepackage{subfigure}
\usepackage{color}
\usepackage{amssymb,latexsym}
\usepackage{amsmath,latexsym}
\usepackage{amscd}

\newcommand{\R}{{\mathbb R}}

\newtheorem{theorem}{Theorem}[section]
\newtheorem{cor}[theorem]{Corollary}
\newtheorem{corollary}[theorem]{Corollary}

\newtheorem{thm}[theorem]{Theorem}
\newtheorem{definition}[theorem]{Definition}

\newtheorem{remark}[theorem]{Remark}

\newtheorem{lemma}[theorem]{Lemma}

\newtheorem{proposition}[theorem]{Proposition}
\newtheorem{claim}[theorem]{Claim}

\begin{document}
\title{Coisotropic Ekeland-Hofer capacities}
\date{November 8, 2020\\
Revised June 4, 2021\\
Second revised August 25, 2022}

\author{Rongrong Jin and Guangcun Lu
\thanks{Corresponding author
\endgraf \hspace{2mm} Partially supported
by the NNSF  11271044 of China and the Fundamental Research Funds for Central Universities, Civil Aviation University of China, 3122021074.
\endgraf\hspace{2mm} 2010 {\it Mathematics Subject Classification.}
 53D35, 53C23 (primary), 70H05, 37J05, 57R17 (secondary).}}
 \maketitle \vspace{-0.3in}

\vspace{0.1in}
\abstract{
 For subsets in the standard symplectic space
 $(\mathbb{R}^{2n},\omega_0)$ whose closures are intersecting with coisotropic
 subspace $\mathbb{R}^{n,k}$ we construct relative versions of
 the Ekeland-Hofer capacities of the subsets
 with respect to $\mathbb{R}^{n,k}$,
establish representation formulas for such capacities of bounded convex domains  intersecting with $\mathbb{R}^{n,k}$.
We also prove a product formula and a fact that the value of this capacity  on a hypersurface $\mathcal{S}$
of restricted contact type containing the origin is equal to the action of a generalized leafwise chord on
$\mathcal{S}$.
} \vspace{-0.1in}


\medskip\vspace{12mm}

\section{Introduction}
\setcounter{equation}{0}

\subsection{Coisotropic capacity}\label{sec:coCap}

Recently,  Lisi and  Rieser \cite{LiRi13}
introduced the notion of a  coisotropic capacity (i.e., a symplectic capacity relative to a coisotropic
 submanifold of a symplectic manifold),  and discussed their motivations and backgrounds.
Let $(M,\omega)$ be a symplectic manifold and $N\subset M$ a coisotropic submanifold.
({\it In this paper all manifolds are assumed to be connected without special statements}!)
An equivalence relation $\sim$ on $N$ was called  a  \textbf{coisotropic equivalence relation}
if $x$ and $y$ are on the same leaf then $x\sim y$ (cf. \cite[Definition 1.4]{LiRi13}).
Special examples are the trivial relation defined by $x\sim y$ for every pair $x, y\in N$ and
the so-called  \textbf{leaf relation} defined by
 $x\sim y$ if and only if $x$ and $y$ are on the same leaf.
For two tuples $(M_0, N_0, \omega_0, \sim_0)$ and $(M_1, N_1, \omega_1, \sim_1)$ as above,
 a \textsf{relative symplectic embedding} from
$(M_0, N_0, \omega_0)$ to $(M_1, N_1, \omega_1)$
is a symplectic embedding $\psi: (M_0, \omega_0)\to (M_1, \omega_1)$
satisfying $\psi^{-1}(N_1) = N_0$ (\cite[Definition 1.5]{LiRi13}). Such an embedding
 $\psi$ is said to \textbf{respect the pair of coisotropic equivalence relations}
    $(\sim_0,\sim_1)$
        if for every $x, y \in N_0$,
    $$
    \psi(x) \sim_1 \psi(y)\quad \Longrightarrow \quad x \sim_0 y.
    $$

The standard symplectic space  $(\mathbb{R}^{2n},\omega_0)$
has coisotropic linear subspaces
$$
\mathbb{R}^{n,k}=\{x\in\mathbb{R}^{2n}\,|\,x=(q_1,\cdots,q_n,p_1,\cdots,p_k,0,\cdots,0)\}
$$
for $k=0,\cdots,n$, where we understand
$\mathbb{R}^{n,0}=\{x\in\mathbb{R}^{2n}\,|\,x=(q_1,\cdots,q_n,0,\cdots,0)\}$.
Denote by  $\sim$  the leaf relation on $\mathbb{R}^{n,k}$, and by
\begin{eqnarray}\label{e:V0}
&&V_0^{n,k}=\{x\in\mathbb{R}^{2n}\;|\;x=(0,\cdots,0,q_{k+1},\cdots,q_n,0,\cdots,0)\},\\
&&V^{n,k}_1=\{x\in\mathbb{R}^{2n}\,|\,x=(q_1,\cdots,q_k,0,\cdots,0,p_1,\cdots,p_k,0,\cdots,0)\}.\label{e:V1}
\end{eqnarray}
Hereafter it is understood that $V_0^{n,0}=\{x\in\mathbb{R}^{2n}\;|\;x=(q_{1},\cdots,q_n,0,\cdots,0)\}=\mathbb{R}^{n,0}$,
$V^{n,n}_0=\{0\}$ and $V^{n,0}_1=\{0\}$, $V^{n,n}_1=\mathbb{R}^{2n}$.
Then $L_0^n:=V_0^{n,0}$ is a Lagrangian subspace, and two points $x, y\in\mathbb{R}^{n,k}$  satisfy
  $x\sim y$ if and only if their difference $x-y$ sits in $V_0^{n,k}$.
Obverse that $\mathbb{R}^{2n}$ has the orthogonal decomposition
$\mathbb{R}^{2n}=J_{2n}V^{n,k}_0\oplus \mathbb{R}^{n,k}=J_{2n}\mathbb{R}^{n,k}\oplus V^{n,k}_0$
 with respect to the standard inner product,
 where $J_{2n}$ denotes the standard complex structure  on $\mathbb{R}^{2n}$ given by
 $(q_1,\cdots,q_n, p_1,\cdots, p_n)\mapsto (p_1,\cdots,p_n, -q_1,\cdots, -q_n)$.

For $a\in\mathbb{R}$ we write ${\bf a}:=(0,\cdots,0,a)\in\mathbb{R}^{2n}$. Denote by
$B^{2n}({\bf a}, r)$ and $B^{2n}(r)$
the open balls of radius $r$ centered at ${\bf a}$ and the origin in $\R^{2n}$ respectively, and by
\begin{eqnarray}\label{e:Ball2}
&&W^{2n}(R) := \left \{ (x_1, \dots, x_n, y_1, \dots, y_n) \in \R^{2n} \; | \; x_n^2 + y_n^2  < R^2\; \text{or }\; y_n < 0
\right\},\\
&&  W^{n,k}(R):=  W^{2n}(R) \cap \R^{n,k}\quad\hbox{and}\quad B^{n,k}(r):= B^{2n}(r) \cap \R^{n,k}.\label{e:Ball3}
\end{eqnarray}
($W^{2n}(R)$ was written as $W(R)$ in \cite[Definition~1.1]{LiRi13}).

According to \cite[Definition~1.7]{LiRi13},
a \textsf{coisotropic capacity} is a functor $c$, which assigns to every tuple $(M,N,\omega, \sim)$ as above
 a non-negative (possibly infinite) number $c(M,N,\omega, \sim)$, such that the following
conditions hold:
  \begin{description}
         \item [(i)] {\bf Monotonicity}. If there exists a relative
symplectic embedding $\psi$ from $(M_0, N_0, \omega_0, \sim_0)$ to $(M_1, N_1, \omega_1, \sim_1)$
respecting the coisotropic equivalence relations where $\dim M_0=\dim M_1$, then
$c(M_0,N_0,\omega_0, \sim_0) \leq c(M_1,N_1,\omega_1, \sim_1)$.

    \item [(ii)]{\bf Conformality}.  
$c(M,N,\alpha\omega, \sim)=|\alpha|c(M,N,\omega, \sim),\;\forall\alpha \in \mathbb{R}\backslash \{0\}$.

    \item [(iii)]{\bf Non-triviality}. With the leaf
    relation $\sim$  it holds that for $k=0,\cdots,n-1$,
    \begin{eqnarray}\label{e:Ball4}
    c(B^{2n}(1),B^{n,k}(1),\omega_0, \sim ) =\frac{\pi}{2}=
    c(W^{2n}(1),W^{n,k}(1),\omega_0, \sim ).
    \end{eqnarray}
     \end{description}

As remarked in  \cite[Remark~1.9]{LiRi13},  any symplectic capacity can not serve as a coisotropic capacity because of
the non-triviality (iii).

\textsf{From now on, we abbreviate $c(M,N,\omega,\sim)$ as $c(M,N,\omega)$
if $\sim$ is the leaf relation on $N$.  
In particular, for domains $D\subset\mathbb{R}^{2n}$ we also abbreviate $c\left(D, D\cap\mathbb{R}^{n,k},\omega_0\right)$ as
   $c\left(D, D\cap\mathbb{R}^{n,k}\right)$ for simplicity.}

Given a ($n+k$)-dimensional coisotropic submanifold $N$ in a  symplectic manifold $(M, \omega)$ of dimension $2n$
we defined in \cite[Definition~1.3]{JinLu1917}
$$
{\it w}_G(N;M,\omega):=\sup\left\{\pi r^2\,\Bigg|\,\begin{array}{ll}
&\exists\;\hbox{a relative symplectic embedding}\\
&\hbox{$(B^{2n}(r), B^{n,k}(r))\to (M,N)$ respecting}\\
&\hbox{ the leaf relations on $B^{n,k}(r)$ and $N$}
\end{array}\right\}
$$
the \textbf{relative Gromov width} of $(M, N, \omega)$. Here we always assume $k\in\{0,1\cdots,n-1\}$.
(If $k=n$ then ${\it w}_G(N;M,\omega)$ is equal to the Gromov width ${\it w}_G(N,\omega|_N)$ of $(N,\omega|_N)$.)

When $k=0$, $N$ is a Lagrangian submanifold and this relative Gromov width was introduced by  Barraud, Biran and Cornea 
 \cite{BaCor06, BaCor07, BirCor08, BirCor09}.
It is easily seen that ${\it w}_G$ satisfies monotonicity, conformality and
$$
{\it w}_G(B^{2n}(r)\cap\mathbb{R}^{n,k}; B^{2n}(r),\omega_0) =\pi r^2,\quad\forall r>0.
$$
In fact  ${\it w}_G(N;M,\omega)/2$ is the smallest coisotropic capacity by the nonsqueezing theorem in  \cite{LiRi13}.
Dimitroglou Rizell \cite{Riz15} observed that the Lagrangian submanifolds of $\mathbb{C}^3$
 constructed by Ekholm, Eliashberg, Murphy and Smith \cite{EkElMS13} have infinite relative Gromov width.

%

Similar to the construction of the Hofer-Zehnder capacity,
Lisi and  Rieser \cite{LiRi13} constructed an analogue relative to a coisotropic submanifold, called the
\textbf{coisotropic Hofer-Zehnder
capacity},  and denoted by $c_{\rm LR}$ in this paper.
By peoperties of this coisotropic capacity, they also studied symplectic embeddings relative
to coisotropic constraints and got some corresponding dynamical results.
The coisotropic capacity $c_{\rm LR}$ also played a key role in the proof of
Humili\'ere-Leclercq-Seyfaddini's  important rigidity result  that symplectic homeomorphisms preserve coisotropic submanifolds
 and their characteristic foliations (\cite{HuLeSe15}).

For the  coisotropic capacity
$c_{\rm LR}\left(D, D\cap\mathbb{R}^{n,k}\right)$ of
a bounded convex domain $D\subset\mathbb{R}^{2n}$,
we \cite{JinLu1917} proved a representation formula,
some interesting corollaries, and  corresponding versions of a Brunn-Minkowski type inequality by Artstein-Avidan and Ostrover
 and a theorem by Evgeni Neduv.

\subsection{A relative version of the Ekeland-Hofer capacity with respect to a coisotropic submanifold $\mathbb{R}^{n,k}$}\label{sec:main}

Prompted by  Gromov's work \cite{Gr}, Ekeland and Hofer \cite{EH89, EH90} constructed a sequence of symplectic invariants
for subsets in the standard symplectic space $(\mathbb{R}^{2n},\omega_0)$,  the so-called Ekeland and Hofer symplectic capacities.
(In this paper, the Ekeland and Hofer symplectic capacity always means
the first Ekeland and Hofer symplectic capacity without special statements.)
We introduced the generalized Ekeland-Hofer and  the symmetric Ekeland-Hofer symplectic capacities
and developed corresponding results  (\cite{JinLu1915, JinLu1916}).
The aim of this paper is to construct a coisotropic analogue of the Ekeland-Hofer capacity for subsets in $(\mathbb{R}^{2n},\omega_0)$
relative to a coisotropic submanifold $\mathbb{R}^{n,k}$, the coisotropic Ekeland-Hofer
capacity.

Fix an integer $0\le k\le n$.
For each  subset $B\subset\mathbb{R}^{2n}$ whose closure $\overline{B}$ has nonempty intersection with
$\mathbb{R}^{n,k}$,  we define a number $c^{n,k}(B)$,
called  \textbf{coisotropic Ekeland-Hofer capacity} of $B$ (though it does not satisfy  the stronger monotonicity as in (i)
above (\ref{e:Ball4})),
which is equal to the Ekeland-Hofer capacity of $B$ if $k=n$.
The coisotropic capacity $c^{n,k}$  satisfies
$c^{n,k}(B)=c^{n,k}(\overline{B})$ and the following:

 \begin{proposition}\label{prop:coEHC.2}
Let $\lambda>0$ and $B\subset A\subset\mathbb{R}^{2n}$ satisfy $\overline{B}\cap \mathbb{R}^{n,k}\neq \emptyset$.
Then
\begin{description}
\item[(i)]{\rm (Monotonicity)} $c^{n,k}(B)\le c^{n,k}(A)$.
\item[(ii)] {\rm (Conformality)} $c^{n,k}(\lambda B)=\lambda^2 c^{n,k}(B)$.
\item[(iii)] {\rm (Exterior regularity)} $c^{n,k}(B)=\inf\{c^{n,k}(U_\epsilon(B))\,|\,\epsilon>0\}$
and so $c^{n,k}(B)=c^{n,k}(\overline{B})$, where $U_\epsilon(B)$ is the $\epsilon$-neighborhood of $B$.
\item[(iv)] {\rm (Translation invariance)} $c^{n,k}(B+ w)=c^{n,k}(B)$ for
 all $w\in \mathbb{R}^{n,k}$, where $B+w=\{z+w\,|\, z\in B\}$.
\end{description}
\end{proposition}

The group  ${\rm Sp}(2n)={\rm Sp}(2n,\mathbb{R})$   of symplectic matrices
in $\mathbb{R}^{2n}$ is a connected Lie group. Kun Shi shows in Appendix~\ref{sec:app} that
 its subgroup
\begin{equation}\label{e:subgroup}
{\rm Sp}(2n,k):=\{A\in {\rm Sp}(2n)\,|\, Az=z\;\forall z\in\mathbb{R}^{n,k}\}
\end{equation}
 is also connected. 

\begin{theorem}[\hbox{\rm symplectic invariance}]\label{th:coEHC.4}
Let $B\subset \mathbb{R}^{2n}$ satisfy $\overline{B}\cap \mathbb{R}^{n,k}\neq \emptyset$.
 Suppose that
$\phi\in{\rm Symp}(\mathbb{R}^{2n},\omega_0)$ satisfies for some $w_0\in \mathbb{R}^{n,k}$,
$$
\phi(w)=w-w_0\;\;\forall w\in\mathbb{R}^{n,k}\quad\hbox{and}\quad
d\phi(w_0)\in {\rm Sp}(2n,k).
$$
Then  $c^{n,k}(\phi(B))=c^{n,k}(B)$.
\end{theorem}

\begin{corollary}\label{cor:coEHC.5}
For a subset $A\subset \mathbb{R}^{2n}$ satisfying $\overline{A}\cap \mathbb{R}^{n,k}\neq \emptyset$,
suppose that there exists a starshaped open neighborhood $U$ of $\overline{A}$
with respect to some point $w_0\in\mathbb{R}^{n,k}$ and  a symplectic embedding $\varphi$ from  $U$  to
 $\mathbb{R}^{2n}$ such that
 \begin{equation}\label{e:coEHC.8+}
 \varphi(w)=w-w_0\;\forall w\in\mathbb{R}^{n,k}\cap U\quad\hbox{and}\quad  d\varphi(w_0)\in {\rm Sp}(2n,k).
  \end{equation}
 Then  $c^{n,k}(\varphi(A))=c^{n,k}(A)$. In particular, for a subset $A\subset \mathbb{R}^{2n}$ satisfying $\overline{A}\cap \mathbb{R}^{n,k}\neq \emptyset$,
  if it is  starshaped with respect to some point $w_0\in\mathbb{R}^{n,k}$ and
 there exists a symplectic embedding $\varphi$ from some open neighborhood $U$ of $\overline{A}$ to
 $\mathbb{R}^{2n}$ such that (\ref{e:coEHC.8+}) holds, then $c^{n,k}(\varphi(A))=c^{n,k}(A)$.
\end{corollary}

There exists a natural class of symplectic mappings satisfying the conditions in Corollary~\ref{cor:coEHC.5}.
For $\epsilon>0$ small, let $\mathbb{R}^{n,k}_\epsilon=\{(q_1,\cdots, q_n, p_1,\cdots,p_n)\,|\, p_{k+1}^2+\cdots+ p_n^2<\epsilon^2\}$,
that is, the tubular open neighborhood of $\mathbb{R}^{n,k}$ of radius $\epsilon$.
 Let $U$ be as in Corollary~\ref{cor:coEHC.5}, and let $H:[0,1]\times\mathbb{R}^{2n}\to\mathbb{R}$
 be any smooth Hamiltonian that vanishes in $[0, 1]\times (\mathbb{R}^{n,k}_\epsilon\cap U)$.
 Suppose that $X_H$ can determine a $1$-parameter family of symplectic
mappings $\phi_H^t$ for $t\in [0, 1]$ as usual, (for example, this can be satisfied if $H$ has compact support).
Then $\varphi:=(\psi_{w_0}\circ \phi_H^1)|_U$
satisfies the conditions in Corollary~\ref{cor:coEHC.5},  where $\psi_{w_0}\in {\rm Symp}(\mathbb{R}^{2n},\omega_0)$
is the translation defined by $\psi_{w_0}(w)=w-w_0$ for $w\in\mathbb{R}^{2n}$.

For a bounded convex domain $D$ in $(\mathbb{R}^{2n},\omega_0)$ with boundary  $\mathcal{S}$,
recall that a nonconstant  absolutely continuous curve $z:[0,T]\to \mathbb{R}^{2n}$ (for some $T>0$)
  is said to be a  \textbf{generalized characteristic} on $\mathcal{S}$  if $z([0,T])\subset \mathcal{S}$ and
    $\dot{z}(t)\in J_{2n}N_\mathcal{S}(z(t))\;\hbox{a.e.}$, where
    $N_\mathcal{S}(x)=\{y\in\mathbb{R}^{2n}\,|\, \langle u-x, y\rangle\le 0\;\forall u\in D\}$
    is the normal cone to $D$ at $x\in\mathcal{S}$ (\cite[Definition~1.1]{JinLu1916}).
When $D\cap \mathbb{R}^{n,k}\ne\emptyset$, such a generalized characteristic $z:[0,T]\to\mathcal{S}$
is called a \textbf{generalized leafwise chord} (abbreviated GLC) on $\mathcal{S}$ for $\mathbb{R}^{n,k}$ if
   $z(0), z(T)\in \mathbb{R}^{n,k}$ and $z(0)-z(T)\in V_0^{n,k}$.
(Generalized characteristics and generalized leafwise chords on $\mathcal{S}$
 become characteristics and leafwise chords on $\mathcal{S}$ respectively if $\mathcal{S}$ is of class $C^1$.)
 The action of a GLC $z:[0,T]\to\mathcal{S}$ is defined by
$$
A(z)=\frac{1}{2}\int_0^T\langle -J_{2n}\dot{z},z\rangle dt,
$$
where $\langle\cdot,\cdot\rangle$ denotes the Euclid norm on $\mathbb{R}^{2n}$.
As generalizations of representation formulas for the Ekeland-Hofer capacities of bounded convex domains
we have:

\begin{thm}\label{th:EHconvex}
Let $D\subset \mathbb{R}^{2n}$ be a bounded convex  domain with
 $C^{1,1}$  boundary $S=\partial D$. If $D\cap\mathbb{R}^{n,k}\ne\emptyset$
 (and so $\partial D$ contains at least two points of $\mathbb{R}^{n,k}$),
then there exists a  leafwise chord
  $x^\ast$ on $\partial D$ for $\mathbb{R}^{n,k}$ such that
   \begin{eqnarray}\label{e:fixpt}
A(x^{\ast})&=&\min\{A(x)>0\;|\;x\;\hbox{is a leafwise chord on\;$\partial D$\;for \;$\mathbb{R}^{n,k}$}\}\nonumber\\
&=&c^{n,k}(D)\label{e:fixpt.1}\\
&=&c^{n,k}(\partial D).\label{e:fixpt.2}
\end{eqnarray}
Moreover, if $D\subset \mathbb{R}^{2n}$ is only a bounded convex  domain such that
 $D\cap\mathbb{R}^{n,k}\ne\emptyset$, then the above conclusions are still true
 after all words ``leafwise chord" are replaced by ``generalized leafwise chord".
\end{thm}

This theorem  may be false if the domain is not convex.
 Consider the following domain
\begin{eqnarray*}
A^{2n}(r_1, r_2):=\{z\in\mathbb{R}^{2n}\;|\; r_1^2<|z|<r_2^2\}=B^{2n}(r_2)\setminus\overline{B^{2n}(r_1)},
\end{eqnarray*}
where $0<r_1<r_2<\infty$.
Then by monotonicity of coisotropic Ekeland-Hofer capacity and Theorem 1.4,
$$
c^{n,k}(\partial B^{2n-1}(r_2))\le c^{n,k}(\overline{A^{2n}(r_1, r_2)})\le c^{n,k}(\overline{B^{2n}(r_2)})=c^{n,k}(\partial B^{2n-1}(r_2))
$$
and
$$
c^{n,k}(A^{2n}(r_1, r_2))=c^{n,k}(\overline{A^{2n}(r_1, r_2)})=\begin{cases}
\frac{\pi}{2}r_2^2,\quad k<n,\\
\pi r_2^2, \quad k=n.
\end{cases}
$$
However, since $\partial A^{2n}(r_1, r_2)=\partial B^{2n-1}(r_1)\cup \partial B^{2n-1}(r_2)$,
 \begin{eqnarray*}
&&\min\{A(x)>0\;|\;x\;\hbox{is a leafwise chord on\;$\partial A^{2n}(r_1, r_2)$\;for \;$\mathbb{R}^{n,k}$}\}\nonumber\\
&=&\min\big\{\min\{A(x)>0\;|\;x\;\hbox{is a leafwise chord on\;$\partial B^{2n-1}(r_1)$\;for \;$\mathbb{R}^{n,k}$}\},\\
&&\hspace{30mm}\min\{A(x)>0\;|\;x\;\hbox{is a leafwise chord on\;$\partial B^{2n-1}(r_2)$\;for \;$\mathbb{R}^{n,k}$}\}\big\}\\
&=&\min\big\{c^{n,k}(\partial B^{2n-1}(r_1)), c^{n,k}(\partial B^{2n-1}(r_2))\bigr\}\\
&=&\begin{cases}
\frac{\pi}{2}r_1^2,\quad k<n,\\
\pi r_1^2, \quad k=n.
\end{cases}
\end{eqnarray*}
Hence Theorem 1.4 is false for $A^{2n}(r_1, r_2)$.

Theorem~\ref{th:EHconvex} and \cite[Theorem~1.5]{JinLu1917} show that
 $c^{n,k}(D)=c_{\rm LR}(D,D\cap \mathbb{R}^{n,k})$ for a bounded convex  domain $D\subset\mathbb{R}^{2n}$
as in Theorem~\ref{th:EHconvex}.
It follows from (\ref{e:coCap+}) and interior regularity of $c_{\rm LR}$  that
\begin{eqnarray}\label{e:fixpt.3}
c^{n,k}(D)=c_{\rm LR}(D,D\cap \mathbb{R}^{n,k})
\end{eqnarray}
for any  convex  domain $D\subset\mathbb{R}^{2n}$ such that $D\cap \mathbb{R}^{n,k}\ne\emptyset$.
Hence Theorems~1.6, 1.12 and Corollaries~1.7--1.10 in \cite{JinLu1917} are still true if $c_{\rm LR}$ is replaced by suitable $c^{n,k}$.


Moser \cite{Mos78} first studied Hamiltonian leaf-wise chords  for understanding perturbations of Hamiltonian dynamical systems,
  his framework has been extended in many directions, which promotes the research of symplectic topology, 
see \cite{AlFr10, AlMo10, Dr08, Gin07, Gur10, Ho90, Ka13, LiRi13, Us11, Zi10} etc.
 Given two autonomous $C^{1,1}$ Hamiltonians $H, G:\mathbb{R}^{2n}\to\mathbb{R}$ and a regular energy surface $G^{-1}(c')$,
one may ask the following natural mechanics problem:
\begin{center}
\textsf{Is there a point on $G^{-1}(c')$ from which two particles start,
move respectively along Hamiltonian trajectories of $X_H$ and $X_G$ and after some finite time return to
an intersection point of these two trajectories}?
\end{center}

Suppose for some $c\in\mathbb{R}$ that $D_0:=\{z\in\mathbb{R}^{2n}\,|\, H(z)<c\}$ is a bounded
convex domain whose intersection with $G^{-1}(c')$ is a nonempty relative open subset
in a $(2n-1)$-dimensional coisotropic subspace $V$ in $\mathbb{R}^{2n}$.
Then  Theorem~\ref{th:EHconvex} implies an affirmative answer to the problem. In fact,
since there exists a linear symplectic transformation $\Psi:(\mathbb{R}^{2n},\omega_0)\to (\mathbb{R}^{2n},\omega_0)$
such that $\Psi(V)=\mathbb{R}^{n,n-1}$, we can replace $H$ and $G$ by $H\circ\Psi^{-1}$ and $G\circ\Psi^{-1}$, respectively,
and therefore reduce the question to the case $V=\mathbb{R}^{n,n-1}$.
The desired conclusion follows from Theorem~\ref{th:EHconvex}.

Theorem~\ref{th:EHconvex} is also closely related to the famous Arnold's chord conjecture in \cite[\S8]{Ar86}.
Many cases for this problem have been proved to be true, see \cite{Abb99, Abb04, Cie02, CriHut16, HutTa11, HutTa13, Mer14, Moh01, Riz15,
RizSu20, San12,Zi16} etc. When $k=0$, the intersection $S\cap\mathbb{R}^{n,0}$ is a closed $C^{1,1}$ Legendrian submanifold of dimension $n-1$
in the contact manifold $S$ with the standard contact form, which is diffeomorphic to the sphere $S^{n-1}$,
and Theorem~\ref{th:EHconvex} affirms the conjecture in this case though for the smooth $S$ it was proved by Mohnke \cite{Moh01}
with a different method. Clearly, our result also gives the action of this chord.


As the Ekeland-Hofer capacity, $c^{n,k}$ satisfies the following product formulas,
which play key roles for computations of $c_{\rm LR}$ and the proof of \cite[Theorem~1.12]{JinLu1917}.

\begin{thm}\label{th:EHproduct}
 For   convex domains $D_i\subset\mathbb{R}^{2n_i}$ containing the origin, $i=1,\cdots,m\ge 2$,
 and integers $0\le l_0\le n:=n_1+\cdots+n_m$, $l_j=\max\{l_{j-1}-n_j,0\}$, $j=1,\cdots, m-1$,
    it holds that
 \begin{eqnarray}\label{e:product1}
 c^{n,l_0}(D_1\times\cdots\times D_m)=\min_ic^{n_i,\min\{n_i,l_{i-1}\}}(D_i).
 \end{eqnarray}
 Moreover, if all these domains $D_i$ are also bounded then
 \begin{eqnarray}\label{e:product2}
 c^{n,l_0}(\partial D_1\times\cdots\times \partial D_m)=\min_ic^{n_i,\min\{n_i,l_{i-1}\}}(D_i).
 \end{eqnarray}
 \end{thm}

Hereafter  $\mathbb{R}^{2n_1}\times \mathbb{R}^{2n_2}\times\cdots\times \mathbb{R}^{2n_m}$ is identified with $\mathbb{R}^{2(n_1+\cdots+n_m)}$
via
$$
\mathbb{R}^{2n_1}\times \mathbb{R}^{2n_2}\times\cdots\times \mathbb{R}^{2n_m}
\ni((q^{(1)},p^{(1)}),\cdots, (q^{(m)},p^{(m)}))\mapsto (q^{(1)},\cdots, q^{(m)},p^{(1)},\cdots, p^{(m)})\in \mathbb{R}^{2n}.
$$

 If $l_0=n$ then $l_i=\sum_{j>i}n_j$ and thus $\min\{n_i,l_{i-1}\}=n_i$ for $i=1,\cdots,m$.
It follows that Theorem~\ref{th:EHproduct} becomes Theorem in \cite[\S~6.6]{Sik90}.
We pointed out in 
\cite[Remark~1.11]{JinLu1917} that
Theorems~\ref{th:EHconvex}, ~\ref{th:EHproduct} and \cite[Theorem~1.5]{JinLu1917}
can be combined together to improve some results therein.

\begin{corollary}\label{cor:EHproduct}
Let $S^1(r_i)$ be boundaries of discs $B^2(0,r_i)\subset\mathbb{R}^{2}$, $i=1,\cdots,n\ge 2$,
 and integers $0\le l_0\le n$, $l_j=\max\{l_{j-1}-1,0\}$, $j=1,\cdots, n-1$. Then
 \begin{eqnarray*}
 c^{n,l_0}(S^1(r_1)\times\cdots\times S^1(r_n))=\min_ic^{1,\min\{1,l_{i-1}\}}(B^2(0,r_i)).
 \end{eqnarray*}
Here $c^{1,1}(B^2(0,r_i))=\pi r_i^2$ and $c^{1,0}(B^2(0,r_i))=\pi r_i^2/2$. Precisely,
\begin{eqnarray*}
&&c^{n,0}(S^1(r_1)\times\cdots\times S^1(r_n))=\min\{\pi r_1^2/2, \cdots, \pi r_n^2/2\},\\
&&c^{n,k}(S^1(r_1)\times\cdots\times S^1(r_n))=\min\{\min_{i\le k}\pi r_i^2, \min_{i>k}\pi r_i^2/2\},\quad 0<k<n,\\
&&c^{n,n}(S^1(r_1)\times\cdots\times S^1(r_n))=\min\{\pi r_1^2, \cdots, \pi r_n^2\}.
\end{eqnarray*}
\end{corollary}

%


Note that Corollary~\ref{cor:EHproduct} becomes \cite[Corollary~6.6]{Sik90} for $l_0=n$.

Define  $U^2(1)=\{(q_n,p_n)\in\mathbb{R}^2\,|\,q_n^2+p_n^2<1\;\hbox{or
$-1<q_n<1$ and $p_n<0$}\}$ and
\begin{eqnarray}\label{e:product3.0}
U^{2n}(1)=\mathbb{R}^{2n-2}\times U^2(1)\quad\hbox{and}\quad U^{n,k}(1)=U^{2n}(1)\cap\mathbb{R}^{n,k}.
\end{eqnarray}
 By (\ref{e:fixpt.3}) and
 \cite[Corollary~1.9]{JinLu1917} we obtain for $k=0,1,\cdots,n-1$,
\begin{eqnarray}\label{e:product3.1}
c^{n,k}(U^{2n}(1))=c_{\rm LR}(U^{2n}(1),U^{2n}(1)\cap \mathbb{R}^{n,k})=\frac{\pi}{2}.
\end{eqnarray}

The proof of Theorem~\ref{th:EHproduct} relies partially on the representation of coisotropic Ekeland-Hofer capacity of convex domains given by Theorem~\ref{th:EHconvex}. It is possible that Theorem~\ref{th:EHproduct} is still true for some product of  non-convex domains.
For integers $0\le l_0\le n:=n_1+ n_2$ and $l_1=\max\{l_0-n_1, 0\}$, as the arguments below Theorem~\ref{th:EHproduct} we can get
\begin{eqnarray*}
c^{n,l_0}({B^{2n_1}(r_2)}\times B^{2n_2}(r_3))&=&
c^{n,l_0}(\partial B^{2n_1-1}(r_2)\times \partial B^{2n_2-1}(r_3))\\
&\le& c^{n,l_0}(\overline{A^{2n_1}(r_1, r_2)}\times \overline{B^{2n_2}(r_3)})\\
&\le& c^{n,l_0}(\overline{B^{2n_1}(r_2)}\times \overline{B^{2n_2}(r_3)})
\end{eqnarray*}
and therefore
\begin{eqnarray*}
c^{n,l_0}({A^{2n_1}(r_1, r_2)}\times B^{2n_2}(r_3))&=&c^{n,l_0}(\overline{A^{2n_1}(r_1, r_2)}\times \overline{B^{2n_2}(r_3)})\\
&=&c^{n,l_0}(\overline{B^{2n_1}(r_2)}\times\overline{ B^{2n_2}(r_3)})\\
&=&\min\{c^{n_1,\min\{n_1,l_0\}}(B^{2n_1}(r_2)), c^{n_2,\min\{n_2,l_{1}\}}(B^{2n_2}(r_3))\}\\
&=&\begin{cases}
\min\{\frac{\pi}{2}r_2^2, \frac{\pi}{2}r_3^2\},\quad l_0<n_1,\\
\min\{\pi r_2^2, c^{n_2, l_0-n_1}(B^{2n_2}(r_3))\}, \quad l_0\ge n_1.
\end{cases}
\end{eqnarray*}
On the other hand
\begin{eqnarray*}
&&\min\bigl\{c^{n_1,\min\{n_1,l_0\}}(A^{2n_1}(r_1, r_2)), c^{n_2,\min\{n_2,l_1\}}(B^{2n_2}(r_3))\bigr\}\\
&=&\begin{cases}
\min\{\frac{\pi}{2}r_2^2, \frac{\pi}{2}r_3^2\},\quad l_0<n_1,\\
\min\{\pi r_2^2, c^{n_2, l_0-n_1}(B^{2n_2}(r_3))\}, \quad l_0\ge n_1.
\end{cases}
\end{eqnarray*}
Hence (\ref{e:product1}) is also true for the produce of ${A^{2n_1}(r_1, r_2)}$ and $B^{2n_2}(r_3)$.

Recall that a vector field $X$ defined on an open set $U\subset\mathbb{R}^{2n}$ is called a Liouville vector field if $L_X\omega_0=\omega_0$. A hypersurface $\mathcal{S}\subset\mathbb{R}^{2n}$ is said to be of  restricted contact type if there exists a Liouville vector field $X$ globally defined on $\mathbb{R}^{2n}$ which is transversal to $\mathcal{S}$. Corresponding to the representation of the Ekeland-Hofer capacity of a bounded domain in $\mathbb{R}^{2n}$
with boundary of restricted contact type we have:

\begin{thm}\label{th:EHcontact}
Let $U\subset (\mathbb{R}^{2n},\omega_0)$ be a bounded domain with $C^{2n+2}$ boundary
$\mathcal{S}$ of restricted contact type. Suppose that $U$ contains the origin and that
there exists a globally defined $C^{2n+2}$ Liouville vector field $X$ transversal to $\mathcal{S}$
 whose flow $\phi^t$ maps $\mathbb{R}^{n,k}$ to $\mathbb{R}^{n,k}$ and preserves the leaf relation of $\mathbb{R}^{n,k}$.  Then
   \begin{equation}\label{e:rescontact}
   \Sigma_{\mathcal{S}}:=
  \{A(x)>0\;|\;x\;\hbox{is a leafwise chord on\;$\mathcal{S}$\;for \;$\mathbb{R}^{n,k}$}\}.
 \end{equation}
 has empty interior and contains $c^{n,k}(U)=c^{n,k}(\mathcal{S})$.
\end{thm}

In order to show that $c^{n,k}$ is a coisotropic capacity (with the weaker monotonicity), we need to prove that $c^{n,k}$ satisfies
the non-triviality as in (\ref{e:Ball4}).
By Theorem~\ref{th:EHconvex} we immediately obtain
\begin{eqnarray}\label{e:Ball5}
    c^{n,k}(B^{2n}(1)) =\frac{\pi}{2},\quad k=0,\cdots,n-1.
    \end{eqnarray}
Proposition~\ref{prop:coEHC.2}(i)
and (\ref{e:product3.1}) also lead to
$c^{n,k}(W^{2n}(1))\ge c^{n,k}(U^{2n}(1)) =\frac{\pi}{2}$ directly.
Using the extension monotonicity of $c_{\rm LR}$ in \cite[Lemma~2.4]{LiRi13},
Lisi and  Rieser proved that
$$
c_{\rm LR}\left(W^{2n}(1), W^{n,k}(1)\right) =c_{\rm LR}\left(U^{2n}(1), U^{n,k}(1)\right)
$$
above \cite[Proposition~3.1]{LiRi13}. However, our Proposition~\ref{prop:coEHC.2} and Theorem~\ref{th:coEHC.4}
cannot yield such strong extension monotonicity for $c^{n,k}$. Instead, we may use
Theorem~\ref{th:EHproduct} and 
Theorem~\ref{th:EHcontact},
(though the latter does not hold for $c_{\rm LR}$ in general), to derive:

\begin{thm}\label{th:non-triviality}
 For $k=0,\cdots,n-1$, it holds that
    \begin{eqnarray*}
    c^{n,k}(W^{2n}(1)) =\frac{\pi}{2}.
    \end{eqnarray*}
\end{thm}

By this theorem, Corollary~\ref{cor:EHproduct} and Theorem~\ref{th:coEHC.4} we deduce:

\begin{cor}\label{cor:embedding}
If $\min\{2\min_{i\le k}r_i^2, \min_{i>k}r_i^2\}>1$ for some $0<k<n$, then
there is no $\phi\in{\rm Symp}(\mathbb{R}^{2n},\omega_0)$ which
satisfies $\phi(w)=w-w_0\;\;\forall w\in\mathbb{R}^{n,k}$ and $d\phi(w_0)\in {\rm Sp}(2n,k)_0$
for some $w_0\in \mathbb{R}^{n,k}$,
such that $\phi$ maps
$S^1(r_1)\times\cdots\times S^1(r_n)= \{ (x_1, \dots, x_n, y_1, \dots, y_n) \in \R^{2n} \; | \; x_i^2 + y_i^2=r_i^2,\;i=1,\cdots,n\}$
 into $W^{2n}(1)$.
\end{cor}

Under the assumptions of Corollary~\ref{cor:embedding} it is easy to see that there always exists a
$\phi\in{\rm Symp}(\mathbb{R}^{2n},\omega_0)$ such that
$\phi(S^1(r_1)\times\cdots\times S^1(r_n))\subset W^{2n}(1)$.

Let $\tau_0\in \mathcal{L}(\mathbb{R}^{2n})$ be the canonical involution on $\mathbb{R}^{2n}$ given by
$\tau_0(x,y)=(x,-y)$. For a subset $B\subset\mathbb{R}^{2n}$ such that $\tau_0B=B$ and $B\cap L^n_0\neq \emptyset$,
let $c_{\rm EH,\tau_0}(B)$ be the $\tau_0$-symmetrical Ekeland-Hofer capacity constructed in
\cite{JinLu1915}.  We shall prove in Section~\ref{sec:compare}:

\begin{thm}\label{th:compare}
The $\tau_0$-symmetrical Ekeland-Hofer capacity $c_{\rm EH,\tau_0}(B)$ of each subset $B\subset\mathbb{R}^{2n}$
satisfying $\tau_0B=B$ and $B\cap L^n_0\neq \emptyset$
is greater than or equal to $c^{n,0}(B)$.
\end{thm}


\noindent{\bf Plan of the paper}.  
In Section 2 we provide  necessary variational  preparations
on the basis of \cite{LiRi13, JinLu1917}.
In  Section~\ref{sec:def} we give the definition of the coisotropic Ekeland-Hofer capacity and proofs of
Proposition~\ref{prop:coEHC.2}, Theorem~\ref{th:coEHC.4} and
Corollary~\ref{cor:coEHC.5}.
In Section~\ref{sec:formula} we prove Theorem~\ref{th:EHconvex}.
In Section 5 we prove a product formula, Theorem~\ref{th:EHproduct}.
In Section 6 we prove Theorem~\ref{th:EHcontact} about the representation of the coisotropic capacity $c^{n,k}$ of a bounded domain in $\mathbb{R}^{2n}$
with boundary of restricted contact type.
In Section 7 we prove Theorem~\ref{th:non-triviality}.


\noindent{\bf Acknowledgments}. We are deeply grateful to the
anonymous referees for giving very helpful comments and suggestions to improve the
exposition.

\section{Variational  preparations}\label{sec:2}
\setcounter{equation}{0}

We follow \cite{LiRi13} and \cite{JinLu1917} to present necessary variational  materials.
Fix an integer $0\le k<n$.
Consider the Hilbert space defined in \cite[Definition 3.6]{LiRi13}
\begin{eqnarray}\label{e:space1}
L^2_{n,k}=\Big\{x\in L^2([0,1],\mathbb{R}^{2n})\,\Big| &&x\stackrel{L^2}{=}\sum_{m\in\mathbb{Z}}e^{m\pi tJ_{2n}}a_m+\sum_{m\in\mathbb{Z}}e^{2m\pi tJ_{2n}}b_m\nonumber\\
&&a_m\in V^{n,k}_0,\quad b_m\in V^{n,k}_1,\nonumber\\
&&\sum_{m\in\mathbb{Z}}(|a_m|^2+|b_m|^2)<\infty
\Big\}
\end{eqnarray}
with $L^2$-inner product.
We proved  in \cite[Proposition~2.3]{JinLu1917} that the Hilbert space $L^2_{n,k}$ is exactly $L^2([0,1],\mathbb{R}^{2n})$.
(If $k=n$ this is clear as usual because  $V^{n,n}_0=\{0\}$ and $V^{n,n}_1=\mathbb{R}^{2n}$.)
For any real $s\ge0$ we follow \cite[Definition~3.6]{LiRi13} to define
\begin{eqnarray}\label{e:space2}
H^s_{n,k}=\Big\{x\in L^2([0,1],\mathbb{R}^{2n})\,\Big| &&x\stackrel{L^2}{=}\sum_{m\in\mathbb{Z}}e^{m\pi tJ_{2n}}a_m+\sum_{m\in\mathbb{Z}}e^{2m\pi tJ_{2n}}b_m\nonumber\\
&&a_m\in V^{n,k}_0, \quad b_m\in V^{n,k}_1,\nonumber\\
&&\sum_{m\in\mathbb{Z}}|m|^{2s}(|a_m|^2+|b_m|^2)<\infty
\Big\}.
\end{eqnarray}

\begin{lemma}[\hbox{\cite[Lemmas 3.8, 3.9]{LiRi13}}]\label{lem:compactemb}
For each $s\ge 0$, $H^s_{n,k}$ is a Hilbert space with the inner product
$$
\langle\phi,\psi\rangle_{s,n,k}=\langle a_0,a_0'\rangle+\langle b_0,b_0'\rangle
+\pi\sum_{m\ne 0}(|m|^{2s}\langle a_m, a_m'\rangle+|2m|^{2s}\langle b_m, b_m'\rangle).
$$
Furthermore, if $s>t$, then the inclusion $\jmath:H^s_{n,k}\hookrightarrow H^t_{n,k}$
and its Hilbert adjoint $\jmath^\ast:H^t_{n,k}\rightarrow H^s_{n,k}$ are compact.
\end{lemma}

Let $\|\cdot\|_{s,n,k}$ denote the norm induced by $\langle\cdot,\cdot\rangle_{s,n,k}$.
For $r\in\mathbb{N}$ or $r=\infty$ let
$C^{r}_{n,k}([0,1],\mathbb{R}^{2n})$
denote the space of $C^r$ maps $x:[0,1]\to\mathbb{R}^{2n}$ such that $x(i)\in\mathbb{R}^{n,k}$, $i=0,1$, and $x(1)\sim x(0)$,
 where $\sim$ is the leaf relation on $\mathbb{R}^{n,k}$.
 ({\bf Note}: $H^s_{n,n}$ is exactly the space $H^s$ on the page 83 of \cite{HoZe94};
 $C^{r}_{n,n}([0,1],\mathbb{R}^{2n})$ is $C^r(\mathbb{R}/\mathbb{Z}, \mathbb{R}^{2n})$.)

\begin{lemma}[\hbox{\cite[Lemma 3.10]{LiRi13}}]
If $x\in H^s_{n,k}$ for $s>1/2+r$ where $r$ is an integer, then $x\in C^{r}_{n,k}([0,1],\mathbb{R}^{2n})$.
\end{lemma}

\begin{lemma}[\hbox{\cite[Lemma 3.11]{LiRi13}}]
$\jmath^\ast(L^2)\subset H^1_{n,k}$ and $\|\jmath^\ast(y)\|_{1,n,k}\le \|y\|_{L^2}$.
\end{lemma}

Let
\begin{equation}\label{e:space3}
E=H^{1/2}_{n,k}\quad\hbox{and}\quad \|\cdot\|_E:=\|\cdot\|_{1/2,n,k}.
\end{equation}
It has an orthogonal decomposition $E=E^-\oplus E^0\oplus E^+$, where
\begin{eqnarray*}
&&E^-=\Big\{x\in H^{1/2}_{n,k}\,\Big|\,x\stackrel{L^2}{=}\sum_{m<0}e^{m\pi tJ_{2n}}a_m+\sum_{m<0}e^{2m\pi tJ_{2n}}b_m\Big\},\\
&&E^0=\{x=x_0\in\mathbb{R}^{n,k}\},\\
&&E^+=\Big\{x\in H^{1/2}_{n,k}\,\Big|\,x\stackrel{L^2}{=}\sum_{m>0}e^{m\pi tJ_{2n}}a_m+\sum_{m>0}e^{2m\pi tJ_{2n}}b_m\Big\}.
\end{eqnarray*}
Let $P^+$, $P^0$ and $P^-$ be the orthogonal projections to $E^+$, $E^0$ and $E^-$
respectively. For $x\in E$ we write
$$x^+=P^+x,\quad x^0=P^0x\quad\hbox{and} \quad x^-=P^-x.$$
Define a functional $\mathfrak{a}:E\rightarrow\mathbb{R}$ by
$$
\mathfrak{a}(x)=\frac{1}{2}(\|x^+\|^2_E-\|x^-\|^2_E).
$$
Then there holds
$$
\mathfrak{a}(x)=\frac{1}{2}\int_0^1\langle -J_{2n}\dot{x},x\rangle dt,\quad\forall x\in C^1_{n,k}([0,1],\mathbb{R}^{2n}).
$$
(See \cite{LiRi13}.) The functional $\mathfrak{a}$ is differentiable with gradient
$\nabla \mathfrak{a}(x)=x^+-x^-$.

From now on we assume that for some $L>0$,
\begin{equation}\label{e:function}
H\in C^1(\mathbb{R}^{2n}, \mathbb{R})\quad\hbox{and}\quad
\|\nabla H(x)-\nabla H(y)\|_{\mathbb{R}^{2n}}\le L\|x-y\|_{\mathbb{R}^{2n}}\;\forall x,y\in \mathbb{R}^{2n}.
\end{equation}
Then there exist positive real numbers $C_i$, $i=1,2,3,4$, such that
$$
|\nabla H(z)|\le C_1|z|+C_2,\quad |H(z)|\le C_3|z|^2+C_4
$$
for all $z\in\mathbb{R}^{2n}$. Define functionals $\mathfrak{b}, \Phi_H: E\rightarrow\mathbb{R}$ by
\begin{equation}\label{e:functional}
\mathfrak{b}(x)=\int_0^1 H(x(t))dt\quad\hbox{and}\quad\Phi_H=\mathfrak{a}-\mathfrak{b}.
\end{equation}

\begin{lemma}[\hbox{\cite[Section~3.3, Lemma~4]{HoZe94}}]\label{lem:Lipsch}
The functional $\mathfrak{b}$ is differentiable. Its gradient $\nabla \mathfrak{b}$ is compact and
satisfies a global Lipschitz condition on $E$. In particular, $\mathfrak{b}$ is $C^{1,1}$.
\end{lemma}

\begin{lemma}[\hbox{\cite[Lemma 2.8]{JinLu1917}}]\label{lem:critic}
$x\in E$ is a critical point of $\Phi_H$ if and only if $x\in C^1_{n,k}([0,1],\mathbb{R}^{2n})$ and solves
$$
\dot{x}=X_H(x)=J_{2n}\nabla H(x).
$$
Moreover, if $H$ is of class $C^l$ ($l\ge 2$) then each critical point of $\Phi_H$ on $E$ is $C^l$.
\end{lemma}

Since $\nabla\Phi_H(x)=x^+-x^--\nabla\mathfrak{b}(x)$ satisfies the global Lipschitz condition,
it has a unique global flow $\mathbb{R}\times E\to E: (u,x)\mapsto
\varphi_u(x)$.

\begin{lemma}[\hbox{\cite[Lemma 3.25]{LiRi13}}]\label{lem:flow}
$\varphi_u(x)$ has the following form
$$
\varphi_u(x) = e^{-u} x^{-} + x^{0} + e^{u}x^{+} + K(u,x),
$$
where $K:\mathbb{R}\times E \to E$ is continuous and maps bounded sets into precompact
sets.
\end{lemma}

This may follow from the proof of Lemma~7 in \cite[Section~3.3]{HoZe94} directly.

\section{The  Ekeland-Hofer capacity relative to a coisotropic subspace}\label{sec:def}
\setcounter{equation}{0}


We  closely follow  Sikorav's approach \cite{Sik90} to the Ekeland-Hofer capacity in \cite{EH89}.
Fix an integer $0\le k\le n$.
Let  $E=H^{1/2}_{n,k}$ be as in (\ref{e:space3}) and $S^+=\{x\in E^+\,|\,\|x\|_E=1\}$.

\begin{definition}\label{defdeform}
  A continuous map $\gamma:E\rightarrow E$ is called an {\bf admissible deformation} if there exists
  a homotopy $(\gamma_u)_{0\le u\le 1}$ such that $\gamma_0={\rm id}$, $\gamma_1=\gamma$ and satisfies
\begin{description}
\item[(i)] $\forall u\in [0,1]$, $\gamma_u(E\setminus(E^-\oplus E^0))\subset E\setminus(E^-\oplus E^0)$, i.e.,
 $\gamma_u(x)^+\neq 0$ for any $x\in E$ such that $x^+\neq 0$.
\item[(ii)]  $\gamma_u(x)=a(x,u)x^++b(x,u)x^0+c(x,u)x^-+K(x,u)$, where $(a,b,c,K)$ is a continuous map
from $E\times [0,1]$ to $(0,+\infty)^3\times E$ and maps any closed bounded sets to compact sets.
\end{description}
\end{definition}

Let $\Gamma_{n,k}$ be the set of all admissible deformations on $E$.
It is not hard to verify that the composition $\gamma \circ\tilde{\gamma} \in\Gamma_{n,k}$ for
 any  $\gamma, \tilde{\gamma}\in\Gamma_{n,k}$. (If $k=n$, $\Gamma_{n,k}$ is equal to $\Gamma$ in  \cite{Sik90}.)
  Corresponding to \cite[Section 3, Proposition 1]{Sik90} or \cite[Section II, Proposition 1]{EH89}
we  can easily prove  the following intersection property.

\begin{proposition}\label{prop:EH.1.1}
   $\gamma(S^+)\cap (E^-\oplus E^0\oplus \mathbb{R}_+e)\neq \emptyset$ for any $e\in E^+\setminus\{0\}$ and $\gamma\in\Gamma_{n,k}$.
\end{proposition}

\begin{definition}
 For $H\in C^0(\mathbb{R}^{2n},\mathbb{R}_{\ge 0})$,  the $\mathbb{R}^{n,k}$-\textsf{coisotropic capacity of $H$} is defined by
\begin{eqnarray}\label{e:coCap0}
c^{n,k}(H):=\sup_{h\in\Gamma_{n,k}}\inf_{x\in h(S^+)}\Phi_H(x)
\end{eqnarray}
where $\Phi_H$ is as in (\ref{e:functional}).
\end{definition}

By Proposition~1 in \cite[Section~3.3]{Sik90}, for any $H\in C^0(\mathbb{R}^{2n},\mathbb{R}_{\ge 0})$ there holds
 \begin{eqnarray}\label{e:coCap1}
 c^{n,n}(H)\le \sup_{z\in\mathbb{C}^n}\left(\pi|z_1|^2-H(z)\right),
 \end{eqnarray}
 where $z_1\in\mathbb{C}$ is the projection of $z\in\mathbb{C}^n\equiv\mathbb{C}\times\mathbb{C}^{n-1}$ to $\mathbb{C}$.
  Correspondingly, we have

\begin{proposition}\label{prop:EH.1.3}
For any $H\in C^0(\mathbb{R}^{2n},\mathbb{R}_{\ge 0})$ there holds
  \begin{equation}\label{e:EH.1.3}
  c^{n,k}(H)\le \sup_{z\in\mathbb{C}^n}\left(\frac{ \pi}{2}|z|^2-H(z)\right),\quad k=0,1,\cdots,n-1.
\end{equation}
\end{proposition}

\begin{proof}
Let $e(t)=e^{\pi J_{2n}t}X$, where $X\in V^{n,k}_0$ and $|X|=1$.   For any $x=y+\lambda e$, where
$y\in E^-\oplus E^0$ and $\lambda>0$, it holds that
$$
\mathfrak{a}(x)\le\frac{1}{2}\|\lambda e\|^2_E= \frac{ \pi }{2}\lambda^2
$$
 and
$$
\int_0^1\langle x(t),e^{ \pi J_{2n}t}X\rangle dt=\int_0^1\langle \lambda e^{ \pi J_{2n}t}X,e^{ \pi J_{2n}t}X\rangle dt=\lambda.
$$
It follows that
$$
\mathfrak{a}(x)\le\frac{\pi }{2}\left(\int_0^1\langle x(t),e^{ \pi J_{2n}t}X\rangle dt\right)^2\le\frac{ \pi }{2}\int_0^1|x(t)|^2 dt.
$$
This and Proposition~\ref{prop:EH.1.1} lead to
$$
\inf_{x\in\gamma(S^+)}\Phi_H(x) \le\sup_{x\in E^-\oplus E^0\oplus\mathbb{R}_+e}\Phi_H(x)\le \sup_{z\in\mathbb{R}^{2n}}\left\{\frac{\pi }{2}|z|^2-H(z)\right\}
\quad\forall\gamma\in\Gamma_{n,k},
$$
and hence (\ref{e:EH.1.3}) is proved.
\end{proof}

 A function $H\in C^0(\mathbb{R}^{2n},\mathbb{R}_{\ge 0})$ is called $\mathbb{R}^{n,k}$-{\bf admissible} if it satisfies:
\begin{description}
\item[(H1)] ${\rm Int}(H^{-1}(0))\ne \emptyset$ and intersects with $\mathbb{R}^{n,k}$,
\item[(H2)] there exists $z_0\in\mathbb{R}^{n,k}$, real numbers $a, b$
such that $H(z)=a|z|^2+ \langle z, z_0\rangle+ b$ outside a compact subset of $\mathbb{R}^{2n}$,
where $a>\pi$ for $k=n$, and $a>\pi/2$ for $0\le k<n$.
\end{description}

Moreover,  a $\mathbb{R}^{n,n}$-admissible $H$
is said to be  {\bf  nonresonant} if $a$ in (H2) does not belong to $\pi\mathbb{N}$; and
a $\mathbb{R}^{n,k}$-admissible $H$ with $k<n$
is called {\bf strong nonresonant} if $a$ in (H2) does not sit in
$\mathbb{N}\pi/2$.

Clearly, for any $\mathbb{R}^{n,k}$-admissible $H\in C^2(\mathbb{R}^{2n},\mathbb{R}_{\ge 0})$,
$\nabla H:\mathbb{R}^{2n}\to\mathbb{R}^{2n}$ satisfies a global Lipschitz condition.

Note that $c^{n,k}(H)<+\infty$ if $H\in C^0(\mathbb{R}^{2n},\mathbb{R}_{\ge 0})$ satisfies
\begin{equation}\label{e:EH.2.1.0}
H(z)\ge a|z|^2+C,\quad\forall z\in\mathbb{R}^{2n}
\end{equation}
for some constant $C$,
 where  $a=\pi$ for $k=n$, and $a=\pi/2$ for $0\le k<n$.
   In particular, we have $c^{n,k}(H)<+\infty$ for any $H\in C^0(\mathbb{R}^{2n},\mathbb{R}_{\ge 0})$ satisfying (H2).
 In fact, for $k=n$ this can be derived from (\ref{e:coCap1})
(cf. \cite{Sik90}). For $0\le k<n$, since there exist constants $a>\pi/2, b$
such that $H(z)\ge a|z|^2+\langle z,z_0\rangle+b$ for all $z\in\mathbb{R}^{2n}$,
using the inequality
$$
|\langle z,z_0\rangle|\le \varepsilon |z|^2+ \frac{1}{4\varepsilon}|z_0|^2
$$
for any $0<\varepsilon<a-\frac{\pi}{2}$,
we deduce that
$$
\frac{\pi}{2}|z|^2-H(z)\le \left(\varepsilon-\left(a-\frac{\pi}{2}\right)\right)|z|^2+\frac{|z_0|^2}{4\varepsilon}-b\le
\frac{|z_0|^2}{4\varepsilon}-b<\infty.
$$
Then Proposition~\ref{prop:EH.1.3} leads to $c^{n,k}(H)<+\infty$.

It is easy proved that  $c^{n,k}(H)$ satisfies:

\begin{proposition}\label{prop:coEHC.1}
Let $H, K\in C^0(\mathbb{R}^{2n},\mathbb{R}_{\ge 0})$  satisfy
(H1) and (H2). Then the following holds:
\begin{description}
\item[(i)]{\rm (Monotonicity)} If $H\le K$ then $c^{n,k}(H)\ge c^{n,k}(K)$.
\item[(ii)] {\rm (Continuity)} $|c^{n,k}(H)-c^{n,k}(K)|\le \sup_{z\in\mathbb{R}^{2n}}|H(z)-K(z)|$.
\item[(iii)] {\rm (Homogeneity)}  $c^{n,k}(\lambda^2H(\cdot/\lambda))=\lambda^2 c^{n,k}(H)$ for $\lambda\ne 0$.
\end{description}
\end{proposition}

 By Proposition~2 in \cite[Section~3.3]{Sik90}
  the following proposition holds for $k=n$.

\begin{proposition}\label{prop:EH.1.4}
Suppose that   $H\in C^0(\mathbb{R}^{2n},\mathbb{R}_{\ge 0})$ satisfies
 \begin{equation}\label{e:EH.2.1+}
 H(z_0+z)\le C_1|z|^2\quad\hbox{and}\quad
 H(z_0+z)\le C_2|z|^3\quad\forall z\in\mathbb{R}^{2n}
\end{equation}
for some $z_0\in\mathbb{R}^{n,k}$ and for constants $C_1>0$ and $C_2>0$. Then $c^{n,k}(H)>0$.
In particular,  $c^{n,k}(H)>0$ for any $\mathbb{R}^{n,k}$-admissible $H\in C^2(\mathbb{R}^{2n},\mathbb{R}_{\ge 0})$.
\end{proposition}

\begin{proof}
We assume $k<n$. For a constant $\varepsilon>0$
define $\gamma_\varepsilon\in\Gamma_{n,k}$ by
$\gamma_\varepsilon(x)=z_0+\varepsilon x\;\forall x\in E$. We claim that
\begin{equation}\label{e:EH.2.2-}
\inf_{y\in\gamma_\varepsilon(S^+)}\Phi_H(y)>0
\end{equation}
for sufficiently small $\varepsilon$. Since
\begin{equation}\label{e:EH.2.2}
\Phi_H(z_0+x)=\frac{1}{2}\|x\|^2_{E}-\int_0^1H(z_0+x) dt\quad\forall x\in E^+,
\end{equation}
it suffices to prove that
\begin{equation}\label{e:EH.2.3}
\lim_{\|x\|_E\rightarrow 0}\frac{\int_0^1 H(z_0+x) dt}{\|x\|^2_E}=0.
\end{equation}

Otherwise, suppose there exists a sequence $(x_j)\subset E$ and $d>0$ satisfying
  \begin{equation}\label{e:EH.2.4}
  \|x_j\|_E\rightarrow 0 \quad \hbox{and} \quad \frac{\int_0^1 H(z_0+x_j) dt}{\|x_j\|^2_E}
  \geq d>0\quad\forall j.
  \end{equation}
  Let $y_j=\frac{x_j}{\|x_j\|_E}$ and hence $\|y_j\|_E=1$. Then Lemma
  \ref{lem:compactemb} implies that $(y_j)$ has a convergent subsequence in $L^2$.
  By a standard result in $L^p$ theory,
  we have $w\in L^2$ and a subsequence of $(y_j)$, still denoted by $(y_j)$, such that
   $y_j(t)\rightarrow y(t)$ a.e. on $(0,1)$ and that
   $|y_j(t)|\leq w(t)$  a.e. on $(0,1)$ for each $j$.
  It follows from (\ref{e:EH.2.1+}) that
  \begin{eqnarray*}
  &&\frac{H(z_0+x_j(t))}{\|x_j\|_E^2}\leq C_1\frac{|x_j(t)|^2}{\|x_j\|_E^2}
  =C_1|y_j(t)|^2\le C_1w(t)^2,\quad \hbox{a.e. on $ (0,1)$},\;\forall j,\\
 &&  \frac{H(z_0+x_j(t))}{\|x_j\|_E^2}\leq C_2\frac{|x_j(t)|^3}{\|x_j\|_E^2}
  =C_2|x_j(t)|\cdot |y_j(t)|^2\le C_2|x_j(t)|w(t)^2,\quad \hbox{a.e. on $(0,1)$},\;\forall j.
   \end{eqnarray*}
  The first claim in (\ref{e:EH.2.4}) implies that $(x_j)$
  has a subsequence such that
  $x_{j_l}(t)\rightarrow 0$, a.e. in $(0,1)$.
   Hence the Lebesgue dominated convergence theorem leads to
  $$
  \int_0^1 \frac{H(z_0+x_{j_l}(t))}{\|x_{j_l}\|_E^2} dt\rightarrow 0.
  $$
  This contradicts the second claim in (\ref{e:EH.2.4}).

  For any fixed $\mathbb{R}^{n,k}$-admissible $H\in C^2(\mathbb{R}^{2n},\mathbb{R}_{\ge 0})$,
   pick some $z_0\in\mathbb{R}^{n,k}\cap{\rm Int}(H^{-1}(0))$. Since (H1) implies that
  $H$ vanishes near $z_0$,  by (H2) and the Taylor expansion of $H$ at $z_0\in\mathbb{R}^{2n}$,
  we have constants $C_1>0$ and $C_2>0$ such that $H$ satisfes (\ref{e:EH.2.1+}).
   \end{proof}

 By (\ref{e:coCap1}) and Propositions~\ref{prop:EH.1.3} and~\ref{prop:EH.1.4} we see that
 $c^{n,k}(H)$ is a finite positive number for each
  $\mathbb{R}^{n,k}$-admissible $H$. The following is a generalization of Lemma~3 in \cite[\S3.4]{Sik90}.

\begin{lemma}\label{lem:PositiveInvariance}
Let $H\in C^0(\mathbb{R}^{2n},\mathbb{R}_{\ge 0})$ satisfy (\ref{e:EH.2.1.0}) and
(\ref{e:EH.2.1+}). Then
$$c^{n,k}(H)=\sup_{F\in\mathcal{F}_{n,k}}\inf_{x\in F}\Phi_H(x),
$$
 where
 \begin{equation}\label{e:fnk}
   \mathcal{F}_{n,k}:=\{\gamma(S^+)\,|\,\gamma\in\Gamma_{n,k}
   \;\text{and}\;\inf(\Phi_H|_{\gamma(S^+)})>0\}.
   \end{equation}
Moreover, if $H$ is also of class $C^2$ and has bounded derivatives of second order, then
$\mathcal{F}_{n,k}$ is positive invariant under the flow $\varphi_u$ of $\nabla \Phi_H$
(which must exist as pointed out above Lemma~\ref{lem:flow}).
   \end{lemma}
\begin{proof}
  Since $c^{n,k}(H)$  is a finite positive number
by Proposition~\ref{prop:EH.1.4} ,
the first claim follows from the arguments above Proposition~\ref{prop:coEHC.1}.

When $H$ has  bounded derivatives of second order, (\ref{e:function}) is satisfied naturally.
Then $\nabla\Phi_H$ satisfies the global Lipschitz condition,
and thus  has a unique global flow $\mathbb{R}\times E\to E: (u,x)\mapsto
\varphi_u(x)$ satisfying Lemma~\ref{lem:flow}, that is,
  $\varphi_u(x)=e^{-u}x^-+x^0+e^ux^++\widetilde{K}(u,x)$,
  where $\widetilde{K}:\mathbb{R}\times E\rightarrow E$ is 
  is continuous and maps bounded sets into precompact sets.
  For a set $F=\gamma(S^+)\in\mathcal{F}_{n,k}$ with $\gamma\in\Gamma_{n,k}$,  we have $\alpha:=\inf(\Phi_H|_{\gamma(S^+)})>0$ by the
  definition of $\mathcal{F}_{n,k}$. Let $\rho:\mathbb{R}\rightarrow [0,1]$ be a smooth function such that $\rho(s)=0$ for $s\le 0$ and $\rho(s)=1$ for $s\ge \alpha$. Define a vector field $V$ on $E$ by
  $$
  V(x)=x^+-x^--\rho(\Phi_H(x))\nabla \mathfrak{b}(x).
  $$
 Clearly $V$ is locally Lipschitz and has linear growth. These imply that $V$ has a unique global flow,
  denoted by $\Upsilon_u$.
  Moreover, it is obvious that $\Upsilon_u$
   has the same property as $\varphi_u$ described in Lemma~\ref{lem:flow}.
   For $x\in E^-\oplus E^0$, we have $\Phi_H(x)\le 0$ and
    hence  $V(x)=-x^-$, which implies that $\Upsilon_u(E^-\oplus E^0)=E^-\oplus E^0$ and
    $\Upsilon_u(E\setminus E^-\oplus E^0)=E\setminus E^-\oplus E^0$
    since $\Upsilon_u$ is a homeomorphism for each $u\in\mathbb{R}$. Therefore, $\Upsilon_u\in\Gamma_{n,k}$ for all $u\in\mathbb{R}$.

    Note that $V|_{\Phi_H^{-1}([\alpha,\infty])}=\nabla \Phi_H(x)$.
    For each $u\ge 0$ we have $\Upsilon_u(x)=\varphi_u(x)$ for any $x\in\Phi_H^{-1}([\alpha,\infty])$,
    and especially $\Upsilon_u(F)=\varphi_u(F)$,
    that is,  $(\Upsilon_u\circ\gamma)(S^+)=\varphi_u(F)$. Since $\Gamma_{n,k}$ is closed
    for the composition operation and
    $$
    \inf(\Phi_H|_{(\Upsilon_u\circ\gamma)(S^+)})=\inf(\Phi_H|_{\varphi_u(F)})\ge\inf(\Phi_H|_F)>0,
    $$
     we obtain $\varphi_u(F)\in \mathcal{F}_{n,k}$, that is,
        $\mathcal{F}_{n,k}$ is positively invariant under the flow $\varphi_u $ of $\nabla\Phi_H$.
  \end{proof}

Clearly, a $\mathbb{R}^{n,k}$-admissible $H\in C^2(\mathbb{R}^{2n},\mathbb{R}_{\ge 0})$
satisfies the conditions of Lemma~\ref{lem:PositiveInvariance}.

  \begin{theorem}\label{th:EH.1.6}
If  an $\mathbb{R}^{n,k}$-admissible $H\in C^2(\mathbb{R}^{2n},\mathbb{R}_{\ge 0})$
is nonresonant for $k=n$, and strong nonresonant for $k<n$,
then $c^{n,k}(H)$ is a positive critical value of $\Phi_H$.
\end{theorem}

The case of $k=n$ was proved in \cite[Section~II, Proposition~2]{EH89} (see also
\cite[Section~3.4, Proposition~1]{Sik90}).
It remains to prove the case $k<n$. By Lemma~\ref{lem:Lipsch},
the functional $\Phi_H$ is $C^{1,1}$ and its gradient $\nabla \Phi_H$
satisfies a global Lipschitz condition on $E$.
By a standard minimax argument, Theorem~\ref{th:EH.1.6} follows from
Lemma~\ref{lem:PositiveInvariance} and the following

\begin{lemma}\label{lem:PSmale}
  If  an $\mathbb{R}^{n,k}$-admissible $H\in C^1(\mathbb{R}^{2n},\mathbb{R}_{\ge 0})$
is  strong nonresonant,
 then each sequence $(x_j)\subset E$ with $\nabla\Phi_{H}(x_j)\to 0$
 has a convergent subsequence. In particular,  $\Phi_H$ satisfies the (PS) condition.
\end{lemma}

\begin{proof}
The functional $\mathfrak{b}$ is differentiable. Its gradient $\nabla \mathfrak{b}$ is compact and
satisfies a global Lipschitz condition on $E$.
   Since $\nabla\Phi_{H}(x)=x^+-x^--\nabla \mathfrak{b}(x)$ for any $x\in E$, we have
   \begin{equation}\label{e:critic}
   x_j^+-x_j^--\nabla \mathfrak{b}(x_j)\rightarrow 0.
   \end{equation}
   \noindent{\bf Case 1}. {\it $(x_j)$ is bounded in $E$}. Then $(x^0_j)$ is a bounded sequence in
   the space $\mathbb{R}^{n,k}$ of finite dimension.
    Hence $(x^0_j)$ has a convergent subsequence. Moreover, since $\nabla \mathfrak{b}$ is compact,
      $(\nabla \mathfrak{b}(x_j))$ has a convergent
   subsequence, and so both $(x_j^+)$ and $(x_j^-)$ have convergent subsequences in $E$.
   It follows that $(x_j)$ has a convergent subsequence.

   \noindent{\bf Case 2}. {\it $(x_j)$ is unbounded in $E$}.  Without loss of generality, we may assume
   $\lim_{j\rightarrow +\infty }\|x_j\|_E=+\infty$.
      For $z_0\in\mathbb{R}^{n,k}$  defined as in (H2), let
   $$
   y_j=\frac{x_j}{\|x_j\|_E}-\frac{1}{2a}z_0.
    $$
     Then $|y_j^0|\le\|y_j\|_E\le 1+|\frac{z_0}{2a}|$, and (\ref{e:critic}) implies
   \begin{equation}\label{e:critic+}
   y_j^+-y_j^--\jmath^{\ast}\left(\frac{ \nabla H(x_j)}{\|x_j\|_E}\right)\rightarrow 0.
   \end{equation}
   Also by (H2) there exist constants $C_1$ and $C_2$ such that
   $$
   \left\|\frac{ \nabla H(x_j)}{\|x_j\|_E}\right\|_{L^2}^2\leq \frac{8a^2 \|x_j\|_{L^2}^2+C_1 }{\|x_k\|_E^2}\leq C_2
   $$
  that is, $(\nabla H(x_j)/\|x_j\|_E)$ is bounded in $L^2$. Hence the sequence $\jmath^{\ast}\left( \nabla H(x_j)/\|x_j\|_E\right)$ is compact.
   (\ref{e:critic+}) implies that $(y_j)$ has a convergent subsequence in $E$. Without loss
   of generality, we may assume that  $y_{j}\rightarrow  y$ in $E$. Since
    (H2) implies that $H(z)=Q(z):=a|z|^2+ \langle z, z_0\rangle+ b$
        for $|z|$ sufficiently large, there exists a constant $C>0$ such that
   $|\nabla H(z)-\nabla Q(z)|\leq C$ for all $z\in\mathbb{R}^{2n}$.
      It follows that as $j\to\infty$,
   \begin{eqnarray*}
    \left\|\frac{\nabla H(x_j)}{\|x_j\|_{E}}-\nabla Q(y)\right\|_{L^2}
   &\leq&\left\|\frac{\nabla H(x_j)}{\|x_j\|_{E}}-\nabla Q(y_j)\right\|_{L^2}+\left\|\nabla Q(y_j)-\nabla Q(y)\right\|_{L^2}\\
   &\leq& \left\|\frac{\nabla H(x_j)-\nabla Q(x_j)}{\|x_j\|_E}\right\|_{L^2}+\frac{|z_0|}{\|x_j\|_E}+2a\|y_j-y\|_{L^2}\\
   &\leq &\frac{C}{\|x_j\|_E}+\frac{|z_0|}{\|x_j\|_E}+2a\|y_j-y\|_{L^2}
   \rightarrow 0.
   \end{eqnarray*}
   This implies that $\jmath^{\ast}\left(\nabla H(x_k)/\|x_k\|_{E}\right)$ tends to $\jmath^{\ast}(\nabla Q(y))$ in $E$, and thus we arrive at
   \begin{eqnarray*}
   y^+-y^--\jmath^{\ast}(\nabla Q (y))=0\quad\hbox{and}\quad \left\|y+\frac{z_0}{2a} \right\|_E=1.
   \end{eqnarray*}
  Then $y$ is smooth and satisfies
    \begin{eqnarray*}
   \dot{y}=J_{2n}\nabla Q(y)\quad\hbox{and}\quad y(1)\sim y(0),\; y(0), \, y(1)\in\mathbb{R}^{n,k}.
   \end{eqnarray*}
   Clearly $y(t)$ is given by
   $$
   y(t)+\frac{1}{2a}z_0=e^{2aJ_{2n}t}(y(0)+\frac{1}{2a}z_0).
   $$
   Since
   $\|y+\frac{1}{2a}z_0\|_{E}=1$
   implies that $y+\frac{1}{2a}z_0$
    is nonconstant, using the boundary condition satisfied by $y$ and the assumption that $z_0\in \mathbb{R}^{n,k}$,
    we deduce that $2a\in m\mathbb{N}\pi$.
     This gives rise to a contradiction because $H$ is strong non-resonant.
\end{proof}

Corresponding to \cite[Section~3.5, Lemma]{Sik90} we have

\begin{lemma}\label{lem:empty}
 Suppose that  $H:\mathbb{R}^{2n}\to\mathbb{R}$ is of class $C^{2n+2}$  and that
$\nabla H:\mathbb{R}^{2n}\to\mathbb{R}^{2n}$ satisfies a global Lipschitz condition.
  Then the set of critical values of $\Phi_{H}$ has empty interior in $\mathbb{R}$.
\end{lemma}

\begin{proof}
The method is similar to that of \cite[Lemma~3.5]{JinLu1917}. For clearness we give it in details.
By Lemma~\ref{lem:Lipsch}, $\Phi_H$ is $C^{1,1}$.
Lemma~\ref{lem:critic} implies that
 all critical points of $\Phi_{H}$ sit in $C^{2n+2}_{n,k}([0,1],\mathbb{R}^{2n})$.
 Thus  the restriction of $\Phi_{H}$ to $C^{1}_{n,k}([0,1],\mathbb{R}^{2n})$, denoted by $\hat\Phi_{H}$,
  and  $\Phi_{H}$ have the same critical value sets.
 As in the proof of \cite[Claim~4.4]{JinLu1916} we can deduce that $\hat\Phi_{H}$ is of class $C^{2n+1}$.

Let $P_0$ and $P_1$ be the orthogonal projections of $\mathbb{R}^{2n}$ to the spaces $V_0^{n,k}$
and $V_1^{2k}$ in (\ref{e:V0}) and (\ref{e:V1}), respectively. Take a smooth
$g:[0,1]\rightarrow [0,1]$
 such that $g$ equals $1$ (resp. $0$)
near $0$ (resp. $1$). Denote by $\phi^t$ the flow of $X_{H}$.
 Since $X_{H}$ is $C^{2n+1}$, we have a $C^{2n+1}$ map
$$
\psi:[0,1]\times\mathbb{R}^{n,k}\rightarrow \mathbb{R}^{2n},\;(t,z)\mapsto g(t)\phi^t(z)+(1-g(t))\phi^{t-1}(P_0\phi^1(z)+P_1z).
$$
For any $z\in\mathbb{R}^{n,k}$, since $\psi(0,z)=\phi^0(z)=z$ and $\psi(1,z)=P_0\phi^1(z)+P_1z$, we have
$$
\psi(1,z),\;\psi(0,z)\in\mathbb{R}^{n,k}\quad\hbox{and}\quad \psi(1,z)\sim\psi(0,z).
$$
These and \cite[Corollary~B.2]{JinLu1916} show that $\psi$ gives rise to a $C^{2n}$ map
$$
\Omega:\mathbb{R}^{n,k}\to C^1_{n,k}([0,1],\mathbb{R}^{2n}),\quad z\mapsto \psi(\cdot,z).
$$
Hence
$\Phi_{H}\circ\Omega: \mathbb{R}^{n,k}\to\mathbb{R}$
is of class $C^{2n}$. By Sard's Theorem we deduce that the critical value sets of $\Phi_{H}\circ\Omega$ is nowhere dense
(since $\dim \mathbb{R}^{n,k}<2n$).

Let $z\in\mathbb{R}^{n,k}$ be such that  $\phi^1(z)\in\mathbb{R}^{n,k}$ and $\phi^1(z)\sim z$.
Then $P_0\phi^1(z)-P_0z=\phi^1(z)-z$ and therefore
$P_0\phi^1(z)+P_1z=\phi^1(z)$,
which implies
$\psi(t,z)=\phi^t(z)\;\forall t\in [0,1]$.

For a critical point $y$ of $\Phi_H$, that is, $y\in C^{2n+2}_{n,k}([0,1],\mathbb{R}^{2n})$ and solves
$\dot{y}=J_{2n}\nabla H(y)=X_H(y)$,
%
with $z_y:=y(0)\in \mathbb{R}^{n,k}$ we have  $y(t)=\phi^t(z_y)\;\forall t\in [0,1]$, which implies that
 $\phi^1(z_y)\in\mathbb{R}^{n,k}$, $\phi^1(z_y)\sim z_y$ and therefore
 $y=\psi(\cdot,z_y)=\Omega(z_y)$. Hence $z_y$ is a critical point of $\Phi_{H}\circ\Omega$
 and $\Phi_{H}\circ\Omega(z_y)=\Phi_{H}(y)$. Thus
the critical value set of $\Phi_{H}$ is contained in that of $\Phi_{H}\circ\Omega$.
The desired claim is obtained.
\end{proof}

Having this lemma we can prove the following proposition, which
corresponds to Proposition~3 in \cite[Section~II]{EH89}.

\begin{proposition}\label{prop:Sinvariance}
 Let $H\in C^{2n+2}(\mathbb{R}^{2n},\mathbb{R}_{\ge 0})$ be
$\mathbb{R}^{n,k}$-admissible with $k<n$ and  strong nonresonant. Suppose that
$[0,1]\ni s\mapsto\psi_s$ is a smooth homotopy of the identity in ${\rm Symp}(\mathbb{R}^{2n},\omega_0)$ satisfying
\begin{equation}\label{e:Sinvariance}
\psi_s(\mathbb{R}^{n,k})=\mathbb{R}^{n,k},\quad
\psi_s(w+ V^{n,k}_0)=\psi_s(w)+ V^{n,k}_0\quad\forall w\in\mathbb{R}^{n,k}
\end{equation}
and
$$
\psi_s(z)=z+ w_s\quad\forall z\in \mathbb{R}^{2n}\setminus B^{2n}(0, R),
$$
where $R>0$ and $[0,1]\ni s\mapsto w_s$ is a smooth path in $\mathbb{R}^{n,k}$.
Then $s\mapsto c^{n,k}(H\circ\psi_s)$ is constant. Moreover, the same conclusion holds true if
all $\psi_s$ are replaced by translations $\mathbb{R}^{2n}\ni z\mapsto z+ w_s$, where
$[0,1]\ni s\mapsto w_s$ is a smooth path in $\mathbb{R}^{n,k}$.
In particular, $c^{n,k}(H(\cdot+ w))=c^{n,k}(H)$ for any $w\in\mathbb{R}^{n,k}$.
\end{proposition}

\begin{proof}
By the assumptions each $H\circ\psi_s$ is also $\mathbb{R}^{n,k}$-admissible  and  strong nonresonant.
Hence $c(H\circ\psi_s)$ is a positive critical value for each $s$.
Let $x\in E$ be a critical point of $\Phi_{H\circ\psi_s}$ with critical value $c(H\circ\psi_s)$. Then
$x\in C^{2n+2}_{n,k}([0,1],\mathbb{R}^{2n})$ and solves $\dot{x}=J_{2n}\nabla (H\circ\psi_s)(x)=X_{H\circ\psi_s}(x)$.
Let $y_s=\psi_s\circ x$. Then $y_s\in C^{2n+2}_{n,k}([0,1],\mathbb{R}^{2n})$ and satisfies
 $$
 \dot{y}_s(t)=(d\psi_s(x(t))\dot{x}(t)= (d\psi_s(x(t))X_{H\circ\psi_s}(x(t))= X_{H}(\psi_s(x(t))=J_{2n}\nabla H(y_s(t))
 $$
since $d\psi_s(z)X_H(z)=X_{H}(\psi_s(z))$ for any $z\in\mathbb{R}^{2n}$ by \cite[page 9]{HoZe94}.
Therefore $y_s$ is a critical point of $\Phi_H$ on $E$.
We claim that
\begin{equation}\label{e:two-action}
\Phi_H(y_s)=\Phi_{H\circ\psi_s}(x).
\end{equation}
Clearly, it suffices to prove the following equality:
\begin{equation}\label{e:two-action+}
A(y_s)=\frac{1}{2}\int^1_0 \langle -J_{2n}\dot{y}_s, y_s\rangle dt= \frac{1}{2}\int_0^1\langle -J_{2n}\dot{x},x\rangle dt=A(x).
\end{equation}
Extend $x$ into a piecewise $C^{2n+2}$-smooth loop $x^\ast:[0, 2]\to \mathbb{R}^{2n}$ by setting $x^\ast(t)=(2-t)x(1)+(t-1)x(0)$ for any $1\le t\le 2$.
We get a piecewise $C^{2n+2}$-smooth loop extending of $y_s$, $y_s^\ast=\psi_s(x^\ast)$. Clearly,
we can extend $x^\ast$ into a piecewise $C^{2n+2}$-smooth $u:D^2\to\mathbb{R}^{2n}$,
where $D^2$ is a closed disc bounded by $\partial D^2\equiv [0,2]/\{0,2\}$. Then $\psi_s\circ u:D^2\to\mathbb{R}^{2n}$
is  piecewise $C^{2n+2}$-smooth and $\psi_s\circ u|_{\partial D^2}=y_s^\ast$.  Stokes theorem yields
\begin{eqnarray*}
&&\frac{1}{2}\int_0^2\langle -J_{2n}\dot{x}^\ast,x^\ast\rangle dt=\int_{D^2}u^\ast\omega_0,\\
&&\frac{1}{2}\int^2_0 \langle -J_{2n}\dot{y}^\ast_s, y^\ast_s\rangle dt=\int_{D^2}(\psi_s\circ u)^\ast\omega_0=\int_{D^2}u^\ast\omega_0.
\end{eqnarray*}
Moreover, for any $t\in [1,2]$ we have $\dot{x}^\ast(t)=x(0)-x(1)\in V_0^{n,k}$ and
$x^\ast(t)\in\mathbb{R}^{n,k}$, and therefore  $\langle -J_{2n}\dot{x}^\ast(t), x^\ast(t)\rangle=0$
because $\mathbb{R}^{2n}$ has the orthogonal decomposition
$\mathbb{R}^{2n}=J_{2n}V^{n,k}_0\oplus \mathbb{R}^{n,k}$.
Then (\ref{e:two-action+}) follows from these.

Since $s\mapsto c^{n,k}(H\circ\psi_s)$ is continuous by Proposition~\ref{prop:coEHC.1},
and a critical point $x$ of $\Phi_{H\circ\psi_s}$ with critical value $c(H\circ\psi_s)$ yields
a critical point $y_s$ of $\Phi_H$ on $E$ satisfying (\ref{e:two-action}), we deduce
that each $c(H\circ\psi_s)$ is also a critical value of $\Phi_H$.
Lemma~\ref{lem:empty} shows that $s\mapsto c^{n,k}(H\circ\psi_s)$ must be constant.

Finally, let $\psi_s(z)=z+ w_s$. It is clear that
$H\circ\psi_s$ is  $\mathbb{R}^{n,k}$-admissible  and  strong nonresonant.
Thus $c(H\circ\psi_s)$ is a positive critical value.
If $x\in E$ is a critical point of $\Phi_{H\circ\psi_s}$ with critical value $c(H\circ\psi_s)$, then
 $y_s:=\psi_s\circ x$ is a critical point of $\Phi_H$ on $E$ and (\ref{e:two-action}) holds.
Hence $s\mapsto c^{n,k}(H(\cdot+ w_s))$ is constant.
\end{proof}

Let $\mathscr{F}_{n,k}(\mathbb{R}^{2n})=\{H\in C^0(\mathbb{R}^{2n},\mathbb{R}_{\ge 0})\,|\,H\;\hbox{satisfies (H2)}\}$.
For each bounded subset $B\subset\mathbb{R}^{2n}$ such that $\overline{B}\cap \mathbb{R}^{n,k}\neq \emptyset$,  we define
\begin{eqnarray}\label{e:fnkb}
\mathscr{F}_{n,k}(\mathbb{R}^{2n},B)=\{H\in \mathscr{F}_{n,k}(\mathbb{R}^{2n})\,|\,H\;\hbox{vanishes near $\overline{B}$}\}.
\end{eqnarray}

\begin{definition}
For each bounded subset $B\subset\mathbb{R}^{2n}$ such that $\overline{B}\cap \mathbb{R}^{n,k}\neq \emptyset$,
\begin{equation}\label{e:coCap}
c^{n,k}(B):=\inf\{c^{n,k}(H)\,|\, H\in \mathscr{F}_{n,k}(\mathbb{R}^{2n},B)\}\in [0, +\infty)
\end{equation}
is called the \textbf{coisotropic Ekeland-Hofer capacity} of $B$ (relative to $\mathbb{R}^{n,k}$).
For any unbounded subset $B\subset\mathbb{R}^{2n}$ such that $\overline{B}\cap \mathbb{R}^{n,k}\neq \emptyset$,
its \textbf{coisotropic Ekeland-Hofer capacity} is defined by
\begin{equation}\label{e:coCap+}
c^{n,k}(B)=\sup\{c^{n,k}(A)\,|\, A\subset B,\;\hbox{$A$ is bounded and $\overline{A}\cap\mathbb{R}^{n,k}\neq \emptyset$}\}.
\end{equation}
\end{definition}

\begin{remark}
{\rm When $k=n$ in the above definition, $c^{n,n}(B)$ is the (first) Ekeland-Hofer capacity of $B$.}
\end{remark}

For each bounded $B\subset\mathbb{R}^{2n}$ such that $\overline{B}\cap \mathbb{R}^{n,k}\neq \emptyset$, we write
\begin{eqnarray*}
&&\mathscr{E}_{n,k}(\mathbb{R}^{2n},B)=\{H\in \mathscr{F}_{n,k}(\mathbb{R}^{2n},B)\,|\,H\;\hbox{is strong nonresonant}\}\quad\hbox{if}\;k<n,\\
&&\mathscr{E}_{n,n}(\mathbb{R}^{2n},B)=\{H\in \mathscr{F}_{n,k}(\mathbb{R}^{2n},B)\,|\,H\;\hbox{is  nonresonant}\}.
\end{eqnarray*}
Clearly,  each $H\in \mathscr{E}_{n,k}(\mathbb{R}^{2n},B)$ satisfies (H1), and
$\mathscr{E}_{n,k}(\mathbb{R}^{2n},B)$ is a {\bf cofinal family} of $\mathscr{F}_{n,k}(\mathbb{R}^{2n},B)$,
that is, for any $H\in\mathscr{F}_{n,k}(\mathbb{R}^{2n},B)$ there exists
$G\in \mathscr{E}_{n,k}(\mathbb{R}^{2n},B)$ such that $G\ge H$.
Moreover, for each $l\in\mathbb{N}\cup\{\infty\}$ the smaller subset
     $\mathscr{E}_{n,k}(\mathbb{R}^{2n},B)\cap C^l(\mathbb{R}^{2n},\mathbb{R}_{\ge 0})$ is also
 a cofinal family of $\mathscr{F}_{n,k}(\mathbb{R}^{2n},B)$.
By the definition, we immediately get:

\begin{proposition}\label{prop:coEHC.2+}
\begin{description}
\item[(i)] $c^{n,k}(B)=c^{n,k}(\overline{B})$.
\item[(ii)]  $\mathscr{F}_{n,k}(\mathbb{R}^{2n},B)$ in  (\ref{e:coCap}) can be replaced by  any cofinal subset of it.
   \item[(iii)] Suppose that $\overline{B}\subset B^{2n}(R)$. For each $l\in\mathbb{N}\cup\{\infty\}$ let
   $\mathscr{E}^l_{n,k}(\mathbb{R}^{2n},B)$ consist of $H\in\mathscr{F}_{n,k}(\mathbb{R}^{2n},B)\cap C^l(\mathbb{R}^{2n},\mathbb{R}_{\ge 0})$
   for which  there exists $z_0\in\mathbb{R}^{n,k}$, real numbers $a, b$
such that $H(z)=a|z|^2+ \langle z, z_0\rangle+ b$ outside the closed ball $\overline{B^{2n}(R)}$,
where $a>\pi$ and $a\notin \pi\mathbb{N}$ for $k=n$, and $a>\pi/2$ and  $a\notin \pi\mathbb{N}/2$   for $0\le k<n$.
Then each $\mathscr{E}^l_{n,k}(\mathbb{R}^{2n},B)$ is a cofinal subset of $\mathscr{F}_{n,k}(\mathbb{R}^{2n},B)$.
   \end{description}
   \end{proposition}

\begin{proof}
We only prove (iii). By (ii) it suffices to prove that for each given
$H\in\mathscr{E}_{n,k}(\mathbb{R}^{2n},B)$ 
there exists
$G\in\mathscr{E}^l_{n,k}(\mathbb{R}^{2n},B)$ such that $G\ge H$.
We may assume that $H(z)=a|z|^2+ \langle z, z_0\rangle+ b$ outside a larger closed ball $\overline{B^{2n}(R_1)}$,
where $a>\pi$ and $a\notin \pi\mathbb{N}$ for $k=n$, and $a>\pi/2$ and  $a\notin \pi\mathbb{N}/2$   for $0\le k<n$.
Let $U_\epsilon(B)$ be the $\epsilon$-neighborhood of $B$. We can also assume that $H$ vanishes in
$U_{2\epsilon}(B)$. Since $\overline{B^{2n}(R_1)}$ is compact, we may find numbers $a'>a$, $b'$ such that
$a'\notin \pi\mathbb{N}$ for $k=n$,  $a'\notin \pi\mathbb{N}/2$   for $0\le k<n$, and
$a'|z|^2+ \langle z, z_0\rangle+ b'\ge H(z)$ for all $z\in\mathbb{R}^{2n}$.
Take a smooth function $f:\mathbb{R}^{2n}\to\mathbb{R}_{\ge 0}$ such that it equals to zero in
$U_{\epsilon}(B)$ and $1$ outside $U_{2\epsilon}(B)$. Define $G(z):=f(z)(a'|z|^2+ \langle z, z_0\rangle+ b')$
for $z\in\mathbb{R}^{2n}$. Then $G\ge H$ and
 $G\in\mathscr{E}^\infty_{n,k}(\mathbb{R}^{2n},B)$.
\end{proof}

\begin{remark}\label{rm:Sikorav}
{\rm Let $\mathscr{H}_{n,k}(\mathbb{R}^{2n},B)$ consist of $H\in C^\infty(\mathbb{R}^{2n},\mathbb{R}_{\ge 0})$
 which vanishes near $\overline{B}$ and for which  there exists $z_0\in\mathbb{R}^{n,k}$ and a real number $a$
such that $H(z)=a|z|^2$ outside a compact subset, where $a>\pi$ and $a\notin \pi\mathbb{N}$ for $k=n$,
and $a>\pi/2$ and  $a\notin \pi\mathbb{N}/2$   for $0\le k<n$.
As in the proof of Proposition~\ref{prop:coEHC.2} it is not hard to prove that
 $\mathscr{H}_{n,k}(\mathbb{R}^{2n},B)$ is a cofinal subset of $\mathscr{F}_{n,k}(\mathbb{R}^{2n},B)$.
When $k=n$ this shows that Sikorav's approach \cite{Sik90} to the Ekeland-Hofer capacity in \cite{EH89}
defines the same capacity.}
\end{remark}

%

\begin{proof}[Proof of Proposition~\ref{prop:coEHC.2}]
Proposition~\ref{prop:coEHC.1}(i)-(iii) lead to the first three claims.
Let us prove (iv). We may assume that $B$ is bounded.
By (\ref{e:coCap}) we have a sequence $(H_j)\subset \mathscr{F}_{n,k}(\mathbb{R}^{2n},B)$ such that
$c^{n,k}(H_j)\to  c^{n,k}(B)$. Note that $H_j(\cdot-w)\in \mathscr{F}_{n,k}(\mathbb{R}^{2n},B+w)$ for each $j$.
Hence
$$
c^{n,k}(B+w)\le\inf_j c^{n,k}(H_j(\cdot-w))= \inf_j c^{n,k}(H_j)=c^{n,k}(B)
$$
by the final claim in Proposition~\ref{prop:Sinvariance}. The same reasoning leads to
$c^{n,k}(B)=c^{n,k}(B+w+(-w))\le c^{n,k}(B+w)$ and so  $c^{n,k}(B+w)=c^{n,k}(B)$.
\end{proof}

\begin{proposition}[\hbox{\rm relative monotonicity}]\label{prop:coEHC.3}
Let subsets $A, B\subset \mathbb{R}^{2n}$ satisfy $\overline{A}\cap \mathbb{R}^{n,k}\neq \emptyset$
and $\overline{B}\cap \mathbb{R}^{n,k}\neq \emptyset$. If there exists
a smooth homotopy of the identity in ${\rm Symp}(\mathbb{R}^{2n},\omega_0)$
as in Proposition~\ref{prop:Sinvariance}, $[0,1]\ni s\mapsto\psi_s$,
 such that  $\psi_1(A)\subset B$, then
 $c^{n,k}(A)=c^{n,k}(\psi_s(A))$ for all $s\in [0,1]$, and in particular
   $c^{n,k}(A)\le c^{n,k}(B)$ by Proposition~\ref{prop:coEHC.2}(i).
\end{proposition}

\begin{proof}
Note that
$$
\mathscr{E}_{n,k}(\mathbb{R}^{2n}, A)\cap C^\infty(\mathbb{R}^{2n},\mathbb{R}_{\ge 0})\to
\mathscr{E}_{n,k}(\mathbb{R}^{2n}, \psi_s(A))\cap C^\infty(\mathbb{R}^{2n},\mathbb{R}_{\ge 0}),\;H\mapsto H\circ\psi_s^{-1}
$$
is a one-to-one correspondence. Then
\begin{eqnarray*}
c^{n,k}(\psi_s(A))&=&\inf\{c^{n,k}(G)\,|\, G\in \mathscr{E}_{n,k}(\mathbb{R}^{2n}, \psi_s(A))\cap C^\infty(\mathbb{R}^{2n},\mathbb{R}_{\ge 0})\}\\
&=&\inf\{c^{n,k}(H\circ\psi_s^{-1})\,|\, H\in \mathscr{E}_{n,k}(\mathbb{R}^{2n}, A)\cap C^\infty(\mathbb{R}^{2n},\mathbb{R}_{\ge 0})\}\\
&=&\inf\{c^{n,k}(H)\,|\, H\in \mathscr{E}_{n,k}(\mathbb{R}^{2n}, A)\cap C^\infty(\mathbb{R}^{2n},\mathbb{R}_{\ge 0})\}=c^{n,k}(A).
\end{eqnarray*}
Here the third equality comes from Proposition~\ref{prop:Sinvariance}.
\end{proof}

\begin{proof}[Proof of Theorem~\ref{th:coEHC.4}]
We may assume that $B$ is bounded, and complete the proof in two steps.

{\bf Step 1}. {\it Prove $c^{n,k}(\Phi(B))=c^{n,k}(B)$ for every $\Phi\in{\rm Sp}(2n,k)$.}
Take a smooth path $[0, 1]\ni t\mapsto\Phi_t\in {\rm Sp}(2n,k)$ such that $\Phi_0=I_{2n}$ and $\Phi_1=\Phi$.
We have a smooth function $[0,1]\times \mathbb{R}^{2n}\ni (t,z)\mapsto G_t(z)\in\mathbb{R}^{2n}$ such that
 the path $\Phi_t$ is generated by $X_{G_t}$ and that $G_t(z)=0\;\forall z\in\mathbb{R}^{n,k}$ (see Step 2 below).
 Since $\cup_{t\in [0, 1]}\Phi_t(\overline{B})$ is compact, it can be contained in
 a ball $B^{2n}(0, R)$ for some  $R>0$.
  Take a smooth cut function $\rho:\mathbb{R}^{2n}\to [0, 1]$
  such that $\rho=1$ on $B^{2n}(0, 2R)$ and $\rho=0$ outside $B^{2n}(0, 3R)$.
 Define a smooth function $\tilde{G}:[0,1]\times \mathbb{R}^{2n}\to \mathbb{R}$ by $\tilde{G}(t,z)=\rho(z)G_t(z)$
 for $(t,z)\in [0,1]\times\mathbb{R}^{2n}$. Denote by $\psi_t$ the Hamiltonian path generated by $\tilde{G}$ in
${\rm Ham}^c(\mathbb{R}^{2n},\omega_0)$. Then $\psi_t(z)=\Phi_t(z)$ for all $(t,z)\in [0,1]\times B^{2n}(0, R)$.
Moreover each $\psi_t$ restricts to the identity on $\mathbb{R}^{n,k}$
 because $\tilde{G}(t,z)=\rho(z)G_t(z)=0$ for all $(t,z)\in [0,1]\times\mathbb{R}^{n,k}$.
 Hence we obtain $c^{n,k}(\Phi(B))=c^{n,k}(\Phi_1(B))=c^{n,k}(B)$ by Proposition~\ref{prop:coEHC.3}.

{\bf Step 2}. {\it Prove $c^{n,k}(\phi(B))=c^{n,k}(B)$ in case $w_0=0$.}
Let $\Phi=(d\phi(0))^{-1}$. Since $c^{n,k}(\Phi\circ\phi(B))=c^{n,k}(\phi(B))$ by Step 1,
and $\Phi\circ\phi(w)=w\;\forall w\in\mathbb{R}^{n,k}$,
replacing $\Phi\circ\phi$ by $\phi$,  we may assume $d\phi(0)={\rm id}_{\mathbb{R}^{2n}}$.
Define a continuous path in ${\rm Symp}(\mathbb{R}^{2n},\omega_0)$,
\begin{equation}\label{e:coEHC.5}
\varphi_t(z)=\left\{\begin{array}{ll}
z& \hbox{if}\;t\le 0,\\
\frac{1}{t}\phi(tz) &\hbox{if}\;t> 0,
\end{array}\right.
\end{equation}
which is smooth except possibly at $t=0$. As in \cite[Proposition A.1]{Schl05} we can smoothen it with a smooth function
$\eta:\mathbb{R}\to\mathbb{R}$ defined by
\begin{equation}\label{e:coEHC.6}
\eta(t)=\left\{\begin{array}{ll}
0& \hbox{if}\;t\le 0,\\
e^2e^{-2/t} &\hbox{if}\;t> 0,
\end{array}\right.
\end{equation}
where $e$ is the Euler number. Namely,  defining $\phi_t(z):=\varphi_{\eta(t)}(z)$ for $z\in\mathbb{R}^{2n}$ and $t\in\mathbb{R}$,
we get a smooth path $\mathbb{R}\ni t\mapsto\phi_t\in {\rm Symp}(\mathbb{R}^{2n},\omega_0)$ such that
 \begin{equation}\label{e:coEHC.7}
\phi_0={\rm id}_{\mathbb{R}^{2n}},\quad \phi_1=\phi,\quad \phi_t(z)=z,\;\forall z\in\mathbb{R}^{n,k},\;\forall t\in\mathbb{R}.
\end{equation}
 Define $X_t(z)=\left(\frac{d}{dt}\phi_t\right)(\phi_t^{-1}(z))$ and
 \begin{equation}\label{e:coEHC.8}
 H_t(z)=\int^z_0 i_{X_t}\omega_0,
 \end{equation}
 where the integral is along any piecewise smooth curve from $0$ to $z$ in $\mathbb{R}^{2n}$.
 Then $\mathbb{R}\times \mathbb{R}^{2n}\ni (t,z)\mapsto H_t(z)\in\mathbb{R}$ is smooth and $X_t=X_{H_t}$.
 By the final condition in (\ref{e:coEHC.7}), for each $(t,z)\in\mathbb{R}\times\mathbb{R}^{n,k}$ we have $X_t(z)=0$
 and therefore $H_t(z)=0$.
As in Step~1, we can assume that $\cup_{t\in [0, 1]}\phi_t(\overline{B})$ is contained a ball
$B^{2n}(0, R)$. Take a smooth cut function $\rho:\mathbb{R}^{2n}\to [0, 1]$ as above, and define
a smooth function $\tilde{H}:[0,1]\times \mathbb{R}^{2n}\to \mathbb{R}$ by $\tilde{H}(t,z)=\rho(z)H_t(z)$
 for $(t,z)\in [0,1]\times\mathbb{R}^{2n}$. Then the Hamiltonian path $\psi_t$ generated by $\tilde{H}$ in
${\rm Ham}^c(\mathbb{R}^{2n},\omega_0)$ satisfies
$$
\psi_t(z)=\phi_t(z), \;\forall (t,z)\in [0,1]\times B^{2n}(0, R)\quad\hbox{and}\quad
\psi_t(z)=z, \;\forall (t,z)\in [0,1]\times \mathbb{R}^{n,k}.
$$
It follows from Proposition~\ref{prop:coEHC.3} that
 $c^{n,k}(\phi(B))=c^{n,k}(\psi_1(B))=c^{n,k}(B)$ as above.

 {\bf Step 3}. {\it Prove $c^{n,k}(\phi(B))=c^{n,k}(B)$ in case $w_0\ne 0$.}
 Define $\varphi(w)=\phi(w+w_0)$ for $w\in \mathbb{R}^{2n}$.
Then $d\varphi(0)=d\phi(w_0)\in {\rm Sp}(2n,k)$ and $\varphi(w)=\phi(w+w_0)=w\;\forall w\in\mathbb{R}^{n,k}$.
By Step 2 we arrive at $c^{n,k}(\varphi(B-w_0))=c^{n,k}(B-w_0)$.
The desired equality follows because $\phi(B)=\varphi(B-w_0)$
and $c^{n,k}(B-w_0)= c^{n,k}(B)$ by Proposition~\ref{prop:coEHC.2}.
\end{proof}



\begin{proof}[Proof of Corollary~\ref{cor:coEHC.5}]
As discussed above the proof is reduced to the case $w_0=0$.
Moreover we can assume that both sets $A$ and $U$ are bounded and that $U$ is also starshaped with respect to
the origan $0\in\mathbb{R}^{2n}$.

Next the proof can be completed following  \cite[Proposition A.1]{Schl05}.
Now $[0,1]\ni t\mapsto\phi_t(\cdot):=\varphi_{\eta(t)}(\cdot)$ given by
(\ref{e:coEHC.5}) and (\ref{e:coEHC.6}) is a smooth path of symplectic embeddings from $U$ to $\mathbb{R}^{2n}$ with properties
 \begin{equation}\label{e:coEHC.9}
\phi_0={\rm id}_U,\quad \phi_1=\varphi,\quad \phi_t(z)=z,\;\forall z\in\mathbb{R}^{n,k}\cap U,\;\forall t\in\mathbb{R}.
\end{equation}
Thus $X_t(z):=\left(\frac{d}{dt}\phi_t\right)(\phi_t^{-1}(z))$ is a symplectic vector field defined on $\phi_t(U)$,
and (\ref{e:coEHC.8}) (where the integral is along any piecewise smooth curve from $0$ to $z$ in $\phi_t(U)$)
defines a smooth function $H_t$ on $\phi_t(U)$ in the present case. Obverse that
$H:\cup_{t\in [0,1]}(\{t\}\times\phi_t(U))\to \mathbb{R}$ defined by $H(t,z)=H_t(z)$ is smooth and generates the path $\phi_t$.
Since $K=\cup_{t\in [0,1]}\{t\}\times\phi_t(\overline{A})$ is a compact subset in $[0,1]\times\mathbb{R}^{2n}$ we can choose
a bounded and relative open neighborhood $W$ of $K$ in $[0,1]\times\mathbb{R}^{2n}$ such that $W\subset \cup_{t\in [0,1]}(\{t\}\times\phi_t(U))$.
Take a smooth cut function $\chi:[0,1]\times\mathbb{R}^{2n}\to\mathbb{R}$ such that $\chi|_K=1$ and $\chi$ vanishes outside $W$.
 Define $\hat{H}:[0,1]\times\mathbb{R}^{2n}\to\mathbb{R}$ by $\hat{H}(t,z)=\chi(t,z)H(t,z)$. It generates
 a smooth homotopy $\psi_t$ ($t\in [0,1]$) of the identity in ${\rm Ham}^c(\mathbb{R}^{2n},\omega_0)$ such that
 $\psi_t(z)=\phi_t(z)$ for all $(t,z)\in [0,1]\times A$.
Moreover, the final condition in (\ref{e:coEHC.7}) implies that $\mathbb{R}^{n,k}\cap U\subset\phi_t(U)$ and
$X_t(z)=0$ for any $t\in [0,1]$ and $z\in \mathbb{R}^{n,k}\cap U$. Hence
for any $(t,z)\in [0,1]\times\mathbb{R}^{n,k}$ we have $\hat{H}(t,z)=\chi(t,z)H(t,z)=0$
and so $\psi_t(z)=z$. Then Proposition~\ref{prop:coEHC.3} leads to
$c^{n,k}(A)= c^{n,k}(\psi_1(A))=c^{n,k}(\phi_1(A))=c^{n,k}(\varphi(A))$.
\end{proof}

\section{Proof of Theorem~\ref{th:EHconvex}}\label{sec:formula}
\setcounter{equation}{0}

The case of $k=n$ was proved in \cite{EH89, EH90, Sik90}. We assume $k<n$ below.
 By Proposition~\ref{prop:coEHC.2}(iv),
 $c^{n,k}(D)=c^{n,k}(D+w)$ for any $w\in\mathbb{R}^{n,k}$.
 Moreover, for each $x\in C^{1}_{n,k}([0,1])$ there holds
$$
A(x)=\frac{1}{2}\int_0^1\langle -J_{2n}\dot{x},x\rangle dt=\frac{1}{2}\int_0^1\langle -J_{2n}\dot{x},x+w\rangle dt=A(x+w),\quad\forall w\in\mathbb{R}^{n,k}.
$$
Recalling that $D\cap\mathbb{R}^{n,k}\ne\emptyset$, we may assume that $D$  contains the origin  $0$ below.

Let $j_D$ be the Minkowski functional associated to $D$, $H:=j_D^2$ and $H^{\ast}$ be the Legendre transform of $H$.
Then    $\partial D=H^{-1}(1)$, and
 there exists a constant $R\geq 1$ such that
\begin{eqnarray}\label{e:dualestimate}
 \frac{|z|^2}{R}\leq H(z)\leq R|z|^2\quad\hbox{and so}\quad
\frac{|z|^2}{4R}\leq H^{\ast}(z)\leq \frac{R}{4}|z|^2
\end{eqnarray}
for all $z\in\mathbb{R}^{2n}$.
 Moreover $H$ is $C^{1,1}$ with uniformly Lipschitz constant.

By  \cite[Theorem~1.5]{JinLu1917}
$$
\Sigma^{n,k}_{\partial D}:=\{A(x)>0\;|\;x\;\hbox{is a leafwise chord on\;$\partial D$\;for \;$\mathbb{R}^{n,k}$}\}
$$
contains a minimum number $\varrho$, that is,  there exists a leafwise chord  $x^\ast$  on $\partial D$ for $\mathbb{R}^{n,k}$
such that
$A(x^\ast)=\min\Sigma^{n,k}_{\partial D}=\varrho$.
Actually, the arguments therein shows that  there exists $w\in C^1_{n,k}([0,1])$ such that
\begin{equation}\label{e:clarke}
A(w)=1\quad\hbox{and}\quad
I(w):=\int_0^1H^\ast(-J_{2n}\dot{w})dt=A(x^\ast)=\varrho.
\end{equation}


Let us prove  (\ref{e:fixpt.1}) and (\ref{e:fixpt.2}) by the following two steps.
As done in  \cite{JinLu1915, JinLu1916} (see also Step 4 below),
by approximating arguments we can assume that $\partial D$ is smooth and strictly convex.
In this case $\Sigma^{n,k}_{\partial D}$
has no interior points in $\mathbb{R}$ because of  \cite[Lemma~3.5]{JinLu1917}, and
we give a complete proof though the ideas which are similar to those of
the proof of \cite{Sik90}  (and \cite[Theorem~1.11]{JinLu1915} and \cite[Theorem~1.17]{JinLu1916}).

\noindent
{\bf Step 1.} {\it Prove that $c^{n,k}(D)\ge \varrho$.}
By the monotonicity of $c^{n,k}$ it suffices to prove $c^{n,k}(\partial D)\ge \varrho$.
 For a given $\epsilon>0$, consider a cofinal family of $\mathcal{F}_{n,k}(\mathbb{R}^{2n},\partial D)$,
\begin{eqnarray}\label{e:EH.3.7}
\mathscr{E}^{n,k}_\epsilon(\mathbb{R}^{2n},\partial D)
\end{eqnarray}
consisting of $\overline{H}=f\circ H$, where $f\in C^\infty(\mathbb{R},\mathbb{R}_{\ge 0})$ satisfies
\begin{eqnarray}\label{e:EH.3.8}
\left.\begin{array}{ll}
&f(s)=0\;\hbox{for $s$ near $1\in\mathbb{R}$},\\
&f'(s)\le 0\;\forall s\le 1,\quad  f'(s)\ge 0\;\forall\;s\ge 1,\\
& f'(s)=\alpha\in\mathbb{R}\setminus\Sigma^{n,k}_{\partial D}
\;\hbox{if}\;f(s)\ge\epsilon\;\&\;s>1
\end{array}\right\}
\end{eqnarray}
and where  $\alpha$ is required to satisfy for some constant $C>0$
\begin{eqnarray}\label{e:EH.3.8+}
\alpha H(z)\ge \frac{\pi}{2}|z|^2-C\quad\hbox{for}\, |z|\,\hbox{sufficiently large}
\end{eqnarray}
because of (\ref{e:dualestimate}) and ${\rm Int}(\Sigma^{n,k}_{\partial D})=\emptyset$.

Then each $\overline{H}\in\mathscr{E}^{n,k}_\epsilon(\mathbb{R}^{2n},\partial D)$
satisfies all conditions in Lemma~\ref{lem:PositiveInvariance}. Indeed, it
 belongs to $C^\infty(\mathbb{R}^{2n},\mathbb{R}_{\ge 0})$, restricts to zero near $\partial D$ and thus satisfies (H1).
Note that $f(s)=\alpha s+ \epsilon-\alpha s_0$ for $s\ge s_0$, where $s_0=\inf\{s>1\,|\, f(s)\ge\epsilon\}$.
(\ref{e:EH.3.8+}) implies that $\overline{H}(z)\ge\frac{\pi}{2}|z|^2-C'\;\forall z\in\mathbb{R}^{2n}$
for some constant $C'>0$, and therefore  $c^{n,k}(\overline{H})<+\infty$ by the arguments above Proposition~\ref{prop:coEHC.1}.
Moreover, it is clear that  $\mathbb{R}^{n,k}\cap{\rm Int}(\overline{H}^{-1}(0))\ne\emptyset$ and
$|\overline{H}_{zz}(z)|$ is bounded on $\mathbb{R}^{2n}$. Then
 (\ref{e:EH.2.1+}) is  satisfied with any $z_0\in\mathbb{R}^{n,k}\cap{\rm Int}(\overline{H}^{-1}(0))$
by the arguments at the end of proof of Proposition~\ref{prop:EH.1.4}.
Hence $c^{n,k}(\overline{H})>0$.

By combining proofs of Lemma~\ref{lem:PSmale}
and \cite[Lemma~3.7]{JinLu1917} we can obtain the first claim of the following.

\begin{lemma}\label{lem:EH8}
For every $\overline{H}\in\mathscr{E}^{n,k}_\epsilon(\mathbb{R}^{2n},\partial D)$, $\Phi_{\overline{H}}$ satisfies the
$(PS)$ condition and hence $c^{n,k}(\overline{H})$ is a positive critical value of $\Phi_{\overline{H}}$.
\end{lemma}

\begin{lemma}\label{lem:EH9}
  For every $\overline{H}\in\mathscr{E}^{n,k}_\epsilon(\mathbb{R}^{2n},\partial D)$, any positive critical value $c$ of
$\Phi_{\overline{H}}$ is greater than $\min\Sigma^{n,k}_{\partial D}-\epsilon$.
  In particular, $c^{n,k}(\overline{H})>\min\Sigma^{n,k}_{\partial D}-\epsilon$.
\end{lemma}

\begin{proof}
   For a critical point $x$ of $\Phi_{\overline{H}}$ with positive critical values there holds
$$
-J\dot{x}(t)=\nabla \overline{H}(x(t))=f'(H(x(t)))\nabla H(x(t)),\quad x(1)\sim x(0),\quad x(1), x(0)\in\mathbb{R}^{n,k}
$$
and $H(x(t))\equiv c_0$ (a positive constant). Since
\begin{eqnarray*}
0<\Phi_{\overline{H}}(x)&=&\frac{1}{2}\int^1_0\langle J_{2n}x(t),\dot{x}(t)\rangle dt-\int^1_0\overline{H}(x(t))dt\\
&=&\frac{1}{2}\int^1_0\langle x(t),f'(c_0)\nabla H(x(t))\rangle dt-\int^1_0f(s)dt\\
&=&f'(c_0)c_0-f(c_0),
\end{eqnarray*}
 we deduce $\beta:=f'(c_0)>0$, and so $c_0>1$. Define
$y(t)=\frac{1}{\sqrt{c_0}}x(t/\beta)$ for $0\le t\le\beta$.
Then
$$
H(y(t))=1,\quad -J_{2n}\dot{y}=\nabla H(y(t)),\quad y(\beta)\sim y(0),\quad y(\beta), y(0)\in\mathbb{R}^{n,k}
$$
and therefore  $f'(c_0)=\beta=A(y)\in\Sigma^{n,k}_{\partial D}$. By the definition of $f$ this implies
 $f(c_0)<\epsilon$ and so
$$
\Phi_{\overline{H}}(x)=f'(c_0)c_0-f(c_0)> f'(c_0)-\epsilon\ge \min\Sigma^{n,k}_{\partial D}-\epsilon.
$$
\end{proof}

Since for any $\epsilon>0$ and $G\in \mathcal{F}_{n,k}(\mathbb{R}^{2n},\partial D)$, there exists
$\overline{H}\in\mathscr{E}^{n,k}_\epsilon(\mathbb{R}^{2n},\partial D)$ such that $\overline{H}\ge G$,
we deduce that $c^{n,k}(G)\ge c^{n,k}(\overline{H})\ge\min\Sigma^{n,k}_{\partial D}-\epsilon$. Hence
$c^{n,k}(\partial D)\ge \min\Sigma^{n,k}_{\partial D}=\varrho$.

\noindent
{\bf Step 2.} {\it Prove that $c^{n,k}(D)\le \varrho$.}
 Denote by $w^\ast$
 the projections of $w$ in (\ref{e:clarke})
 onto $E^\ast$ (according to the decomposition $E = E^{1/2}
=E^+\oplus E^-\oplus E^0$), $\ast=0,-,+$.  Then $w^+\ne 0$.
 (Otherwise, a contradiction occurs because
$1=A(w) = A(w^0\oplus w^-) =-\frac{1}{2}\|w^-\|^2$.)
Define $y:=w/\sqrt{\varrho}$. Then $y\in C^{1}_{n,k}([0,1])$ satisfies
$I(y)=1$ and $A(y)=\frac{1}{\varrho}$.
It follows from the definition of $H^\ast$ that for any $\lambda\in\mathbb{R}$ and $x\in E$,
\begin{eqnarray*}
\lambda^2=I(\lambda y)&=&\int^1_0H^\ast(-\lambda J_{2n}\dot{y}(t))dt
\ge\int^1_0\{\langle x(t), -\lambda J_{2n}\dot{y}(t)\rangle- H(x(t))\}dt
\end{eqnarray*}
and so
\begin{eqnarray*}
\int^1_0H(x(t))dt\ge \int^1_0\langle x(t), -\lambda J_{2n}\dot{y}(t)\rangle dt-\lambda^2
=\lambda \int^1_0\langle x(t), - J_{2n}\dot{y}(t)\rangle dt-\lambda^2.
\end{eqnarray*}
In particular, taking $\lambda=\frac{1}{2} \int^1_0\langle x(t), - J_{2n}\dot{y}(t)\rangle dt$
we arrive at
\begin{eqnarray}\label{e:EH.3.1}
\int^1_0H(x(t))dt\ge\left(\frac{1}{2} \int^1_0\langle x(t), - J_{2n}\dot{y}(t)\rangle dt\right)^2,\quad
\forall x\in E.
\end{eqnarray}

Since $y^+=w^+/\sqrt{\varrho}\ne 0$ and $E^-\oplus E^0+\mathbb{R}_+y=E^-\oplus E^0\oplus\mathbb{R}_+y^+$, by  Proposition~\ref{prop:EH.1.1}(ii),
$$
\gamma(S^+)\cap(E^-\oplus E^0+\mathbb{R}_+y)\ne\emptyset,\quad\forall\gamma\in\Gamma_{n,k}.
$$
Fixing  $\gamma\in\Gamma_{n,k}$ and $x\in \gamma(S^+)\cap(E^-\oplus E^0+\mathbb{R}_+y)$,
 write $x=x^{-0}+ sy=x^{-0}+ sy^{-0}+ sy^+$ where $x^{-0}\in E^-\oplus E^0$, and
consider the polynomial
$$
P(t)=\mathfrak{a}(x+ty)=\mathfrak{a}(x)+ t\int^1_0\langle x, - J_{2n}\dot{y}\rangle dt + \mathfrak{a}(y)t^2=\mathfrak{a}(x^{-0}+(t+s)y).
$$
 Since $\mathfrak{a}|_{E^-\oplus E^0}\le 0$
implies $P(-s)\le 0$, and $\mathfrak{a}(y)=1/\varrho>0$ implies
$P(t)\to+\infty$ as $|t|\to+\infty$,
 there exists $t_0\in\mathbb{R}$ such that $P(t_0)=0$. It follows that
$$
\left(\int^1_0\langle x, - J_{2n}\dot{y}\rangle dt\right)^2\ge 4\mathfrak{a}(y)\mathfrak{a}(x).
$$
This and (\ref{e:EH.3.1}) lead to
\begin{eqnarray}\label{e:EH.3.2}
\mathfrak{a}(x)&\le& (\mathfrak{a}(y))^{-1}\left(\frac{1}{2}\int^1_0\langle x, - J_{2n}\dot{y}\rangle dt\right)^2
\le\varrho\int^1_0H(x(t))dt.
\end{eqnarray}

 In order to prove that that $c^{n,k}(D)\le \varrho$, it suffices to prove that for any $\varepsilon>0$ there exists  $\tilde{H}\in  \mathscr{F}_{n,k}(\mathbb{R}^{2n}, D)$
such that $c^{n,k}(\tilde{H})< \varrho+\varepsilon$,
which is reduced to prove: \textsf{for any given $\gamma\in\Gamma_{n,k}$ there exists $x\in h(S^+)$ such that}
\begin{eqnarray}\label{e:EH.3.10}
\Phi_{\tilde{H}}(x)< \varrho+\varepsilon.
\end{eqnarray}

Now for $\tau>0$ there exists $H_\tau\in \mathscr{F}_{n,k}(\mathbb{R}^{2n}, D)$
such that
\begin{eqnarray}\label{e:EH.3.11}
H_\tau\ge \tau\left(H-\left(1+\frac{\varepsilon}{2\varrho}\right)\right).
\end{eqnarray}
For $\gamma\in\Gamma_{n,k}$ choose $x\in h(S^+)$ satisfying (\ref{e:EH.3.2}). We shall prove that for $\tau>0$
large enough $\tilde{H}=H_\tau$ satisfies the requirements.

$\bullet$ If $\int^1_0H(x(t))dt\le\left(1+\frac{\varepsilon}{\varrho}\right)$, then by $H_\tau\ge 0$ and (\ref{e:EH.3.2}), we have
$$
\Phi_{H_\tau}(x)\le \mathfrak{a}(x)\le \varrho\int^1_0H(x(t))dt\le \varrho\left(1+\frac{\varepsilon}{\varrho}\right)<\varrho+\varepsilon.
$$

$\bullet$ If $\int^1_0H(x(t))dt>\left(1+\frac{\varepsilon}{\varrho}\right)$, then (\ref{e:EH.3.11}) implies
\begin{eqnarray}\label{e:EH.3.12}
\int^1_0H_\tau(x(t))dt&\ge& \tau\left(\int^1_0H(x(t))dt-\left(1+\frac{\varepsilon}{2\varrho}\right)\right)\nonumber\\
&\ge&\tau \frac{\varepsilon}{2a}\left(1+\frac{\varepsilon}{\varrho}\right)^{-1}\int^1_0H(x(t))dt
\end{eqnarray}
because
$$
\left(1+\frac{\varepsilon}{2\varrho}\right)=
\left(1+\frac{\varepsilon}{2\varrho}\right)\left(1+\frac{\varepsilon}{\varrho}\right)^{-1}
\left(1+\frac{\varepsilon}{\varrho}\right)<
\left(1+\frac{\varepsilon}{2\varrho}\right)\left(1+\frac{\varepsilon}{\varrho}\right)^{-1}\int^1_0H(x(t))dt
$$
and
$$
1-\left(1+\frac{\varepsilon}{2\varrho}\right)\left(1+\frac{\varepsilon}{\varrho}\right)^{-1}=
\left(1+\frac{\varepsilon}{\varrho}\right)^{-1}\left[\left(1+\frac{\varepsilon}{\varrho}\right)-
\left(1+\frac{\varepsilon}{2\varrho}\right)\right]=
\frac{\varepsilon}{2\varrho}\left(1+\frac{\varepsilon}{\varrho}\right)^{-1}.
$$
Choose $\tau>0$ so large that the right side of the last equality is more than $\varrho$.
Then
$$
\int^1_0H_\tau(x(t))dt\ge \varrho\int^1_0H(x(t))dt
$$
by (\ref{e:EH.3.12}),
and hence (\ref{e:EH.3.2}) leads to
$$
\Phi_{H_\tau}(x)=\mathfrak{a}(x)-\int^1_0H_\tau(x(t))dt\le \mathfrak{a}(x)-\varrho\int^1_0H(x(t))dt\le 0.
$$

In summary, in the above two cases we have $\Phi_{H_\tau}(x)<\varrho+\varepsilon$.
(\ref{e:EH.3.10}) is proved.


\noindent
{\bf Step 3.} {\it Prove the final claim.}
By \cite[Theorem~1.5]{JinLu1917} we have
$$
 c_{\rm LR}(D,D\cap \mathbb{R}^{n,k})=\min\{A(x)>0\;|\;x\;\hbox{is a leafwise chord on\;$\partial D$\;for \;$\mathbb{R}^{n,k}$}\}.
$$
Using Proposition~1.12 and Corollary~2.41 in \cite{Kr15}
we can choose two sequences of $C^\infty$ strictly convex domains with  boundaries,
 $(D^+_j)$ and $(D^-_j)$, such that
 \begin{description}
 \item[(i)] $D^-_1\subset D^-_2\subset\cdots\subset D$ and $\cup^\infty_{j=1}D^-_j=D$,
 \item[(ii)] $D^+_1\supseteq D^+_2\supseteq\cdots\supseteq D$ and $\cap^\infty_{j=1}D^+_j=D$,
 \item[(iii)] for any small neighborhood $O$ of $\partial D$ there exists an integer
 $N>0$ such that $\partial D^+_k\cup\partial D^-_k\subset O\;\forall k\ge N$.
 \end{description}
Now Step 1-Step 2 and \cite[Theorem~1.5]{JinLu1917} give rise to
$c_{\rm LR}(D^+_j,D\cap \mathbb{R}^{n,k})=c^{n,k}(D^+_j)$ and $c_{\rm LR}(D^-_j,D\cap \mathbb{R}^{n,k})=c^{n,k}(D^-_j)$
for each $j=1,2,\cdots$. We have also that the sequence
$c_{\rm LR}(D^+_j,D\cap \mathbb{R}^{n,k})$  converges decreasingly to $c_{\rm LR}(D,D\cap \mathbb{R}^{n,k})$ as $j\to\infty$
and that the sequence $c_{\rm LR}(D^-_j,D\cap \mathbb{R}^{n,k})$  converges increasingly to $c_{\rm LR}(D,D\cap \mathbb{R}^{n,k})$ as $j\to\infty$.
Moreover for each $j$ there holds $c^{n,k}(D^-_j)\le c^{n,k}(D)\le c^{n,k}(D^+_j)$ by the monotonicity of $c^{n,k}$.
These lead to $c^{n,k}(D)=c_{\rm LR}(D,D\cap \mathbb{R}^{n,k})$.


\section{Proof of Theorem~\ref{th:EHproduct}}\label{sec:product}
\setcounter{equation}{0}

Clearly, the proof of Theorem~\ref{th:EHproduct}
 can be reduced to the case that $m=2$ and all $D_i$ are also bounded.
 Moreover, by an approximation argument in Step 3 of Section~\ref{sec:formula}
we only need to prove the  following:

\begin{theorem}\label{th:EHproduct1}
 For  bounded strictly convex domains $D_i\subset\mathbb{R}^{2n_i}$ with $C^2$-smooth boundary and containing
 the origin, $i=1,2$,
 and any integer  $0\le k\le n:=n_1+n_2$
 it holds that
 \begin{eqnarray*}
 &&c^{n,k}(\partial D_1\times\partial D_2)=c^{n,k}(D_1\times D_2)\\
 &=&\min\{c^{n_1,\min\{n_1,k\}}(D_1),c^{n_2,\max\{k-n_1,0\}}(D_2)\}.
 \end{eqnarray*}
 \end{theorem}


We first prove two lemmas. For convenience we write $E=H^{\frac{1}{2}}_{n,k}$ as
$E_{n,k}$, and $E^\ast$ as $E^\ast_{n,k}$, $\ast=+,-, 0$. As a generalization of
Lemma 2 in \cite[\S~6.6]{Sik90} we have:

\begin{lemma}\label{lem:9.2}
Let $D\subset\mathbb{R}^{2n}$ be a bounded strictly convex domain
  with $C^2$-smooth boundary and containing $0$. Then for any given integer $0\le k\le n$,
  function ${H}\in\mathscr{F}_{n,k}(\mathbb{R}^{2n}, \partial D)$ and any $\epsilon>0$
  there exists $\gamma\in\Gamma_{n,k}$ such that
   \begin{equation}\label{e:EH.3.14}
      \Phi_{{H}}|_{\gamma(B^+_{n,k}\setminus\epsilon B^+_{n,k})}\ge c^{n,k}(D)-\epsilon
      \quad\hbox{and}\quad \Phi_{H}|_{\gamma(B^+_{n,k})}\ge 0,
   \end{equation}
    where $B_{n,k}^{+}$ is the closed unit ball in $E_{n,k}^{+}$.
\end{lemma}

 \begin{proof}
 The case $k=n$ was proved in Lemma 2 of \cite[\S~6.6]{Sik90}. We assume $k<n$ below.
Let $S_{n,k}^{+}=\partial B_{n,k}^{+}$ and $\mathscr{E}^{n,k}_{\epsilon/2}(\mathbb{R}^{2n},\partial D)$ be as in
 (\ref{e:EH.3.7}). Replacing $H$ by a greater function we may assume
 $H\in\mathscr{E}^{n,k}_{\epsilon/2}(\mathbb{R}^{2n},\partial D)$.
 Since $H=0$ near $\partial{D}$,
by the arguments at the end of proof of Proposition~\ref{prop:EH.1.4}, the condition
 (\ref{e:EH.2.1+}) may be satisfied with any $z_0\in\mathbb{R}^{n,k}\cap{\rm Int}({H}^{-1}(0))$.
 Fix such a $z_0\in\mathbb{R}^{n,k}\cap{\rm Int}({H}^{-1}(0))$.
 It follows that  there exists $\alpha>0$ such that
\begin{equation}\label{e:EH.3.15}
\inf \Phi_{H}|_{(z_0+\alpha S^+_{n,k})}>0\quad\hbox{and}\quad \Phi_{H}|_{(z_0+\alpha B^+_{n,k})}\ge 0,
 \end{equation}
 (see (\ref{e:EH.2.2-})-(\ref{e:EH.2.3}) in the proof of Proposition~\ref{prop:EH.1.4}).
 Define $\gamma_\varepsilon:E_{n,k}\to E_{n,k}$ by $\gamma_\varepsilon(z)=z_0+ \alpha z$.
 It is easily seen that $\gamma_\varepsilon\in\Gamma_{n,k}$.
The first inequality in (\ref{e:EH.3.15}) shows that $\gamma_\varepsilon(S^+_{n,k})$ belongs to
the set $\mathcal{F}_{n,k}=\{\gamma(S^+_{n,k})\,|\,\gamma\in\Gamma_{n,k}
   \;\text{and}\;\inf(\Phi_H|_{\gamma(S^+_{n,k})})>0\}$ in  (\ref{e:fnk}).
 Lemma~\ref{lem:PositiveInvariance} shows that
$$
c^{n,k}(H)=\sup_{F\in\mathcal{F}_{n,k}}\inf_{x\in F}\Phi_H(x),
$$
 and $\mathcal{F}_{n,k}$ is positively invariant under the flow $\varphi_u$ of $\nabla \Phi_H$.
  Define $S_u=\varphi_u(z_0+\alpha S^+_{n,k})$ and $d(H)=\sup_{u\ge 0}\inf(\Phi_{H}|_{S_u})$.
It follows from these and (\ref{e:EH.3.15})  that
$$
0<\inf \Phi_{H}|_{S_0}\le d(H)\le \sup_{F\in\mathcal{F}_{n,k}}\inf_{x\in F}\Phi_H(x)=c^{n,k}(H)<\infty.
$$
Since $\Phi_{H}$ satisfies the (PS) condition by Lemma~\ref{lem:EH8}, $d(H)$ is a positive critical value of $\Phi_{H}$, and
  $d(H)\ge c^{n,k}(D)-\epsilon/2$ by Lemma~\ref{lem:EH9}. Moreover,
  by the definition of
$d({H)}$ there exists $r>0$ such that $\Phi_{H}|_{S_r}\ge d(H)-\epsilon/2$ and thus
\begin{equation}\label{e:EH.3.16}
 \Phi_{H}|_{S_r}\ge c^{n,k}(D)-\epsilon.
 \end{equation}
 Because $\Phi_{H}$ is nondecreasing along the flow $\varphi_u$, we arrive at
\begin{equation}\label{e:EH.3.17}
\Phi_{H}|_{S_u}\ge \Phi_{H}|_{S_0}\ge \inf (\Phi_{H}|_{S_0})>0,\quad\forall u\ge 0.
 \end{equation}
Define $\gamma:E_{n,k}\to E_{n,k}$ by
$\gamma(x^++x^0+x^-)=\widetilde{\gamma}(x^+)+x^0+x^-$,
where
\begin{eqnarray*}
&&\widetilde{\gamma}(x)=z_0+ 2(\alpha/\epsilon)x\hspace{14mm}\quad\quad\hbox{if}\quad x\in E^+_{n,k}\;\hbox{and}\;\|x\|_{E_{n,k}}\le\frac{1}{2}\epsilon,\\
&&\widetilde{\gamma}(x)=\varphi_{r(2\|x\|_{E_{n,k}}-\epsilon)/\epsilon}(z_0+\alpha x/\|x\|_{E_{n,k}})\quad\hbox{if}\quad x\in E^+_{n,k}\;\hbox{and}\;
\frac{1}{2}\epsilon<\|x\|_{E_{n,k}}\le \epsilon,\\
&&\widetilde{\gamma}(x)=\varphi_{r}(z_0+\alpha x/\|x\|_{E_{n,k}})\hspace{8mm}\quad\quad\hbox{if}\quad x\in E^+_{n,k}\;\hbox{and}\;
\|x\|_{E_{n,k}}>\epsilon.
\end{eqnarray*}
The first and second lines imply
$\gamma(\frac{\epsilon}{2}B^+_{n,k})=(z_0+\alpha B^+_{n,k})$ and
$\gamma(B^+_{n,k}\setminus\frac{\epsilon}{2}B^+_{n,k})=\bigcup_{0\le u\le r}S_u$,
 respectively, and so
$$
\gamma(B^+_{n,k})=(z_0+\alpha B^+_{n,k})\bigcup_{0\le u\le r}S_u;
$$
the third line implies $\gamma(B^+_{n,k}\setminus\epsilon B^+_{n,k})=S_r$.
It follows from these, (\ref{e:EH.3.15}) and (\ref{e:EH.3.16})-(\ref{e:EH.3.17})
 that $\gamma$ satisfies (\ref{e:EH.3.14}).

Finally, we can also know that $\gamma\in\Gamma_{n,k}$ by considering the homotopy
\begin{eqnarray*}
\gamma_0(x)=2(\alpha/\epsilon)x^++x^0+x^-,\quad
\gamma_u(x)=\frac{1}{u}{(\gamma(ux)-z_0)}+ z_0,\quad 0<u\le 1.
\end{eqnarray*}
\end{proof}

\begin{lemma}\label{lem:EHproductlem1}
 Let integers $n_1, n_2\ge 1$, $0\le k\le n:=n_1+n_2$.
For a bounded strictly convex domain $D\subset\mathbb{R}^{2n_1}$  with $C^2$ smooth boundary $\mathcal{S}$
and containing $0$, it holds that
\begin{equation}\label{e:product11}
c^{n,k}(D\times\mathbb{R}^{2n_2})=c^{n_1,\min\{n_1,k\}}(D).
\end{equation}
Moreover, if $\Omega\subset\mathbb{R}^{2n_2}$ is  a bounded strictly convex domain with $C^2$ smooth boundary
and containing $0$, then
 $$
  c^{n,k}(\mathbb{R}^{2n_1}\times \Omega)=c^{n_2,\max\{k-n_1,0\}}(\Omega).
  $$
\end{lemma}

\begin{proof}
Let $H(z)=(j_D(z))^2$ for $z\in\mathbb{R}^{n_1}$ and define
$$
E_R=\{(z,z')\in\mathbb{R}^{2n_1}\times\mathbb{R}^{2n_2}\,|\,H(z)+ (|z'|/R)^2<1\}.
$$
By the definition and the monotonicity of $c^{n,k}$ we have
 $$
 c^{n,k}(D\times\mathbb{R}^{2n_2})=\sup_R c^{n,k}(E_R).
 $$
Since the function $\mathbb{R}^{2n_1}\times\mathbb{R}^{2n_2}\ni (z,z')\mapsto G(z,z'):=H(z)+ (|z'|/R)^2\in\mathbb{R}$
is convex and  of class $C^{1,1}$, $E_R$ is convex
 and $\mathcal{S}_R=\partial E_R$ is of class $C^{1,1}$.
By Theorem~\ref{th:EHconvex} we arrive at
 $$
 c^{n,k}(E_R)=\min\Sigma_{\mathcal{S}_R}^{n,k}.
 $$
  Let $\lambda$ be a positive number and  $u=(x,x'):[0,\lambda]\rightarrow \mathcal{S}_R$ satisfy
  \begin{equation}\label{e:product+1}
  \dot{u}=X_G(u)\quad\hbox{and}\quad u(\lambda),u(0)\in\mathbb{R}^{n,k},\quad u(\lambda)\sim u(0).
 \end{equation}
  Namely, $u$ is a leafwise chord on $\mathcal{S}_R$ for  $\mathbb{R}^{n,k}$ with action $\lambda$.
Let $k_1=\min\{n_1,k\}$ and $k_2=\max\{k-n_1,0\}$.
Clearly, $k_1+k_2=k$, and (\ref{e:product+1}) is equivalent to the following
  \begin{eqnarray}
&&\dot{x}=X_H(x)\quad\hbox{and}\quad x(\lambda),x(0)\in\mathbb{R}^{n_1,k_1},\quad x(\lambda)\sim x(0),\label{e:product+2}\\
&&\dot{x}'=2J_{2n_2}x'/R^2\quad\hbox{and}\quad x'(\lambda),x'(0)\in\mathbb{R}^{n_2, k_2},\quad x'(\lambda)\sim x'(0)
\label{e:product+3}
\end{eqnarray}
because $\mathbb{R}^{n,k}\equiv(\mathbb{R}^{n_1,k_1}\times\{0\}^{2n_2})+ (\{0\}^{2n_1}\times\mathbb{R}^{n_2,k_2})$.
Note that nonzero constant vectors cannot be solutions of (\ref{e:product+2}) and (\ref{e:product+3}) and that
$H(z)$ and $(|z'|/R)^2$ take constant values along solutions of (\ref{e:product+2}) and (\ref{e:product+3}), respectively.
There exist three possibilities for solutions of (\ref{e:product+2}) and (\ref{e:product+3}):\\
$\bullet$   $x\equiv 0$, $|x'|=R$ and so $2\lambda/R^2\in\pi\mathbb{N}$ if $k<n_1+n_2$, and $2\lambda/R^2\in 2\pi\mathbb{N}$ if $k=n_1+n_2$
 by (\ref{e:product+3}).\\
 $\bullet$  $x'\equiv 0$, $H(x)\equiv 1$ and so $\lambda\in \Sigma_\mathcal{S}^{n_1,\min\{n_1,k\}}$ by (\ref{e:product+2}).\\
 $\bullet$ $H(x)\equiv \delta^2\in (0,1)$ and $|x'|^2=R^2(1-\delta^2)$, where $\delta>0$.
 Then $y(t):=\frac{1}{\delta}x(t)$ and $y'(t):=x'(t/\delta)$ satisfy respectively the following two lines:
  \begin{eqnarray*}
&& \dot{y}=X_H(y)\quad\hbox{and}\quad y(\lambda), y(0)\in\mathbb{R}^{n_1,k_1},\quad y(\lambda)\sim y(0),\quad H(y)\equiv 1,\\
&& \dot{y}'=2J_{2n_2}y'/R^2\quad\hbox{and}\quad y'(\lambda), y'(0)\in\mathbb{R}^{n_2, k_2},\quad y'(\lambda)\sim y'(0),\quad|y'|\equiv R.
 \end{eqnarray*}
 Hence we have also $\lambda\in \Sigma_\mathcal{S}^{n_1,\min\{n_1,k\}}$ by the first line, and
 \begin{center}
 $\lambda\in \frac{R^2\pi}{2}\mathbb{N}$ if $k<n_1+n_2$, \quad  $\lambda\in \pi R^2\mathbb{N}$ if $k=n_1+n_2$
 \end{center}
 by the second line.

 In summary, we always have
  \begin{eqnarray}
&&\Sigma_{\mathcal{S}_R}^{n,k}\subset \Sigma_\mathcal{S}^{n_1,\min\{n_1,k\}}\bigcup \frac{R^2\pi}{2}\mathbb{N}\quad\hbox{if}\;k<n_1+n_2,
\label{e:product+4}\\
&&\Sigma_{\mathcal{S}_R}^{n,k}\subset \Sigma_\mathcal{S}^{n_1,\min\{n_1,k\}}\bigcup {R^2\pi}\mathbb{N}\quad\hbox{if}\;k=n_1+n_2.\label{e:product+5}
 \end{eqnarray}
A solution $x$ of (\ref{e:product+2}) siting on $\mathcal{S}$ gives a solution $u=(x,0)$ of
 (\ref{e:product+1}) on $\mathcal{S}_R$.
 It follows that
$$
\min\Sigma_{\mathcal{S}_R}^{n,k}=\min\Sigma_\mathcal{S}^{n_1,\min\{n_1,k\}}
$$
for $R$ sufficiently large. (\ref{e:product11}) is proved.

The second claim can be proved in the similar way.
\end{proof}

\begin{proof}[Proof of Theorem~\ref{th:EHproduct1}]
Since $D_1\times D_2\subset D_1\times\mathbb{R}^{2n_2}$ and $D_1\times D_2\subset \mathbb{R}^{2n_1}\times D_2$,
 we get
$$
c^{n,k}(D_1\times D_2)\le \min\{c^{n_1,\min\{n_1,k\}}(D_1),c^{n_2,\max\{k-n_1,0\}}(D_2)\}
$$
by Lemma \ref{lem:EHproductlem1}.
In order to prove  the inverse direction inequality  it suffices to prove
\begin{equation}\label{e:oppoproduct}
c^{n,k}(\partial D_1\times \partial D_2)\ge \min\{c^{n_1,\min\{n_1,k\}}(D_1),c^{n_2,\max\{k-n_1,0\}}(D_2)\}
 \end{equation}
because $c^{n,k}(D_1\times D_2)\ge c^{n,k}(\partial D_1\times \partial D_2)$
by the monotonicity.

 We  assume  $n_1\le k$. (The case $n_1>k$ is similar!)  Then (\ref{e:oppoproduct}) becomes
\begin{eqnarray}\label{e:product111}
 c^{n,k}(\partial D_1\times \partial D_2)\ge \min\{c_{\rm EH}(D_1), c^{n_2,k-n_1}(D_2)\}
 \end{eqnarray}
because $c^{n_1,n_1}(D_1)=c_{\rm EH}(D_1)$ by definition.
Note that for each ${H}\in\mathscr{F}_{n,k}(\mathbb{R}^{2n}, \partial D_1\times\partial D_2)$
we may choose $\widehat{H}_1\in\mathscr{F}_{n_1,n_1}(\mathbb{R}^{2n_1}, \partial D_1)$ and $\widehat{H}_2\in \mathscr{F}_{n_2,k-n_1}(\mathbb{R}^{2n_2}, \partial D_2)$ such that
$$
\widehat{H}(z):=\widehat{H}_1(z_1)+\widehat{H}_2(z_2)\ge H(z),\quad\forall z.
$$
Let $k_1=n_1$ and $k_2=n-k_1$. By Lemma~\ref{lem:9.2}, for any
$$
0<\epsilon<\min\{c^{n_1,n_1}(D_1),c^{n_2,k-n_1}(D_2),  1/4\}
$$
and each $i\in\{1,2\}$ there exists  $\gamma_i\in\Gamma_{n_i,k_i}$
such that
\begin{eqnarray}\label{e:product112}
  \Phi_{\widehat{H}_i}|_{\gamma_i(B^+_{n_i,k_i}\setminus\epsilon B^+_{n_i,k_i})}\ge c^{n_i,k_i}(D_i)-\epsilon
      \quad\hbox{and}\quad \Phi_{\widehat{H}_i}|_{\gamma_i(B^+_{n_i,k_i})}\ge 0.
\end{eqnarray}
 Put $\gamma=\gamma_1\times\gamma_2$, which is in $\Gamma_{n,k}$.
Since for any $x=(x_1,x_2)\in S^+_{n,k}\subset B^+_{n_1,k_1}\times B^+_{n_2,k_2}$ there exists some
$j\in\{1,2\}$ such that
$$
x_j\in B_{n_j,k_j}^+\setminus 4^{-1}B_{n_j,k_j}^+\subset B_{n_j,k_j}^+\setminus \epsilon B_{n_j,k_j}^+,
$$
it follows from this and (\ref{e:product112}) that
$$
\Phi_{\widehat{H}}(\gamma(x))=\Phi_{\widehat{H}_1}(\gamma_1(x_1))+
\Phi_{\widehat{H}_2}(\gamma_2(x_2))
\ge\min\{c^{n_1,n_1}(D_1),c^{n_2,k-n_1}(D_2)\}-\epsilon>0
$$
and hence
$$
c^{n,k}({H})\ge
c^{n,k}(\widehat{H})=\sup_{h\in\Gamma_{n,k}}\inf_{y\in h(S^+_{n,k})}\Phi_{\widehat{H}}(y)\ge
\min\{c^{n_1,n_1}(D_1),c^{n_2,k-n_1}(D_2)\}-\epsilon.
$$
This leads to (\ref{e:product111}) because $c^{n_1,n_1}(D_1)=c_{\rm EH}(D_1)$.
\end{proof}

\section{Proof of Theorem~\ref{th:EHcontact}}\label{sec:contact}
\setcounter{equation}{0}

\subsection{The interior of $\Sigma_{\mathcal{S}}$ is empty}

Let $\lambda:=\imath_X\omega_0$, and  $\lambda_0:=\frac{1}{2}(qdp-pdq)$, where $(q,p)$ is the standard
  coordinate on $\mathbb{R}^{2n}$.

 \begin{claim}
   For every leafwise chord  on $\mathcal{S}$ for $\mathbb{R}^{n,k}$,
$x:[0,T]\rightarrow\mathcal{S}$,  there holds
\begin{equation}\label{e:contact1}
A(x)=\int_x\lambda_0=\int_x\lambda.
\end{equation}
\end{claim}
\begin{proof}
Since $\mathcal{S}$ is of class $C^{2n+2}$, so is $x$.
Define $y:[0, T]\rightarrow\mathbb{R}^{n,k}$ by $y(t)= tx(0)+(1-t)x(T)$.
As below (\ref{e:two-action+}) we can take  a piecewise $C^{2n+2}$-smooth map $u$ from
 a suitable closed disc $D^2$ to $\mathbb{R}^{2n}$ such that $u|\partial D^2$
 is equal to the loop $x\cup(-y)$. Now
it is easily checked that
$\int_y \lambda_0=0$ and hence
\begin{equation}\label{e:contact2}
\int_x\lambda_0=\int_{x\cup (-y)}\lambda_0=\int_{u(D^2)}d\lambda_0=\int_{u(D^2)}\omega_0.
\end{equation}
On the other hand, since the flow of $X$ maps $\mathbb{R}^{n,k}$ to $\mathbb{R}^{n,k}$, $X$
is tangent to $\mathbb{R}^{n,k}$ and therefore
$\omega_0(X,\dot{y})=0$, i.e., $y^\ast\lambda=0$.
It follows that
$$
\int_x\lambda=\int_{x\cup(-y)}\lambda=\int_{u(D^2)}d\lambda=\int_{u(D^2)}\omega_0.
$$
This and (\ref{e:contact2})  lead to (\ref{e:contact1}).
\end{proof}

 Choosing   $\varepsilon>0$  so small that $\mathbb{R}^{2n}\setminus \cup_{t\in (-\varepsilon, \varepsilon)}\phi^t(\mathcal{S})$   has two components, we obtain  a very special parameterized family of $C^{2n+2}$ hypersurfaces modelled on $\mathcal{S}$, given by
 $$
\psi:(-\varepsilon, \varepsilon)\times\mathcal{S}\ni (s, z)\mapsto \psi(s,z)=\phi^s(z)\in \mathbb{R}^{2n}
$$
which is $C^{2n+2}$ because  both $\mathcal{S}$ and $X$ are  $C^{2n+2}$.
 Define $U:=\cup_{t\in (-\varepsilon, \varepsilon)}\phi^t(\mathcal{S})$ and
$$
K_\psi:U\to\mathbb{R},\;w\mapsto \tau
$$
if $w=\psi(\tau,z)\in U$ where $z\in\mathcal{S}$. This is $C^{2n+2}$. Denote by  $X_{K_\psi}$  the
Hamiltonian vector field of $K_\psi$ defined by $\omega_0(\cdot,X_{K_\psi})=dK_\psi$. Then it is not hard to prove
  $$
 X_{K_\psi}(\psi(\tau,z))=e^{-\tau}
  d\phi^\tau(z)[X_{K_\psi}(z)]  \quad\forall (\tau,z)\in(-\varepsilon, \varepsilon)\times\mathcal{S},
$$
and for  $w=\phi^\tau(z)=\psi(\tau,z)\in U$ there holds
  \begin{eqnarray}\label{e:EH.4.8}
 \lambda_w(X_{K_\psi})=(\omega_0)_w(X(w), X_{K_\psi}(w))
 =\frac{d}{ds}|_{s=0}K_\psi(\phi^s(w))=1.
\end{eqnarray}

 %
%

Let $\mathcal{S}_\tau:=\psi(\{\tau\}\times\mathcal{S})$.
Since $\phi^t$ preserves the leaf of $\mathbb{R}^{n,k}$,
$y:[0,T]\to \mathcal{S}_\tau$    satisfies
$$
\dot{y}(t)=X_{K_\psi}(y(t)),\quad y(0), y(T)\in\mathbb{R}^{n,k}\quad\hbox{and}\quad y(T)\sim y(0)
$$
if and only if
    $y(t)=\phi^\tau (x(e^{-\tau} t))$, where $x:[0, e^{-\tau} T]\to\mathcal{S}$
    satisfies
    $$
    \dot{x}(t)=X_{K_\psi}(x(t)),\quad x(0), x(e^{-\tau} T)\in\mathbb{R}^{n,k}\quad\hbox{and}\quad x(e^{-\tau} T)\sim x(0).
    $$
    In addition,  $y(t)=\phi^\tau (x(e^{-\tau} t))$ implies
    $\int_y\lambda =e^\tau\int_x\lambda $. By (\ref{e:contact1}) and (\ref{e:EH.4.8}) we deduce
    $$
    A(y)=\int_y\lambda_0=\int_y\lambda=\int^T_0\lambda(\dot y) dt=\int^T_0\lambda_w(X_{K_\psi}) dt=T
    \quad\hbox{and}\quad A(x)=e^{-\tau}T.
    $$

    Fix $0<\delta<\varepsilon$. Let ${\bf A}_\delta$ and ${\bf B}_\delta$
 denote the unbounded and bounded components of
  $\mathbb{R}^{2n}\setminus \cup_{t\in (-\delta, \delta)}\phi^t(\mathcal{S})$,
  respectively.   Then
$\psi(\{\tau\}\times \mathcal{S})\subset {\bf B}_\delta$ for $-\varepsilon<\tau<-\delta$.
Let $\mathscr{F}_{n,k}(\mathbb{R}^{2n})$ be given by (\ref{e:fnk}).
We call $H\in\mathscr{F}_{n,k}(\mathbb{R}^{2n})$ {\bf adapted to} $\psi$ if
\begin{equation}\label{e:EH.4.10}
H(x)=\left\{\begin{array}{ll}
C_0\ge 0 &{\rm if}\;x\in {\bf B}_\delta,\\
f(\tau) &{\rm if}\;x=\psi(\tau,y),\;y\in\mathcal{S},\;\tau\in [-\delta,\delta],\\
C_1\ge 0 &{\rm if}\;x\in {\bf A}_\delta\cap B^{2n}(0,R),\\
h(|x|^2) &{\rm if}\;x\in {\bf A}_\delta\setminus B^{2n}(0,R),
\end{array}\right.
\end{equation}
where $f:(-1,1)\to\mathbb{R}$ and $h:[0, \infty)\to\mathbb{R}$ are smooth functions satisfying
\begin{eqnarray}
&&f|_{(-1,-\delta]}=C_0,\quad f|_{[\delta,1)}=C_1,\label{e:EH.4.11}\\
&&sh'(s)-h(s)\le 0\quad\forall s.\label{e:EH.4.12}
\end{eqnarray}
Clearly, $H$ defined by (\ref{e:EH.4.10}) is $C^{2n+2}$ and its gradient $\nabla H:\mathbb{R}^{2n}\to\mathbb{R}^{2n}$ satisfies a global Lipschitz condition.

\begin{lemma}\label{lem:coEH.4.2}
\begin{description}
\item[(i)] If $x$ is a nonconstant critical point of $\Phi_H$ on $E$ such that
$x(0)\in \psi(\{\tau\}\times\mathcal{S})$ for some $\tau\in (-\delta,\delta)$ satisfying $f'(\tau)>0$,
then
$$
e^{-\tau}f'(\tau)\in \Sigma_{\mathcal{S}}\quad\hbox{and}\quad \Phi_H(x)=f'(\tau)-f(\tau).
$$
\item[(ii)] If some $\tau\in (-\delta,\delta)$ satisfies $f'(\tau)>0$ and $ e^{-\tau}f'(\tau)\in \Sigma_{\mathcal{S}}$,
then there is a nonconstant critical point $x$ of $\Phi_H$ on $E$ such that $x(0)\in \psi(\{\tau\}\times\mathcal{S})$
and $\Phi_H(x)=f'(\tau)-f(\tau)$.
\end{description}
\end{lemma}

\begin{proof}
  {\bf (i)} By Lemma~\ref{lem:critic} $x$ is $C^{2n+2}$ and satisfies
   $\dot{x}=X_H(x)=f'(\tau)X_{K_\psi}(x)$, $x(j)\in\mathbb{R}^{n,k}$, $j=0,1$, and $x(1)\sim x(0)$.
  Moreover $x(0)\in \psi(\{\tau\}\times\mathcal{S})$ implies $H(x(1))=H(x(0))=f(\tau)$ and therefore
 $x(1)\in \psi(\{\tau\}\times\mathcal{S})$ by the construction of $H$ above.
   These show that $x$ is a leafwise chord on  $\psi(\{\tau\}\times\mathcal{S})$
  for $\mathbb{R}^{n,k}$. By the arguments below (\ref{e:EH.4.8}),
   $[0,1]\ni t\mapsto y(t):=\phi^{-\tau}(y(t))$ is a leafwise chord on  $\mathcal{S}$
 for $\mathbb{R}^{n,k}$. It follows from (\ref{e:EH.4.8}) and (\ref{e:contact1}) that
$$
f'(\tau)=\int_0^1f'(\tau)\lambda(X_{K_\psi})dt=\int_0^1\lambda(X_H)dt=
\int_{[0,1]} x^{\ast}\lambda=\int_{[0,1]} y^{\ast}(\phi^{\tau})^{\ast}\lambda=e^\tau\int_{[0,1]} y^{\ast}\lambda=e^\tau A(y)
$$
These show that $e^{-\tau}f'(\tau)=A(y)\in \Sigma_{\mathcal{S}}$.
 By (\ref{e:contact1}) we have
$$
\Phi_H(x)=A(x)-\int_0^1 H(x(t))dt=\int_{[0,1]} x^{\ast}\lambda-\int_0^1 H(x(t))dt=
f'(\tau)-f(\tau).
$$

\noindent{\bf (ii)} By the assumption there exists $y:[0, 1]\to\mathcal{S}$ satisfying
$$
\dot{y}(t)=e^{-\tau}f'(\tau)X_{K_\psi}(y(t)),\quad y(0), y(1)\in\mathbb{R}^{n,k}\quad\hbox{and}\quad y(1)\sim y(0).
$$
Hence $x(t)=\psi(\tau, y(t))=\phi^\tau(y(t))$ satisfies
\begin{eqnarray*}
&&\dot{x}(t)=d\phi^\tau(y(t))[\dot{y}(t)]=e^{-\tau}f'(\tau)d\phi^\tau(y(t))[X_{K_\psi}(y)]\\
&&\quad\quad=f'(\tau)X_{K_\psi}(\phi^\tau(y(t)))=f'(\tau)X_{K_\psi}(x(t))=X_H(x(t)),\\
&&x(0, x(1)\in\mathbb{R}^{n,k}, j=0,1, \quad x(1)\sim x(0)\in \phi^{\tau}(\mathcal{S}).
\end{eqnarray*}
By Lemma~\ref{lem:critic}, $x$ is a critical point of $\Phi_H$. Moreover $\Phi_H(x)=f'(\tau)-f(\tau)$ as in (i).
\end{proof}

\begin{proposition}\label{prop:EH.4.1}
Let  $\mathcal{S}$ be as in Theorem~\ref{th:EHcontact}. Then the interior of
$\Sigma_{\mathcal{S}}$ in $\mathbb{R}$
is empty.
\end{proposition}

\begin{proof}
Otherwise, suppose that $T\in\Sigma_{\mathcal{S}}$ is an interior point of $\Sigma_{\mathcal{S}}$.
 Then for some small $0<\epsilon_1<\delta$ the open neighborhood
  $O:=\{e^{-\tau}T\,|\,\tau\in (-\epsilon_1,\epsilon_1)\}$ of $T$ is contained in
  $\Sigma_{\mathcal{S}}$.
 Let us choose the  function $f$ in (\ref{e:EH.4.10}) such that
$f(u)=Tu+\overline{C}\ge 0\;\forall u\in [-\epsilon_1, \epsilon_1]$
  (by shrinking $0<\epsilon_1<\delta$ if necessary). By Lemma~\ref{lem:coEH.4.2}(ii) we deduce
  $$
   (-\epsilon_1,\epsilon_1)\subset\left\{\tau\in (-\epsilon_1,\epsilon_1)\,|\,e^{-\tau}T\in \Sigma_{\mathcal{S}}\right\}\subset \left\{\tau\in (-\epsilon_1,\epsilon_1)\,|\,T-f(\tau)\;\hbox{is a critical value of }\; \Phi_H\right\}
  $$
 It follows that the critical value set of $\Phi_H$ has nonempty interior. This is a contradiction by
Lemma~\ref{lem:empty}.  Hence  $\Sigma_{\mathcal{S}}$ has empty interior.
\end{proof}


\subsection{$c^{n,k}(U)=c^{n,k}(\mathcal{S})$ belongs to $\Sigma_{\mathcal{S}}$}

This can be obtained by slightly modifying the proof of \cite[Theorem~7.5]{Sik90} (or \cite[Theorem~1.18]{JinLu1915} or \cite[Theorem~1.17]{JinLu1916}).
For completeness we give it in details.
For $C>0$ large enough and $\delta>2\eta>0$ small enough, define
  $H=H_{C,\eta}\in\mathscr{F}_{n,k}(\mathbb{R}^{2n})$
   adapted to $\psi$ as follows:
\begin{equation}\label{e:EH.4.14}
H_{C,\eta}(x)=\left\{\begin{array}{ll}
C\ge 0 &{\rm if}\;x\in {\bf B}_\delta,\\
f_{C,\eta}(\tau) &{\rm if}\;x=\psi(\tau,y),\;y\in\mathcal{S},\;\tau\in [-\delta,\delta],\\
C&{\rm if}\;x\in {\bf A}_\delta\cap B^{2n}(0,R),\\
h(|x|^2) &{\rm if}\;x\in {\bf A}_\delta\setminus B^{2n}(0,R)
\end{array}\right.
\end{equation}
where $B^{2n}(0,R)\supseteq\overline{\psi((-\varepsilon,\varepsilon)\times\mathcal{S})}$
(the closure of $\psi((-\varepsilon,\varepsilon)\times\mathcal{S})$),
$f_{C,\eta}:(-\varepsilon, \varepsilon)\to\mathbb{R}$ and $h:[0, \infty)\to\mathbb{R}$ are smooth functions satisfying
\begin{eqnarray*}
&&f_{C,\eta}|_{[-\eta,\eta]}\equiv 0,\quad f_{C,\eta}(s)=C\;\hbox{if}\;|s|\ge 2\eta,\\
&&f'_{C,\eta}(s)s>0\quad\hbox{if}\quad \eta<|s|<2\eta,\\
&&f'_{C,\eta}(s)-f_{C,\eta}(s)>c^{n,k}(\mathcal{S})+1\quad\hbox{if}\;s>0\;\hbox{and}\;\eta<f_{C,\eta}(s)<C-\eta,\\
&&h_{C,\eta}(s)=a_Hs+b\quad\hbox{for $s>0$ large enough}, a_H=C/R^2> \frac{\pi}{2}, a_H\notin \frac{\pi}{2}\mathbb{N},\\
&&sh'_{C,\eta}(s)-h_{C,\eta}(s)\le 0\quad\forall s\ge 0.
\end{eqnarray*}
We can choose such a family $H_{C,\eta}$ ($C\to+\infty$, $\eta\to 0$)  to be cofinal in
 $\mathscr{F}^{n,k}(\mathbb{R}^{2n},\mathcal{S})$ defined by (\ref{e:fnkb}) and
also to have the property that
\begin{equation}\label{e:EH.4.14.1}
C\le C'\Rightarrow H_{C,\eta}\le H_{C',\eta},\qquad \eta\le \eta'\Rightarrow H_{C,\eta}\ge H_{C,\eta'}.
\end{equation}
It follows that
$$
c^{n,k}(\mathcal{S})=\lim_{\eta\to 0, C\to+\infty}c^{n,k}(H_{C,\eta}).
$$
By Proposition~\ref{prop:coEHC.1}(i) and (\ref{e:EH.4.14.1}),
 $\eta\le \eta'$ implies that $c^{n,k}(H_{C,\eta})\le c^{n,k}(H_{C,\eta'})$, and hence
\begin{equation}\label{e:EH.4.14.2}
\Upsilon(C):=\lim_{\eta\to 0}c^{n,k}(H_{C,\eta})
\end{equation}
exists, and
$$
\Upsilon(C)=\lim_{\eta\to 0}c^{n,k}(H_{C,\eta})\ge \lim_{\eta\to 0}c^{n,k}(H_{C',\eta})=\Upsilon(C'),
$$
i.e.,  $C\mapsto\Upsilon(C)$ is non-increasing.
We claim
\begin{equation}\label{e:EH.4.14.3}
 c^{n,k}(\mathcal{S})=\lim_{C\to+\infty}\Upsilon(C).
\end{equation}
In fact, for any $\epsilon>0$ there exists $\eta_0>0$ and $C_0>0$ such that
$|c^{n,k}(H_{C,\eta})-c^{n,k}(\mathcal{S})|<\epsilon$ for all $\eta<\eta_0$ and $C>C_0$.
Letting $\eta\to 0$ leads to
$|\Upsilon(C)-c^{n,k}(\mathcal{S})|\le\epsilon$ for all $C>C_0$.
 (\ref{e:EH.4.14.3}) holds.

\begin{claim}\label{cl:EH.4.1}
Let $\overline{\Sigma_{\mathcal{S}}}$ be the closure of $\Sigma_{\mathcal{S}}$. Then
$\overline{\Sigma_{\mathcal{S}}}\subset \Sigma_{\mathcal{S}}\cup\{0\}$.
\end{claim}
\begin{proof}
In fact, let $\varphi^t$ denote the flow of $X_{K_\psi}$. It is not hard to prove
$$
\Sigma_{\mathcal{S}}=\{T>0\,|\,\exists z\in\mathcal{S}\cap\mathbb{R}^{n,k}\;\hbox{such that}\;\varphi^T(z)\in\mathcal{S}\cap\mathbb{R}^{n,k}\;\&\; \varphi^T(z)\sim z\}.
$$
Suppose that $(T_k)\subset \Sigma_{\mathcal{S}}$ satisfy $T_k\to T_0\ge 0$.
Then there exists a sequence $(z_k)\subset\mathcal{S}\cap\mathbb{R}^{n,k}$ such that
$\varphi^{T_k}(z_k)\in\mathcal{S}\cap\mathbb{R}^{n,k}$ and $\varphi^{T_k}(z_k)\sim z_k$ for $k=1,2,\cdots$. Define $\gamma_k(t)=\varphi^{T_kt}(z_k)$
for $t\in [0,1]$ and $k\in\mathbb{N}$. Then $\dot{\gamma}_k(t)=T_kX_{K_\psi}(\gamma_k(t))$.
By the Arzel\'a-Ascoli theorem $(\gamma_k)$ has a subsequence converging to
some $\gamma_0$ in $C^\infty([0, 1],\mathcal{S})$, which satisfies
 the following relations
\begin{eqnarray*}
&&\hbox{$\dot{\gamma}_0(t)=T_0X_{K_\psi}(\gamma_0(t))$ for all $t\in [0, 1]$,}\\
&&\hbox{$\gamma_0(0)=\lim_{k\to\infty}\gamma_k(0)=\lim_{k\to\infty}z_k\in \mathcal{S}\cap\mathbb{R}^{n,k}$,}\\
&&\hbox{$\gamma_0(1)=\lim_{k\to\infty}\gamma_k(1)=\lim_{k\to\infty}\varphi^{T_k}(z_k)\in \mathcal{S}\cap\mathbb{R}^{n,k}$,}\\
&&\hbox{$\gamma_0(1)-\gamma_0(0)=\lim_{k\to\infty}(\gamma_k(1)-\gamma_k(0))\in V_0^{n,k}$, i.e., $\gamma_0(1)\sim\gamma_0(0)$.}
\end{eqnarray*}
Hence $\gamma_0(t)=\varphi^{T_0t}(z_0)$ and $T_0\in \Sigma_{\mathcal{S}}$ if $T_0>0$.
It follows that $\overline{\Sigma_{\mathcal{S}}}\subset \Sigma_{\mathcal{S}}\cup\{0\}$.
\end{proof}


{Note that so far  we do not use the assumption $a_H\notin\mathbb{N}\pi/2$.}

\begin{claim}\label{cl:EH.4.2}
If $a_H\notin\mathbb{N}\pi/2$ then either $\Upsilon(C)\in \overline{\Sigma_{\mathcal{S}}}$ or
\begin{eqnarray}\label{e:EH.4.17}
\Upsilon(C)+C\in \overline{\Sigma_{\mathcal{S}}}.
\end{eqnarray}
\end{claim}
\begin{proof}
Since $a_H\notin\mathbb{N}\pi/2$, by Theorem~\ref{th:EH.1.6} we get that $c^{n,k}(H_{C,\eta})$ is  a positive critical value of
$\Phi_{H_{C,\eta}}$ and the associated
critical point $x\in E$ gives rise to  a nonconstant leafwise chord
 sitting in the interior of $U$. Then
  Lemma~\ref{lem:coEH.4.2}(i) yields
$$
c^{n,k}(H_{C,\eta})=\Phi_{H_{C,\eta}}(x)=f'_{C,\eta}(\tau)-f_{C,\eta}(\tau),
$$
where $f'_{C,\eta}(\tau)\in e^\tau\Sigma_{\mathcal{S}}$ and $\eta<|\tau|<2\eta$.
Choose $C>0$ so large that
$c^{n,k}(H_{C,\eta})<c^{n,k}(\mathcal{S})+1$.
By the choice of $f$ below (\ref{e:EH.4.14}) we get
either $f_{C,\eta}(\tau)<\eta$ or $f_{C,\eta}(\tau)>C-\eta$.
Moreover $c^{n,k}(H_{C,\eta})>0$ implies $f'_{C,\eta}(\tau)>f_{C,\eta}(\tau)\ge 0$ and so $\tau>0$.

Take a sequence of positive numbers $\eta_n\to 0$. By the arguments above, passing to a subsequence
we have the following two cases.

\noindent{\bf Case 1}. For each $n\in\mathbb{N}$,
$c^{n,k}(H_{C,\eta_n})=f'_{C,\eta_n}(\tau_n)-f_{C,\eta_n}(\tau_n)
=e^{\tau_n}a_n-f_{C,\eta_n}(\tau_n)$, where $a_n\in\Sigma_\mathcal{S}$, $0\le f_{C,\eta_n}(\tau_n)<\eta_n$ and $\eta_n<\tau_n<2\eta_n$.

\noindent{\bf Case 2}. For each $n\in\mathbb{N}$,
$c^{n,k}(H_{C,\eta_n})=f'_{C,\eta_n}(\tau_n)-f_{C,\eta_n}(\tau_n)
=e^{\tau_n}a_n-f_{C,\eta_n}(\tau_n)=e^{\tau_n}a_n-C-(f_{C,\eta_n}(\tau_n)-C)$,
where $a_n\in\Sigma_\mathcal{S}$, $C-\eta_n<f_{C,\eta_n}(\tau_n)\le C$ and $\eta_n<\tau_n<2\eta_n$.

In Case 1, since $c^{n,k}(H_{C,\eta_n})\to\Upsilon(C)$ by (\ref{e:EH.4.14.2}), the sequence $a_n=e^{-\tau_n}(c^{n,k}(H_{C,\eta_n})+f_{C,\eta_n}(\tau_n))$
is  bounded. Passing to a subsequence we may assume  $a_n\to a_C\in \overline{\Sigma_{\mathcal{S}}}$.
Then
\begin{eqnarray*}
a_C=\lim_{n\to\infty}a_n=\lim_{n\to\infty}\left(e^{-\tau_n}(c^{n,k}(H_{C,\eta_n})+f_{C,\eta_n}(\tau_n))\right)
=\Upsilon(C)
\end{eqnarray*}
because $e^{-\tau_n}\to 1$ and $f_{C,\eta_n}(\tau_n)\to 0$.

Similarly,  we can prove $\Upsilon(C)+C=a_C\in \overline{\Sigma_{\mathcal{S}}}$ in Case 2.
\end{proof}

\noindent{\bf Step 1}. {\it Prove $c^{n,k}(\mathcal{S})\in\overline{\Sigma_{\mathcal{S}}}$.}\quad
 Suppose that there exists an increasing sequence $C_n$ tending to $+\infty$ such that $C_n/R^2\notin\mathbb{N}\pi/2$  and
  $\Upsilon(C_n)\in \Sigma_{\mathcal{S}}$ for each $n$.
Since $(\Upsilon(C_n))$ is non-increasing we conclude
\begin{equation}\label{e:EH.4.17+}
c^{n,k}(\mathcal{S})=\lim_{n\to\infty}\Upsilon(C_n)\in \overline{\Sigma_\mathcal{S}}.
\end{equation}
Otherwise, we have
\begin{equation}\label{e:EH.4.18}
\left.\begin{array}{ll}
&\hbox{there exists $\bar{C}>0$ such that
(\ref{e:EH.4.17}) holds}\\
&\hbox{for each $C\in (\bar{C}, +\infty)$
satisfying $C/R^2\notin\mathbb{N}\pi/2$.}
\end{array}\right\}
\end{equation}

\begin{claim}\label{cl:EH.4.3}
Let $\bar{C}>0$ be as in (\ref{e:EH.4.18}).
Then for any $C<C'$ in $(\bar{C}, +\infty)$ there holds
$$
\Upsilon(C)+C\ge \Upsilon(C')+C'.
$$
\end{claim}

Its proof is carried out later.
Since $\Xi:=\{C>\bar{C}\,|\, C\;\hbox{satisfying $C/R^2\notin\mathbb{N}\pi/2$}\}$
is dense in $(\bar{C}, +\infty)$, it follows from Claim~\ref{cl:EH.4.3} that
$\Upsilon(C')+C'\le \Upsilon(C)+C$ if $C'>C$ are in
$\Xi$. Fix a $C^\ast\in\Xi$. Then
$\Upsilon(C')+C'\le \Upsilon(C^\ast)+C^\ast$ for all $C'\in\{C\in\Xi\,|\, C>C^\ast\}$.
Taking a sequence  $(C_n')\subset \{C\in\Xi\,|\, C>C^\ast\}$ such that $C_n'\to +\infty$,
we deduce that $\Upsilon(C_n')\to-\infty$. This contradicts the fact that
$\Upsilon(C_n')\to c^{n,k}(\mathcal{S})>0$. Hence
(\ref{e:EH.4.18}) does not  hold! (\ref{e:EH.4.17+}) is proved.

\begin{proof}[Proof of Claim~\ref{cl:EH.4.3}]
By contradiction we assume that for some $C'>C>\overline{C}$,
\begin{equation}\label{e:EH.4.20}
\Upsilon(C)+C<\Upsilon(C')+C'.
\end{equation}
Let us prove that (\ref{e:EH.4.20}) implies:
\begin{equation}\label{e:EH.4.21}
\left.\begin{array}{ll}
 &\hbox{ for any given
$d\in (\Upsilon(C)+C,\Upsilon(C')+C')$}\\
&\hbox{ there exists $C_0\in (C, C')$ such that
$\Upsilon(C_0)+C_0=d$.}
\end{array}\right\}
\end{equation}
Clearly, this contradicts the facts that
${\rm Int}(\Sigma_{\mathcal{S}})=\emptyset$ and (\ref{e:EH.4.17})
holds for all large $C$ satisfying $C/R^2\notin\mathbb{N}\pi/2$.

It remains to prove (\ref{e:EH.4.21}).
Put $\Delta_d=\{C''\in (C, C')\,|\, C''+\Upsilon(C'')>d\}$. Since
$\Upsilon(C')+C'>d$ and
$\Upsilon(C')\le\Upsilon(C'')\le\Upsilon(C)$ for any $C''\in (C, C')$
we obtain $\Upsilon(C'')+C''>d$ if $C''\in (C, C')$ is sufficiently close to $C'$.
Hence $\Delta_d\ne\emptyset$. Set $C_0=\inf \Delta_d$. Then $C_0\in [C, C')$.

Let $(C_n'')\subset \Delta_d$ satisfy $C_n''\downarrow C_0$.
Since $\Upsilon(C_n'')\le \Upsilon(C_0)$, we have
$d<C_n''+\Upsilon(C_n'')\le\Upsilon(C_0)+ C_n''$ for each $n\in\mathbb{N}$,
and thus $d\le\Upsilon(C_0)+C_0$ by letting $n\to\infty$.

We conclude $d=\Upsilon(C_0)+C_0$, and so (\ref{e:EH.4.21}) is proved.
By contradiction suppose that
\begin{equation}\label{e:EH.4.22}
d<\Upsilon(C_0)+C_0.
\end{equation}
Since $d>C+\Upsilon(C)$, this implies $C\ne C_0$ and so $C_0>C$.
For $\hat{C}\in (C, C_0)$, as $\Upsilon(\hat{C})\ge\Upsilon(C_0)$  we derive from
(\ref{e:EH.4.22}) that $\Upsilon(\hat{C})+\hat{C}>d$
if $\hat{C}$ is close to $C_0$. Hence such $\hat{C}$ belongs to
$\Delta_d$, which contradicts $C_0=\inf\Delta_d$.
%
\end{proof}

\noindent{\bf Step 2.} {\it Prove $c^{n,k}(U)=c^{n,k}(\mathcal{S})$}.
Note that $c^{n,k}(U)=\inf_{\eta>0, C>0}c^{n,k}(\hat{H}_{C,\eta})$,
where
$$
\hat{H}_{C,\eta}(x)=\left\{\begin{array}{ll}
0 &{\rm if}\;x\in {\bf B}_\delta,\\
\hat{f}_{C,\eta}(\tau) &{\rm if}\;x=\psi(\tau,y),\;y\in\mathcal{S},\;\tau\in [-\delta,\delta],\\
C &{\rm if}\;x\in {\bf A}_\delta\cap B^{2n}(0,R),\\
\hat{h}(|x|^2) &{\rm if}\;x\in {\bf A}_\delta\setminus B^{2n}(0,R)
\end{array}\right.
$$
where $B^{2n}(0,R)\supseteq\overline{\psi((-\varepsilon,\varepsilon)\times\mathcal{S})}$, $\hat{f}_{C,\eta}:(-\varepsilon, \varepsilon)\to\mathbb{R}$ and $\hat{h}:[0, \infty)\to\mathbb{R}$ are smooth functions satisfying  the following conditions
\begin{eqnarray*}
&&\hat{f}_{C,\eta}|_{(-\infty,\eta]}\equiv 0,\quad \hat{f}_{C,\eta}(s)=C\;\hbox{if}\;s\ge 2\eta,\\
&&\hat{f}'_{C,\eta}(s)s>0\quad\hbox{if}\quad \eta<s<2\eta,\\
&&\hat{f}'_{C,\eta}(s)-\hat{f}_{C,\eta}(s)>c^{n,k}(\mathcal{S})+1\quad\hbox{if}\;s>0\;\hbox{and}\;
\eta<\hat{f}_{C,\eta}(s)<C-\eta,\\
&&\hat{h}_{C,\eta}(s)=a_Hs+b\quad\hbox{for $s>0$ large enough}, a_H=C/R^2>\frac{\pi}{2}, a_H\notin \frac{\pi}{2}\mathbb{N},\\
&&s\hat{h}'_{C,\eta}(s)-\hat{h}_{C,\eta}(s)\le 0\quad\forall s\ge 0.
\end{eqnarray*}

For $H_{C,\eta}$ in (\ref{e:EH.4.14}),  choose an associated
$\hat{H}_{C,\eta}$, where $\hat{f}_{C,\eta}|_{[0,\infty)}={f}_{C,\eta}|_{[0,\infty)}$
and $\hat{h}_{C,\eta}={h}_{C,\eta}$. Consider $H_s=sH_{C,\eta}+(1-s)\hat{H}_{C,\eta}$, $0\le s\le 1$, and put
$\Phi_s(x):=\Phi_{H_s}(x)$ for $x\in E$.

It suffices to prove $c^{n,k}(H_0)=c^{n,k}(H_1)$. If $x$ is a critical point of $\Phi_s$ with $\Phi_s(x)>0$,
as in Lemma~\ref{lem:coEH.4.2}, we have $x([0,1])\in \mathcal{S}_\tau=\psi(\{\tau\}\times\mathcal{S})$
for some $\tau\in (\eta,2\eta)$. The choice of $\hat{H}_{C,\eta}$ shows
$H_s(x(t))\equiv {H}_{C,\eta}(x(t))$ for $t\in [0,1]$. This implies that
each $\Phi_s$ has the same positive critical value as $\Phi_{H_{C,\eta}}$.
By the continuity in Proposition~\ref{prop:coEHC.1}(ii), $s\mapsto c^{n,k}(H_s)$ is continuous
and takes values in the set of positive critical value of $\Phi_{H_{C,\eta}}$
(which has measure zero by Sard's theorem). Hence  $s\mapsto c^{n,k}(H_s)$ is constant. We get
$c^{n,k}(\hat{H}_{C,\eta})=c^{n,k}(H_0)=c^\Psi_{\rm EH}(H_1)=c^{n,k}(H_{C,\eta})$.

Summarizing the above arguments  we have proved that
$c^{n,k}(\mathcal{S})=c^{n,k}(U)\in \overline{\Sigma_{\mathcal{S}}}$.
Noting that $c^{n,k}(U)>0$, we deduce
$c^{n,k}(\mathcal{S})=c^{n,k}(U)\in \Sigma_{\mathcal{S}}$ by Claim~\ref{cl:EH.4.1}.


\section{Proof of Theorem~\ref{th:non-triviality}}\label{sec:normal}
\setcounter{equation}{0}

For $W^{2n}(1)$ in (\ref{e:Ball2}),
note that $W^{2n}(1)\equiv \mathbb{R}^{2n-2}\times W^2(1)\supseteq\mathbb{R}^{2n-2}\times U^2(1)$ via the identification under (\ref{e:product2}).
 For each integer $0\le k<n$, (\ref{e:product3.1}) and  (\ref{e:product1}) yield
 $$
 c^{n,k}(W^{2n}(1))\ge\min\{c^{n-1,k}(\mathbb{R}^{2n-2}), c^{1,0}(U^2(1))\}=\frac{\pi}{2}.
$$
We only need to prove  the inverse direction of the inequality.

 Fix a number $0<\varepsilon<\frac{1}{100}$. For  $N>2$ define
$$
W^{2}(1,N):=\left \{(x_n,y_n)\in W^{2}(1)\;|\;|x_n|< N,\; |y_n|<N\right\}.
$$
Let us smoothen $W^{2}(1)$ and $W^{2}(1,N)$ in the following way.
Choose positive numbers $\delta_1, \delta_2\ll 1$ and a smooth even function
$g:\mathbb{R}\to \mathbb{R}$ satisfying the following conditions:
\begin{description}
\item[(i)] $g(t)=\sqrt{1-t^2}$ for $0\le t\le 1-\delta_1$,
\item[(ii)] $g(t)=0$ for $t\ge 1+\delta_2$,
\item[(iii)] $g$ is strictly monotone decreasing, and $g(t)\ge \sqrt{1-t^2}$ for $1-\delta_1\le t\le 1$.
\end{description}
Denote by
\begin{eqnarray*}
W^{2}_g(1):=\{(x_n,y_n)\in\mathbb{R}^{2}\,|\, y_n<g(x_n)\},
\end{eqnarray*}
and by $W^{2}_{g}(1,N)$ the open subset in $\mathbb{R}^2(x_n,y_n)$ surrounded by
curves $y_n=g(x_n)$, $y_n=-N$, $x_n=N$ and $x_n=-N$
(see Figure 2 ). Then $W^{2}_{g}(1,N)$ contains $W^2(1,N)$,
and we can require $\delta_1,\delta_2$ so small that
\begin{equation}\label{e:7.9}
0<{\rm Area}(W^{2}_{g}(1,N))-{\rm Area}(W^2(1,N))<\frac{\varepsilon}{2}.
\end{equation}
Take another smooth function
$h:[0, \infty)\to \mathbb{R}$ satisfying the following conditions:
\begin{description}
\item[(iv)] $h(0)=\frac{\varepsilon}{2}$ and $h(t)=0$ for $t>\frac{\varepsilon}{2}$,
\item[(v)] $h'(t)<0$ and $h''(t)>0$  for any $t\in (0, \frac{\varepsilon}{2})$,
\item[(vi)] the curve $\{(t, h(t))\,|\, 0\le t\le\frac{\varepsilon}{2}\}$ is symmetric with respect to line $s=t$ in $\mathbb{R}^2(s,t)$.
\end{description}

Let $\triangle_1$ be the closed domain in $\mathbb{R}^2(x_n,y_n)$ surrounded by
curves $y_n=h(x_n)$, $y_n=0$ and $x_n=0$ (see Figure~\ref{fig:1} ). Denote by
\begin{eqnarray*}
\triangle_2=\{(x_n,y_n)\in \mathbb{R}^2\,|\,(-x_n,y_n)\in \triangle_1\},\quad \triangle_3=-\triangle_1,\quad \triangle_4=-\triangle_2.
\end{eqnarray*}

\begin{figure}[!htb]
\centering
\includegraphics[clip,width=8cm]{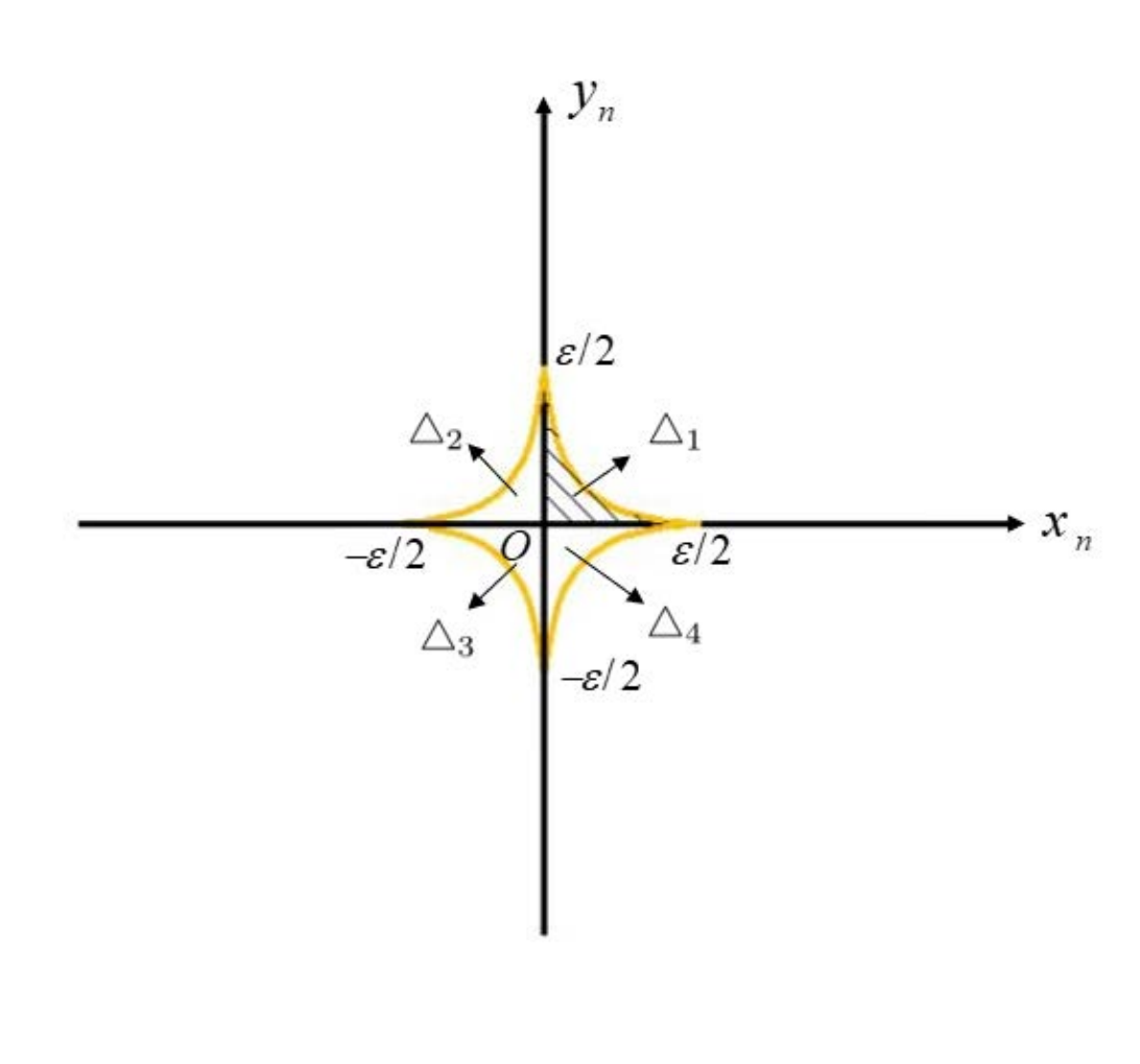}
\caption{The domains $\triangle_i$, $i=1,2,3,4$.}\label{fig:1}
\end{figure}
Let $p_1=(N,0),\; p_2=(-N,0),\; p_3=(-N,-N),\; p_4=(N,-N)$. Define
$$
W^{2}_{g,\varepsilon}(1,N)=W^{2}_{g}(1,N)\setminus\big((p_1+ \triangle_3)\cup(p_2+ \triangle_4)\cup(p_3+ \triangle_1)\cup(p_4+ \triangle_2)\bigr).
$$
Then $W^{2}_{g,\varepsilon}(1,N)$ has smooth boundary (see Figure~\ref{fig:2}) and
$$
0<{\rm Area}(W^2_g(1,N))-{\rm Area}(W^{2}_{g,\varepsilon}(1,N))=4{\rm Area}(\triangle_1)<4\left(\frac{\varepsilon}{2}\right)^2=\varepsilon^2<\frac{\varepsilon}{2}.
$$

\begin{figure}[!htb]
\centering
\includegraphics[clip,width=8cm]{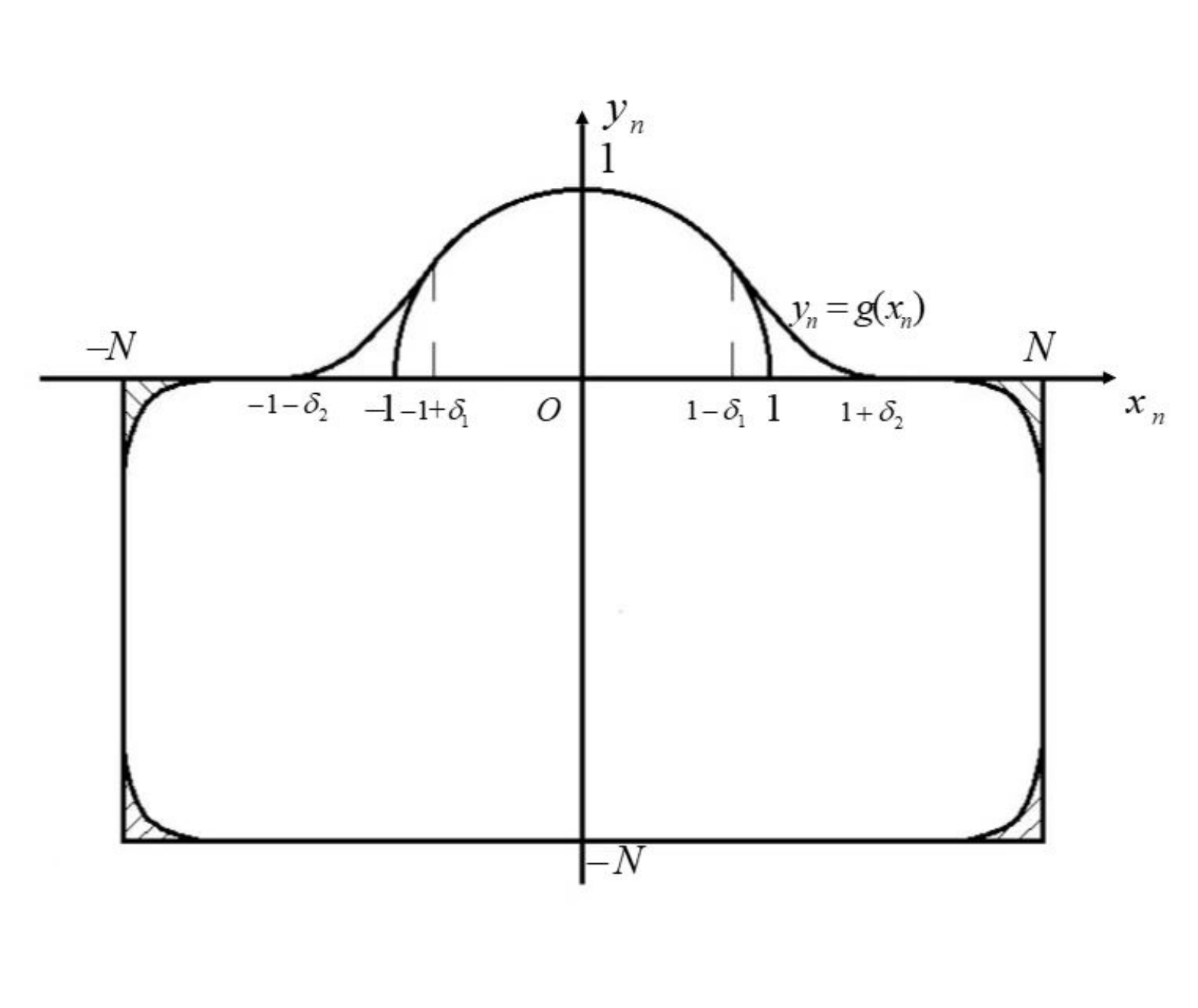}
\caption{The domain $W^{2}_{g,\varepsilon}(1,N)$.}\label{fig:2}
\end{figure}


For $n>1$ and $N>2$ we define
\begin{eqnarray*}
&&W^{2n}_{g}(1):=\{(x,y)\in\mathbb{R}^{2n}\,|\, (x_n,y_n)\in W^2_{g}(1)\}=\mathbb{R}^{2n-2}\times W^{2}_{g}(1) ,\\
&&W^{2n}(1,N):=\left \{(x,y)\in W^{2n}(1)\;|\;|x_n|< N,\; |y_n|<N\right\}=\mathbb{R}^{2n-2}\times W^{2}(1,N),\\
&&W^{2n}_{g}(1,N):=\{(x,y)\in\mathbb{R}^{2n}\,|\, (x_n,y_n)\in W^2_{g}(1, N)\}=\mathbb{R}^{2n-2}\times W^{2}_g(1,N),\\
&&W^{2n}_{g,\varepsilon}(1,N):=\{(x,y)\in\mathbb{R}^{2n}\,|\, (x_n,y_n)\in W^2_{g,\varepsilon}(1,N)\}=\mathbb{R}^{2n-2}\times W^{2}_{g,\varepsilon}(1,N).
\end{eqnarray*}
Clearly, $W^{2n}_{g,\varepsilon}(1, N)\subset W^{2n}_{g,\varepsilon}(1,M)$ for any $M>N>2$, and each bounded subset of
$W^{2n}_g(1)$ can be contained in $W^{2n}_{g,\varepsilon}(1,N)$ for some large $N>2$.
It follows that
\begin{equation}\label{e:7.10}
c^{n,k}(W^{2n}_g(1))=\sup_{N>2}\{c^{n,k}(W^{2n}_{g,\varepsilon}(1,N))\}=\lim_{N\to+\infty}c^{n,k}(W^{2n}_{g,\varepsilon}(1,N)).
\end{equation}

Let us estimate $c^{n,k}(W^{2n}_{g,\varepsilon}(1,N))$ with Theorem \ref{th:EHcontact}.
 Regrettably, $W^{2}_{g,\varepsilon}(1,N)$ is not  star-shaped with respect to the origin.
Fortunately, it can be approximated arbitrarily by star-shaped domains with respect to the origin and with smooth boundary.
Indeed, for a very small $0<\eta<\varepsilon$ the set
$$
W^{2}_{g,\varepsilon}(1,N,\eta):=W^{2}_{g,\varepsilon}(1,N)\cup (W^{2}_{g,\varepsilon}(1,N)+ (0,\eta))
$$
is the desired one.

Define $j_{g,N,\epsilon, \eta}:\mathbb{R}^2\to\mathbb{R}$ by
$$
j_{g,N,\epsilon,\eta}(z_n):=\inf\left\{\lambda>0\;\Big|\frac{z_n}{\lambda}\in W^{2}_{g,\varepsilon}(1,N,\eta)\right\},\quad\forall z_n=(x_n,y_n)\in\mathbb{R}^2.
$$
Then $j_{g,N,\epsilon,\eta}$ is positively homogeneous, and smooth in $\mathbb{R}^2\setminus\{0\}$. For
$(x,y)\in\mathbb{R}^{2n}$ we write $(x,y)=(\hat{z},z_n)$ and  define
$$
W^{2n}_{g,\varepsilon,R}(1,N,\eta):=\left\{\frac{|\hat{z}|^2}{R^2}+j_{g,N,\epsilon,\eta}^2(z_n)<1\right\}, \quad \forall R>0.
$$
Then we have  $W^{2n}_{g,\varepsilon,R_1}(1,N,\eta)\subset W^{2n}_{g,\varepsilon, R_2}(1,N,\eta)$ for $R_1<R_2$, and
$$
W^{2n}_{g,\varepsilon}(1,N,\eta)=\bigcup_{R>0}W^{2n}_{g,\varepsilon,R}(1,N,\eta),
$$
which implies by (\ref{e:coCap+}) that
\begin{equation}\label{e:7.11}
c^{n,k}(W^{2n}_{g,\varepsilon}(1,N,\eta))=\lim_{R\to+\infty}c^{n,k}(W^{2n}_{g,\varepsilon,R}(1,N,\eta)).
\end{equation}

Observe that for  arbitrary $N>2$ and $R>0$ we can shrink $0<\eta<\varepsilon$ so that there holds
$$
W^{2n}_{g,\varepsilon,R}(1,N,\eta)\subset W^{2n}_{g,\varepsilon}(1,N,\eta)\subset U^{2n}(N),
$$
where for $r>0$,
$$
U^{2n}(r):=\{(x,y)\in\mathbb{R}^{2n}\;|\;x_n^2+y_n^2<r^2\}\cup \{(x,y)\in\mathbb{R}^{2n}\;|\;|x_n|<r\;\hbox{and}\;y_n<0\}.
$$
We obtain
\begin{equation}\label{e:bound}
c^{n,k}(W^{2n}_{g,\varepsilon,R}(1,N,\eta))\le c^{n,k}(W^{2n}_{g,\varepsilon}(1,N,\eta))\le c^{n,k}(U^{2n}(N))=\frac{\pi}{2}N^2.
\end{equation}
Note that $W^{2n}_{g, \varepsilon, R}(1,N,\eta)$ is a star-shaped domain with respect to the origin and with smooth boundary
$\mathcal{S}_{N, g, \varepsilon,R,\eta}$
 transversal to  the  globally defined Liouville vector field
$X(z)=z$.
Since the flow $\phi^t$ of $X$, $\phi^t(z)=e^tz$, maps $\mathbb{R}^{n,k}$ to $\mathbb{R}^{n,k}$ and preserves the leaf relation of $\mathbb{R}^{n,k}$, by Theorem \ref{th:EHcontact} we obtain
$$
c^{n,k}(W^{2n}_{g,\varepsilon,R}(1,N,\eta))\in \Sigma_{\mathcal{S}_{N, g, \varepsilon,R,\eta}}
$$
where
$$
\Sigma_{\mathcal{S}_{g,N,\varepsilon, R,\eta}}=\{A(x)>0\;|\;x\;\hbox{is a leafwise chord on\;$\mathcal{S}_{N, g, \varepsilon,R,\eta}$\;for \;$\mathbb{R}^{n,k}$}\}.
$$
Arguing as in the proof of (\ref{e:product+4})
we get that
$$
\Sigma_{\mathcal{S}_{g,N,\varepsilon,R,\eta}}\subset \Sigma _{\partial W^2_{g,\varepsilon}(1,N,\eta)}\bigcup \frac{\pi R^2}{2}\mathbb{N}.
$$
Hence for $R>N$, by (\ref{e:bound}) we have
\begin{equation}\label{e:bound+}
c^{n,k}(W^{2n}_{g,\varepsilon,R}(1,N,\eta))\in \Sigma _{\partial W^2_{g,\varepsilon}(1,N,\eta)}.
\end{equation}

Let us compute $\Sigma _{\partial W^2_{g,\varepsilon}(1,N,\eta)}$.
Note that the part of $\partial W^{2}_{g,\varepsilon}(1,N)$ over the line $y_n=-\frac{\varepsilon}{2}$ and between lines $x_n=-N$ and $x_n=N$ is
$\{(x_n, f(x_n))\in\mathbb{R}^2\,|\, |x_n|\le N\}$, where
$$
f(x_n)=\left\{\begin{array}{ll}
-h(x_n+N) &{\rm if}\;-N\le x_n\le -N+\frac{\varepsilon}{2},\\
g(x_n) &{\rm if}\;-N+\frac{\varepsilon}{2}<x_n<N-\frac{\varepsilon}{2},\\
-h(-x_n+N) &{\rm if}\;N-\frac{\varepsilon}{2}<x_n\le N.
\end{array}\right.
$$
Let $t_0\in (0, \varepsilon/2)$ be the unique number satisfying $h(t_0)=\eta$.
Then  there only exist two leafwise chords on $\partial W^2_{g,\varepsilon}(1,N,\eta)$ for $\mathbb{R}^{1,0}$.
One is the curve in $\mathbb{R}^2(x_n,y_n)$,
$$
\gamma_1:=\{(x_n, \eta+ f(x_n))\in\mathbb{R}^2\,|\, t_0-N\le x_n\le N-t_0\},
$$
 and  the other is
$\gamma_2:=\partial W^{2}_{g,\varepsilon}(1,N,\eta)\setminus\gamma_1$.
Then $A(\gamma_1)$ is equal to the area of the domain
in $\mathbb{R}^2(x_n,y_n)$ surrounded by
curves $\gamma_1$ and $x_n$-axis, that is,
\begin{eqnarray}\label{e:7.73}
A(\gamma_1)&=&\int^{N-t_0}_{t_0-N}(\eta+ f(x_n))dx_n\nonumber\\
&=&2(N-t_0)\eta+
{\rm Area}(W^{2}_{g}(1,N))-2N^2-2\int_{t_0}^{\frac{\varepsilon}{2}}h(t)dt,
\end{eqnarray}
 and
 \begin{eqnarray}\label{e:7.74}
 A(\gamma_2)
&=&2N^2- 2{\rm Area}(\triangle_1)-2\int^{t_0}_0h(t)dt\nonumber\\
&\ge& 2N^2- 4{\rm Area}(\triangle_1)\nonumber\\
&>&2N^2-\varepsilon.
\end{eqnarray}
Hence $\Sigma _{\partial W^2_{g,\epsilon}(1,N,\eta)}=\left\{
A(\gamma_1), A(\gamma_2)\right\}$. Let us choose  $N>2$ so large that
$\frac{\pi}{2}N^2<2N^2-\varepsilon$. Then
(\ref{e:bound}), (\ref{e:bound+}) and (\ref{e:7.74}) lead to
 \begin{eqnarray}\label{e:7.75}
c^{n,k}(W^{2n}_{g,\varepsilon,R}(1,N,\eta))=A(\gamma_1).
\end{eqnarray}
Note that $2N^2- 4{\rm Area}(\triangle_1)>2N^2-\varepsilon$ and that (\ref{e:7.9}) implies
$$
{\rm Area}(W^{2}_{g}(1,N))-2N^2 < {\rm Area}(W^2(1,N))+ \frac{\varepsilon}{2}-2N^2= \frac{\pi}{2}+\frac{\varepsilon}{2}.
$$
It follows from this,  (\ref{e:7.73}) and (\ref{e:7.75}) that
$$
c^{n,k}(W^{2n}_{g,\varepsilon,R}(1,N,\eta))=A(\gamma_1)< \frac{\pi}{2}+  \frac{\varepsilon}{2}+2(N-t_0)\eta-2\int_{t_0}^{\frac{\varepsilon}{2}}h(t)dt.
$$
For fixed $N$ and $\varepsilon$ we may choose $0<\eta<\varepsilon$ so small that $2(N-t_0)\eta<\frac{\varepsilon}{2}$. Then
$$
c^{n,k}(W^{2n}_{g,\varepsilon,R}(1,N,\eta))< \frac{\pi}{2}+  \varepsilon.
$$
%
From this and (\ref{e:7.10})-(\ref{e:7.11}) we derive
$$
c^{n,k}(W^{2n}(1))\le
c^{n,k}(W^{2n}_g(1))\le \frac{\pi}{2}+ \varepsilon
$$
and hence $c^{n,k}(W^{2n}(1))\le \frac{\pi}{2}$ by letting $\varepsilon\to 0+$.

\section{Comparison to symmetrical Ekeland-Hofer capacities}\label{sec:compare}


For each $i=1,\cdots,n$, let $e_i$ be the vector in $\mathbb{R}^{2n}$ with $1$ in the $i$-th position and $0$s elaewhere.
Then $\{e_i\}_{i=1}^n$ is an orthonormal basis for
$L_0^n:=V_0^{n,0}=\{x\in\mathbb{R}^{2n}\;|\;x=(q_{1},\cdots,q_n,0,\cdots,0)\}=\mathbb{R}^{n,0}$.
It was proved in \cite[Corollary~2.2]{JinLu1917} that
$L^2([0,1],\mathbb{R}^{2n})$ has an orthogonal basis
$$
\{e^{m\pi tJ_{2n}}e_i\}_{1\le i\le n, m\in\mathbb{Z}},
$$
and every $x\in L^2([0,1],\mathbb{R}^{2n})$ can be uniquely expanded as form
$x=\sum_{m\in\mathbb{Z}}e^{m\pi tJ_{2n}}x_m$,
where $x_m\in L_0^n$ for all $m\in\mathbb{Z}$  and satisfies
$\sum_{m\in\mathbb{Z}}|x_m|^2<\infty$. Noting that
$V^{n,0}_1=\{0\}$, the spaces in (\ref{e:space1}) and (\ref{e:space2}) become, respectively,
\begin{eqnarray*}
L^2_{n,0}&=&\Big\{x\in L^2([0,1],\mathbb{R}^{2n})\,\Big|\, x\stackrel{L^2}{=}\sum_{m\in\mathbb{Z}}e^{m\pi tJ_{2n}}a_m,\;
a_m\in L^{n}_0,\;\sum_{m\in\mathbb{Z}}|a_m|^2<\infty\Big\}\\
&=&L^2([0,1],\mathbb{R}^{2n})
\end{eqnarray*}
and
\begin{eqnarray*}
H^s_{n,0}=\Big\{x\in L^2([0,1],\mathbb{R}^{2n})\,\Big|\, x\stackrel{L^2}{=}\sum_{m\in\mathbb{Z}}e^{m\pi tJ_{2n}}a_m,
\; a_m\in L^n_0, \; \sum_{m\in\mathbb{Z}}|m|^{2s}|a_m|^2<\infty\Big\}
\end{eqnarray*}
for any real $s\ge0$. 
%
It follows that the space $\mathbb{E}$ in \cite[\S1.2]{JinLu1915} is a subspace of
$E=H^{1/2}_{n,0}$ in (\ref{e:space3}). Denote by  $\widehat{\Gamma}$ the set of  the admissible deformations on $\mathbb{E}$ (see \cite[\S1.2]{JinLu1915}) and $\widehat{S}^+$
the unit sphere in $\mathbb{E}$. Then
$\Gamma_{n,0}|_{\mathbb{E}}\subset \widehat{\Gamma}$ and $\widehat{S}^+\subset S^+_{n,0}$.
Note that each function $H\in C^0(\mathbb{R}^{2n},\mathbb{R}_{\ge 0})$ satisfying
 the conditions (H1), (H2) and (H3) below \cite[Definition~1.4]{JinLu1915} is naturally $\mathbb{R}^{n,0}$-admissible. Then
 \begin{eqnarray*}
 c^{n,0}(H) &=&\sup_{\gamma\in\Gamma_{n,0}}\inf_{x\in \gamma(S^+_{n,0})}\Phi_H(x)\\
 &\le& \sup_{\gamma\in\Gamma_{n,0}}\inf_{x\in \gamma(\widehat{S}^+)}\Phi_H(x)\\
 &\le& \sup_{\gamma\in\widehat{\Gamma}}\inf_{x\in \gamma(\widehat{S}^+)}\Phi_H(x)
 =c_{\rm EH,\tau_0}(H).
 \end{eqnarray*}
It follows that $c^{n,0}(B)\le c_{\rm EH,\tau_0}(B)$ for each $B\subset\mathbb{R}^{2n}$ intersecting with $\mathbb{R}^{n,0}$.

\appendix
\section[Connectedness of the subgroup ${\rm Sp}(2n,k)\subset {\rm Sp}(2n)$ (by Kun Shi)]{
Connectedness of the subgroup ${\rm Sp}(2n,k)\subset {\rm Sp}(2n)$
\textnormal{(by Kun Shi\footnote{School of Mathematical Sciences, Beijing Normal University, Beijing 100875, People's Republic of China,
\texttt{{shikun@mail.bnu.edu.cn}}})}}
\label{sec:app}

Let $e_1,\cdots, e_{2n}$ be the standard symplectic basis in the standard symplectic
Euclidean space $(\mathbb{R}^{2n},\omega_0)$. Then $\omega_0(e_i,e_j)=\omega_0(e_{n+i},e_{n+j})=0$
and $\omega_0(e_{i},e_{n+j})=\delta_{ij}$ for all $1\le i,j\le n$.

\begin{claim}
$A\in {\rm Sp}(2n)$ belongs to ${\rm Sp}(2n,k)$ if and only if
\begin{equation}\label{e:app.1}
A=\left(\begin{array}{cc}
I_{n+k}&\left(\begin{array}{c}
O_{k\times(n-k)}\\
B_{(n-k)\times(n-k)}\\
O_{k\times(n-k)}
\end{array}\right)
\\
O_{(n-k)\times(n+k)}&I_{n-k}
\end{array}\right)
\end{equation}
for some $B_{(n-k)\times(n-k)}=(B_{(n-k)\times(n-k)})^t\in \mathbb{R}^{(n-k)\times(n-k)}$.
Consequently, $tA_0+(1-t)A_1\in {\rm Sp}(2n,k)$ for any $0\le t\le 1$ and $A_i\in {\rm Sp}(2n,k)$,
$i=0,1$. Specially, ${\rm Sp}(2n,k)$ is a connected subgroup of ${\rm Sp}(2n)$.
\end{claim}

The following proof of this claim is presented by Kun Shi.

Let $A\in {\rm Sp}(2n,k)$. Then $Ae_i=e_i$ for $i=1,\cdots,n+k$. For $k<j\le n$, suppose
$Ae_{n+j}=\sum^{2n}_{s=1}a_{s(n+j)}e_s$, where $a_{st}\in\mathbb{R}$.
For $1\le j\le k$ and $k<l\le n$, we may obtain
\begin{eqnarray}\label{e:app.2}
0&=&\omega_0(e_{n+l}, e_{n+j})=\omega_0(Ae_{n+l}, Ae_{n+j})=\omega_0(Ae_{n+l},e_{n+j})\nonumber\\
&=&\sum^{2n}_{s=1}a_{s(n+l)}\omega_0(e_s,e_{n+j})=\sum^{2n}_{s=1}a_{s(n+l)}\delta_{sj}=a_{j(n+l)}
\end{eqnarray}
by a straightforward computation.
Similarly,  for $1\le j\le n$ and $k<l\le n$, we have
\begin{eqnarray*}
-\delta_{jl}&=&\omega_0(e_{n+l}, e_{j})=\omega_0(Ae_{n+l}, Ae_{j})=\omega_0(Ae_{n+l},e_{j})
=\sum^{2n}_{s=1}a_{s(n+l)}\omega_0(e_s,e_{j})\\
&=&\sum^{2n}_{s=n+1}a_{s(n+l)}\omega_0(e_s,e_{j})=
\sum^{n}_{i=1}a_{(n+i)(n+l)}
(-\delta_{ji})=-a_{(n+j)(n+l)}.
\end{eqnarray*}
It follows from this and (\ref{e:app.2}) that $Ae_{n+l}=e_{n+l}+ \sum^n_{j=k+1}a_{j(n+l)}e_j$.
By substituting this and $Ae_{n+s}=e_{n+s}+ \sum^n_{j=k+1}a_{j(n+s)}e_j$ into
$\omega_0(e_{n+l},e_{n+s})=\omega_0(Ae_{n+l}, Ae_{n+s})$ we obtain $a_{j(n+l)}=a_{l(n+j)}$
for all $k<j,l\le n$.

Conversely, suppose that $A\in {\rm Sp}(2n)$ has form (\ref{e:app.1}),
that is, $A$ satisfies: $Ae_i=e_i$ for $i=1,\cdots,n+k$, and
$Ae_{n+l}=e_{n+l}+ \sum^n_{j=k+1}a_{j(n+l)}e_j$ for $k<l\le n$,
 where $a_{j(n+l)}=a_{l(n+j)}\in\mathbb{R}$ for $k<j,l\le n$.
Then it is easy to check that $A\in {\rm Sp}(2n,k)$.

\medskip

\begin{tabular}{l}
 Department of Mathematics, Civil Aviation University of China\\
 Tianjin  300300, The People's Republic of China\\
 E-mail address: rrjin@cauc.edu.cn\\
 \\
  School of Mathematical Sciences, Beijing Normal University\\
 Laboratory of Mathematics and Complex Systems, Ministry of Education\\
 Beijing 100875, The People's Republic of China\\
 E-mail address: gclu@bnu.edu.cn\\
\end{tabular}
\medskip


%
%

\end{document}